\documentclass{amsart}
\usepackage{amssymb,amsmath}
\usepackage{color}
\usepackage{mathtools}
\usepackage{hyperref}
\usepackage{bbm}
\usepackage{ulem}
\usepackage{lscape}
\usepackage{xcolor}
\usepackage{graphbox}
\usepackage{tikz-cd}
\usepackage[a4paper, margin=2.5cm]{geometry} 
\usepackage{appendix}
\usepackage{graphicx}
\usepackage{subcaption}
\usepackage{algorithm}
\usepackage{algpseudocode}
\usepackage{booktabs}
\def\restriction#1#2{\mathchoice
              {\setbox1\hbox{${\displaystyle #1}_{\scriptstyle #2}$}
              \restrictionaux{#1}{#2}}
              {\setbox1\hbox{${\textstyle #1}_{\scriptstyle #2}$}
              \restrictionaux{#1}{#2}}
              {\setbox1\hbox{${\scriptstyle #1}_{\scriptscriptstyle #2}$}
              \restrictionaux{#1}{#2}}
              {\setbox1\hbox{${\scriptscriptstyle #1}_{\scriptscriptstyle #2}$}
              \restrictionaux{#1}{#2}}}
\def\restrictionaux#1#2{{#1\,\smash{\vrule height .8\ht1 depth .85\dp1}}_{\,#2}} 

\newcommand{\zeroes}{\ell} 

\newlength{\halfbls}

\newcommand{\cC}{\mathcal{C}}

\newcommand{\cH}{\mathcal{H}}

\newcommand{\cM}{\mathcal{M}}

\newcommand{\cQ}{\mathcal{Q}}

\newcommand{\cZ}{\mathcal{Z}}

\newcommand{\cST}{\mathcal{ST}}
\newcommand{\cG}{\mathcal{G}}

\newcommand{\N}{{\mathbb N}}
\newcommand{\Z}{{\mathbb Z}}
\newcommand{\Proj}{{\mathbb P}}
\newcommand{\Q}{{\mathbb Q}}
\newcommand{\R}{{\mathbb R}}
\newcommand{\C}{{\mathbb C}}
\renewcommand{\L}{{\mathbb L}}

\newcommand{\Area}{\operatorname{Area}}
\newcommand{\Vol}{\operatorname{Vol}}

\newcommand{\cVol}{\operatorname{\overline{Vol}}}
\newcommand{\Aut}{\operatorname{Aut}}

\newcommand{\card}{\operatorname{card}}

\newcommand{\Cov}{\operatorname{Cov}}
\newcommand{\we}{\operatorname{wt}}

\newcommand{\Covol}{\operatorname{Covol}}

\renewcommand{\epsilon}{\varepsilon}

\renewcommand{\u}{\underline}

\newcommand{\valu}{{\underline{\kappa}}}
\newcommand{\dgu}{{\underline{k}}}
\newcommand{\val}{{\kappa}}
\newcommand{\dg}{{k}}

\newtheorem{Theorem}{Theorem}[section]
\newtheorem*{Theorem*}{Theorem}
\newtheorem{defi}[Theorem]{Definition}
\newtheorem{Corollary}[Theorem]{Corollary}
\newtheorem{prop}[Theorem]{Proposition}
\newtheorem{lemma}[Theorem]{Lemma}
\newtheorem{conj}[Theorem]{Conjecture}

\theoremstyle{definition}
\newtheorem{Convention}[Theorem]{Convention}

\theoremstyle{remark}
\newtheorem{Remark}{Remark}[section]

\everymath{\displaystyle}

\newcommand{\RGc}{\mathcal{RG}}
\newcommand{\vp}{\operatorname{vp}}
\newcommand{\ov}{\overline}
\newcommand{\bl}{\bullet}
\newcommand{\wh}{\circ}

\title[]
{Volumes of odd strata of quadratic differentials}
\author[E.~Duryev]{Eduard Duryev}
\address{
}
\email{edwardduriev@gmail.com}
\author[\'E.~Goujard]{\'Elise Goujard}
\address{
Institut de Math\'ematiques de Bordeaux,
Universit\'e de Bordeaux,
351 cours de la Lib\'eration, 33405 Talence, FRANCE
and Institut Universitaire de France
}
\thanks{Research of the second author is partially supported by
the grants ANR-19-CE40-0003, ANR-23-CE40-0020-02, ANR-23-CE40-0001, ANR-23-CE48-0018, as well as IUF}
\email{elise.goujard@gmail.com}
\author[I.~Yakovlev]{Ivan Yakovlev}
\address{
Max Planck Institute for Mathematics,
Vivatsgasse 7, 53111 Bonn, Germany
}
\email{yakovlev@mpim-bonn.mpg.de}
\date{February 2025}
\begin{document}

\begin{abstract}
    We express the Masur--Veech volumes of ``completed" strata of quadratic differentials with only odd singularities as a sum over stable graphs. This formula generalizes the formula of Delecroix-Goujard-Zograf-Zorich for principal strata. The coefficients of the formula are in our case intersection numbers of psi classes with the Witten-Kontsevich combinatorial classes; they naturally appear in the count of metric ribbon graphs with prescribed odd valencies. The ``completed" strata that we consider are unions of odd strata and some adjacent strata, that contribute to the Masur--Veech volume with explicit weights. We present several conjectures on the large genus asymptotics of these Masur--Veech volumes that could be tackled with this formula.
\end{abstract}

\maketitle
\tableofcontents
\section{Introduction}

Let $g, n$ be non-negative integers with $2g-2+n > 0$. Consider the moduli space $\cM_{g,n}$ of complex curves of genus $g$ with $n$ distinct labeled marked points. The total space $\cQ_{g,n}$ of the cotangent bundle over $\cM_{g,n}$ can be identified with the moduli space of pairs $(X, q)$, where $X \in\cM_{g,n}$ is a smooth complex curve with $n$ labeled marked points and $q$ is a meromorphic quadratic differential on $X$ with at most simple poles at the marked points and no other poles. $\cQ_{g,n}$ is stratified by the subsets of differentials with fixed degrees of zeros, namely 
$\cQ_{g,n}=\bigsqcup_{\dgu=[\tilde\dgu, -1^n]}\cQ(\dgu)$ 
where $\tilde\dgu$ is a partition of $4g-4+n$ and $\dgu=[\tilde\dgu, -1^n]$ is the set of singularity degrees: zeroes of degrees given by $\tilde \dgu$ and $n$ poles of degree $-1$ (see Convention~\ref{convpart} for more details about the notations). 
On each stratum there exist natural coordinates (see section~\ref{sec:counting}) that endow the stratum with a structure of a complex orbifold of dimension $2g-2+\ell(\dgu)=2g-2+n+\ell(\tilde\dgu)$, and such that the Lebesgue measure in these coordinates gives rise to a well defined measure on the stratum. 
A non-zero differential $q$ in $\cQ(\dgu)$ defines a flat metric $|q|$ on the complex curve $X$. This metric has a conical singularity of angle $(k_i+2)\pi$ at each singularity of $q$ of degree $k_i$.
The total area of $(X, q)$
\[\Area(X, q) =\int_X |q|\]
is positive and finite. For any real $a > 0$, consider the following subset in $\cQ(\dgu)$:
\[ \cQ(\dgu)^{\Area\leq a}:= \{(X, q) \in\cQ(\dgu) | \Area(X, q) \leq a\}. \]
The set $\cQ(\dgu)^{\Area\leq a}$ is a ball bundle over $\cM_{g,n}$, in particular, it is non-compact. However, by independent results of H. Masur~\cite{Masur} and W. Veech~\cite{Veech}, the total mass of $\cQ(\dgu)^{\Area\leq a}$ with respect to the chosen measure is finite. Up to some normalization choices (see section \ref{subsec:norm}), this mass is called the Masur--Veech volume of the stratum $\cQ(\dgu)$. The values of the Masur--Veech volumes are particularly relevant to study quantitative dynamics in rational polygonal billiards: they are related by renormalization to the dynamics of the $SL(2, \R)$-action on the moduli spaces $\cQ(\dgu)$ with respect to the Masur--Veech measure~\cite{Zo}.
In this paper, we will consider strata $\cQ(\dgu)$ of quadratic differentials with only odd singularities, that we call \emph{odd strata}, and at least three singularities (the last condition is only technical and is explained in Remark~\ref{rem:ge-3-sing}), and we provide a formula for their Masur--Veech volumes. This work extends previous results of~\cite{DGZZ-vol}, where the Masur--Veech volumes of principal strata (strata $\cQ(1^{4g-4+n}, -1^n)$ of full dimension $6g-6+2n=\dim(\cQ_{g,n})$) are expressed as polynomials in the intersection numbers $\int_{\overline{\cM}_{g',n'}}\psi_1^{d_1}\dots\psi_{n'}^{d_{n'}}$ with some explicit rational coefficients. This result was derived by counting lattice points of strata called square-tiled surfaces (see section~\ref{sec:counting}) via counting integer metrics on trivalent ribbon graphs using~\cite{Kon}. 
In this paper we prove a similar formula, that expresses the ``completed" Masur--Veech volumes of odd strata as polynomials in the intersection numbers $\int_{W_{m_*,n'}}\psi_1^{d_1}\dots \psi_n^{d_{n'}}$, where $W_{m_*,n'}$ are the Witten-Kontsevich combinatorial classes $W_{m^*,n}$ defined in \cite{Kon} and \cite{Witten} (see section~\ref{subsec:combi_classes}), with explicit rational coefficients. The ``completed" volumes are equal to the volumes of the strata plus a linear combination of volumes of adjacent strata, with explicit rational coefficients. The formula is still obtained by counting square-tiled surfaces in the strata, based on counting integer metrics on ribbon graphs with odd valencies. Unlike the case of trivalent ribbon graphs, the counting functions for integer metrics on ribbon graphs with odd valencies are only piecewise quasi-polynomial, and are only known (by results of \cite{Kon}) outside of some hyperplanes called walls. The novelty of our work is to determine these counting functions on the walls (section~\ref{sec:compl}), which is the key to our formula. The discontinuities of the counting functions on the walls come from the counts of metrics on degenerated ribbon graphs. These contributions explain the presence of additional terms (volumes of adjacent strata) in the volume formula. To illustrate all these ideas, we provide a detailed computation for the stratum $\cQ(3, -1^3)$ in section~\ref{subsec:ex}.

\subsection{Context} \label{subsec:context}

The computation of Masur--Veech volumes of strata of quadratic differentials dates back to the work of Eskin--Okounkov~\cite{EO_pillow}, where they prove that the counting function for lattice points in a stratum (in this case ramified covers of the ``pillowcase") is a quasimodular form of level 2 and bounded weight for a certain subgroup of $SL(2,\Z)$, using techniques of representation of the symmetric group and vertex operator algebras. This provides an algorithm to get volumes of strata of small dimension~\cite{Gou}. Independently, closed formula for the volumes of all strata in genus 0 were obtained using the dynamics of the $GL(2, \R)$-action on the strata in~\cite{AEZ_right}.
More recently, the volumes of odd strata were interpreted as Hodge integrals on appropriate compactifications of the projectivized strata in \cite{CMS}, and a similar conjecture for other strata was formulated. In parallel, formulas for the Masur--Veech volumes of principal strata in terms of intersection numbers on $\cM_{g,n}$ were obtained in~\cite{DGZZ-vol} using combinatorial techniques introduced first in~\cite{AEZ}, that we generalize in this paper. Both these last results led to important advances in the computation of volumes of principal strata and their asymptotics in the regime where $g\to\infty$ or $n\to\infty$ \cite{Agg, CMS, ABC, Kaz, YZZ}. The main motivation for the formula presented in this paper is to extend several of these results to all odd strata of quadratic differentials.

Actually all these techniques were first applied to the study of Masur--Veech volumes of strata $\cH(\dgu)$ of Abelian differentials (i.e. holomorphic 1-forms) with zeroes of prescribed degrees $\dgu$. Notice that these strata can be interpreted as the subsets of strata $\cQ(2\dgu)$ formed by quadratic differentials that are globally squares of some Abelian differentials. Concerning the Masur--Veech volumes of these strata, much more is known: expressions in terms of Hodge integrals, explicit recursions, large genus asymptotic expansions~\cite{Sau_min, CMSZ, Sau_large}.

Finally, note that the families of odd strata of quadratic differentials (that we focus on in this paper) play the same role as the minimal strata $\cH(2g-2)$ with respect to all other strata of Abelian differentials: they are minimal in the sense that their natural coordinates involve only absolute periods (integrals over absolute cycles in homology), and no relative periods (relative to the singularities), contrary to other strata, see section~\ref{subsec:norm} for more details. This minimality property suggests that these families can be used to initiate a recursion between the Masur--Veech volumes of all strata as it is conjectured in~\cite{CMS} and as it is done in~\cite{CMSZ} in the case of Abelian differentials.

\subsection{Illustrative example: stratum $\cQ(3, -1^3)$}\label{subsec:ex}

We review all the ideas of this paper informally on the example of the stratum $\cQ(3, -1, -1, -1)=\cQ(3, -1^3)$ of quadratic differentials on a torus with one zero of degree 3 and three labeled poles. At each step we refer to the corresponding sections of the paper for the formal and general proofs and definitions.

To evaluate the volume of $\cQ(3, -1^3)$, we count integer points in the stratum, that is, square-tiled surfaces, see section~\ref{sec:counting}. For this stratum, it is easy to list all possible types of square-tiled surfaces by considering all possible ways the square-tiled surfaces decompose into horizontal cylinders. The corresponding admissible patterns for each type are given in the first column of Table~\ref{tab:contrib} (zeros and poles are denoted by dots and crosses respectively). Each type of decomposition is encoded in a graph (called \emph{stable graph}, see Definition~\ref{def:stable:graph}) such that: the vertices correspond to connected components of the surface obtained by cutting along the waist curves of the cylinders, each vertex is decorated by the angles of conical singularities in the corresponding component divided by $\pi$ (i.e.\ degrees of zeros/poles increased by 2), edges correspond to cylinders, and an edge links two vertices if the corresponding cylinder links the corresponding components in the surface. Half-edges correspond to singularities of angle $\pi$, i.e.\ poles of the corresponding differential. See the second column of Table~\ref{tab:contrib}.

Then, for each type, we evaluate the number of possible choices for the integer parameters $b_i$, $h_i$  and $t_i$ (width, height and twist of the cylinder $i$), as well as the $\ell_j$ (length of the saddle connection $j$), given the following constraints : $\u b\cdot \u h=\sum_i b_i h_i\leq 2N$ (bound on the number of squares), $0\leq t_i < b_i$ (choices of integer twists up to full Dehn twists), and the $b_i$ given as two linear combinations of the $\ell_j$, each equality coming from a boundary of the cylinder $i$. 

In Table~\ref{tab:contrib} for the first type for instance we have $b=2(\ell_1+\ell_2)+\ell_3=2\ell_4+\ell_3$ and there are $\sim \frac{b^2}{8}$ choices of integers tuples $(\ell_1, \ell_2, \ell_3, \ell_4)$ satisfying this constraint as $b$ grows. The number of possible twists is $\sim b$ as $b$ grows, and considering that there are 3! possible ways to label the poles, we get that the number $\cST(\mbox{type }1, 2N)$ of square-tiled surfaces of this type with at most $2N$ squares is
\[\cST(\mbox{type }1, 2N)\sim 3!\sum_{\substack{bh\leq 2N\\b,h\in \N}}\frac{b^3}{8}\sim 3\zeta(4)N^4\mbox{ as }{N\to \infty}\]

In the normalization that we choose here (see section~\ref{subsec:norm}) the contribution of this type of square-tiled surfaces to the Masur--Veech volume of the stratum is \[2 d \cdot
\lim_{N\to+\infty}
\frac{\card(\cST(\mbox{type }1, 2N))}{N^{d}},\]
where $d=4$ is the complex dimension of the stratum, hence here we get $24\zeta(4)$.

\begin{table}
\begin{tabular}{cccc}
\toprule
Patterns & Type & Count & Contribution\\
\midrule
\includegraphics[align=c, scale=0.6]{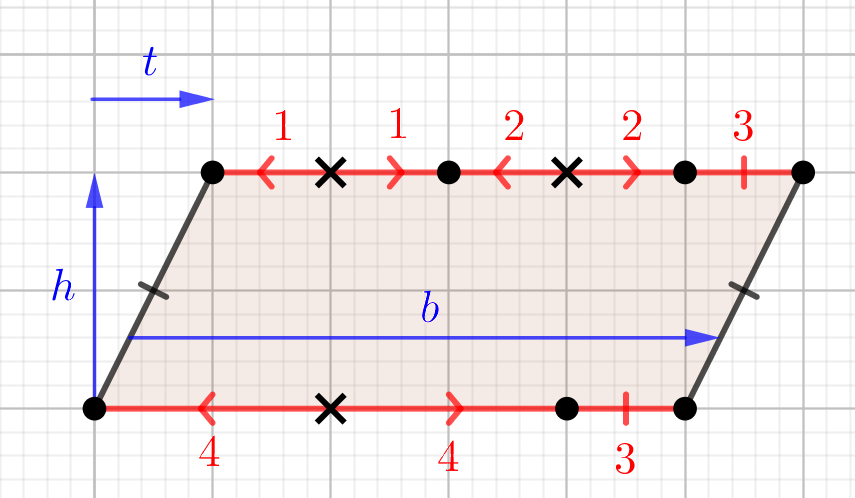} &  {\includegraphics[align=c, scale=0.3]{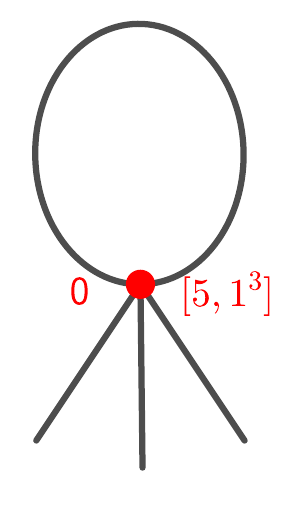} }&$3!\sum_{\substack{bh\leq 2N\\b,h\in \N}}\frac{b^3}{8}$ & $24\zeta(4)=\frac{4\pi^4}{15}$ \\
\midrule
{\includegraphics[align=c, scale=0.6]{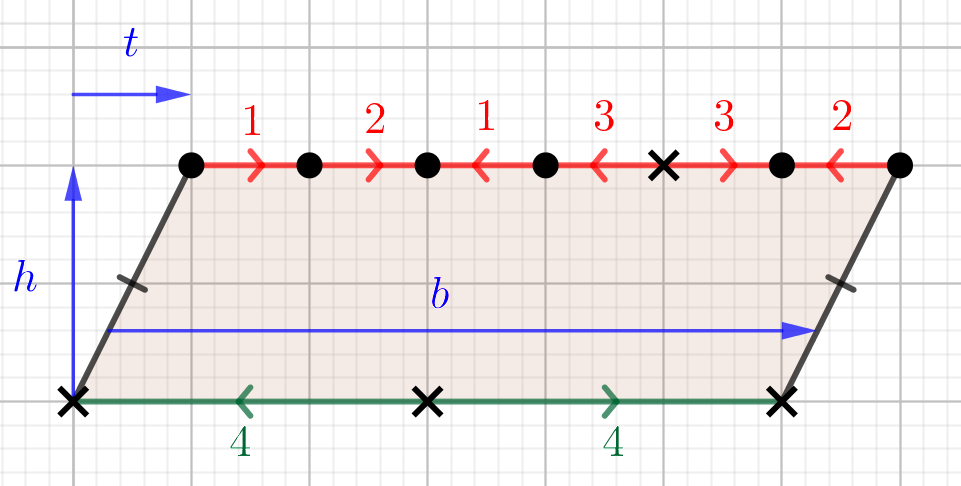}} & {\includegraphics[align=c, scale=0.3]{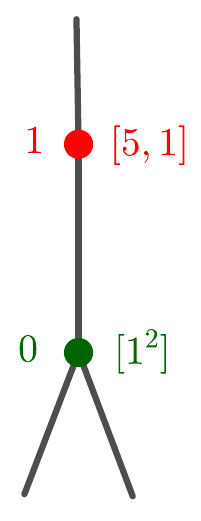}}&
$3\sum_{\substack{bh\leq 2N\\b\in 2\N, h\in\N}}\frac{b^3}{8}$ & $6\zeta(4)=\frac{\pi^4}{15}$\\
\midrule
{\includegraphics[align=c, scale=0.6]{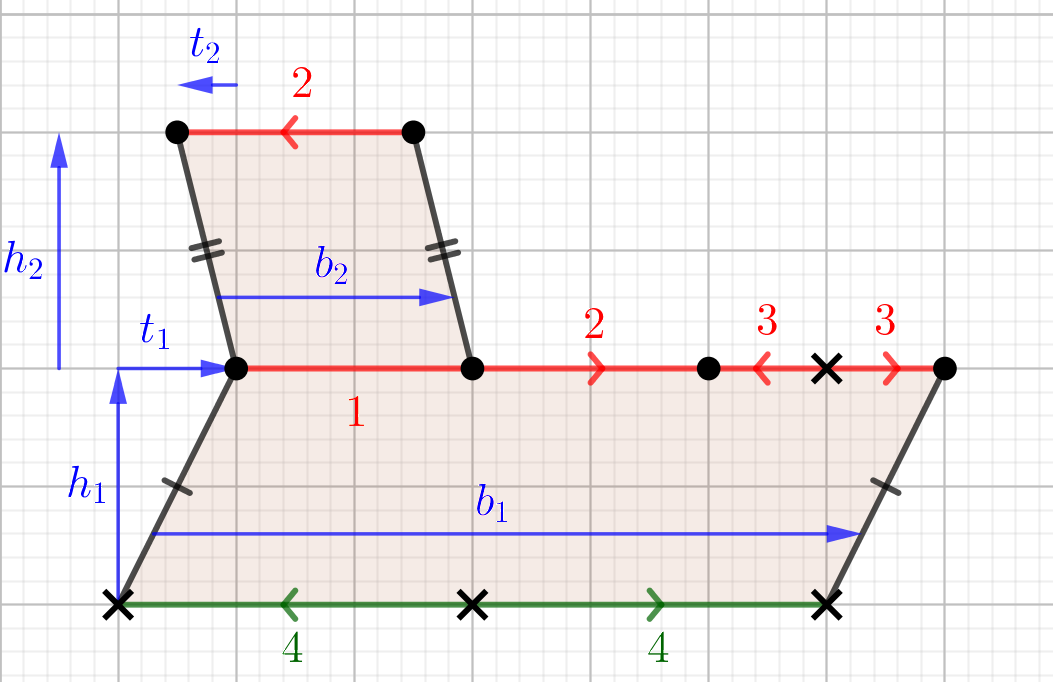}} && $3\sum_{\substack{\u b\cdot \u h\leq 2N\\b_1\in 2\N,\\ b_2, h_1, h_2\in\N\\ 2b_2<b_1}}b_1b_2$ &\\
{ \includegraphics[align=c, scale=0.6]{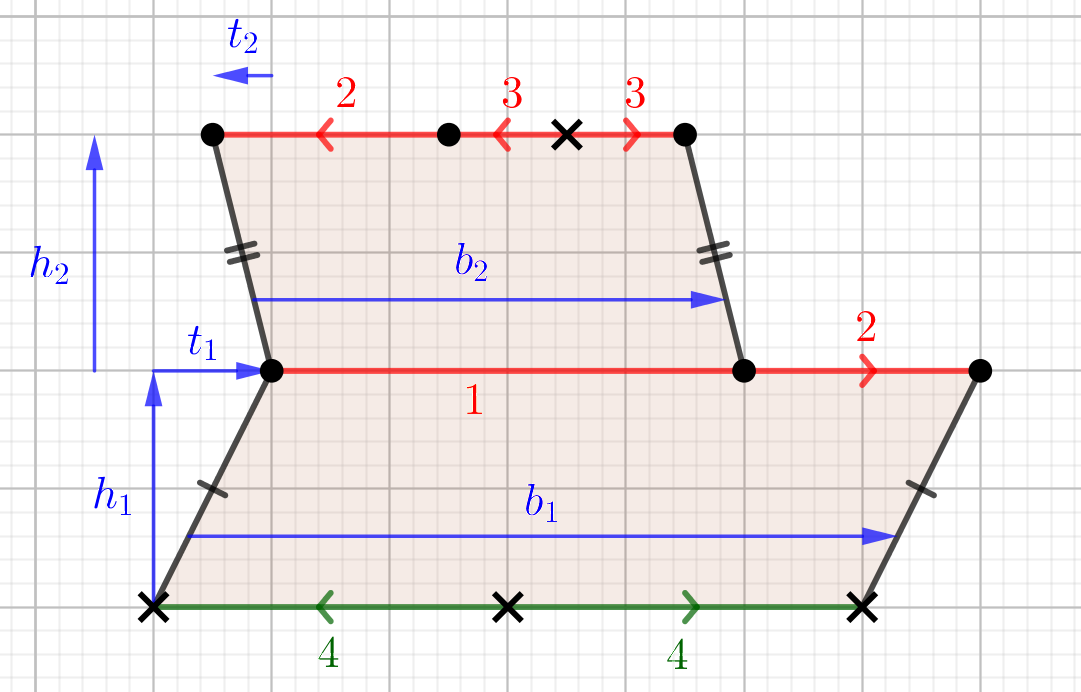}} &  {\includegraphics[align=c, scale=0.3]{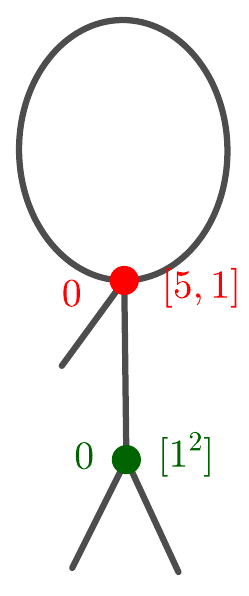}} &
$+ 3\sum_{\substack{\u b\cdot \u h\leq 2N\\b_1\in 2\N,\\ b_2, h_1, h_2\in\N\\b_2<b_1<2b_2}}b_1b_2$
&  $8\zeta(2)^2=\frac{2\pi^4}{9}$\\ 
{\includegraphics[align=c, scale=0.6]{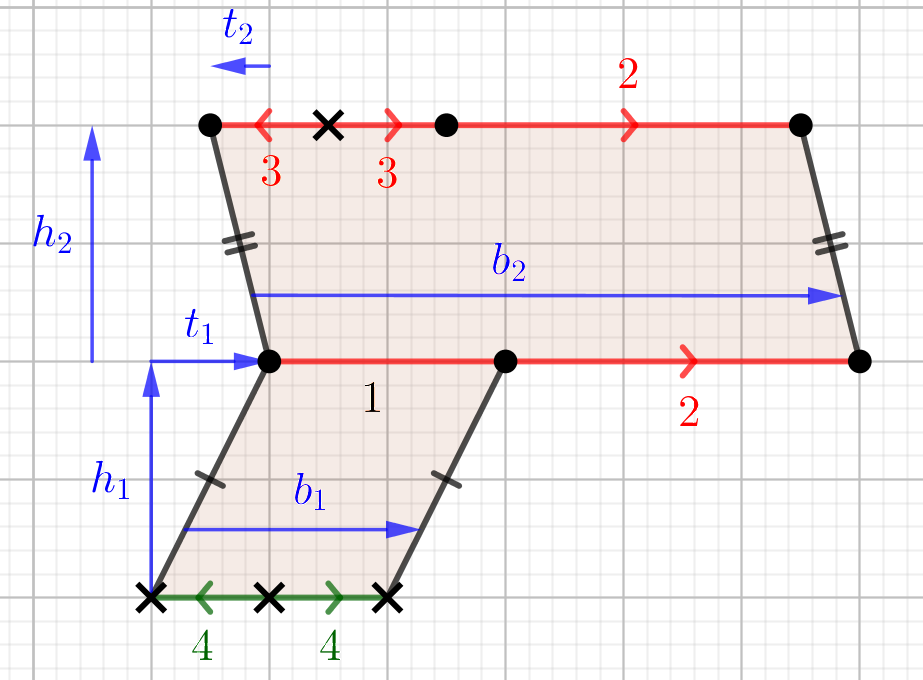}} &&
$+ 3\sum_{\substack{\u b\cdot \u h\leq 2N\\b_1\in 2\N,\\ b_2, h_1, h_2\in\N\\b_1<b_2}}b_1b_2$
&\\
 \bottomrule
\end{tabular}
\caption{Contribution to the Masur--Veech volume of each type of square-tiled surfaces.}\label{tab:contrib} 
\end{table}

Summing all contributions listed in Table~\ref{tab:contrib} we obtain a total volume of $\Vol\cQ(3, -1^3)=\frac{5}{9}\pi^4$ (the detailed computations can be found in \cite[\S 4.2]{Gou}).

For each type, the general formula for the number of square-tiled surfaces is of the form 
\[\operatorname{const} \cdot \sum_{\substack{\u b\cdot \u h\leq 2N\\b_i,h_i\in \N}}b_1\dots b_k \cdot F(b_1, \dots, b_k),\]
where $F$ accounts for the number of integer solutions $\ell_j$ of the system of equations $b_i=\sum\ell_j$, which express the perimeter of each cylinder in terms of the lengths of saddle connections. If $F$ is a polynomial, Lemma 3.7 in \cite{AEZ} (see also Lemma~\ref{lm:evaluation:for:monomial} in this paper) allows to compute the large $N$ asymptotics of these sums, hence also the contribution to the volume of strata of each type of square-tiled surfaces. Note that (although it is the case in the example) in general $F$ is not a polynomial (it is only a piecewise quasi-polynomial, see section~\ref{sec:metric}), even when restricting to the highest degree terms, making the evaluation of large $N$ asymptotics very subtle. Currently, it is only conjectured that each of these sums divided by $N^{\dim{stratum}}$ converges to a rational combination of multiple zeta values. 

Now observe that the number of solutions of the system of equations $b_i=\sum\ell_j$ is the number of integer metrics on the ribbon graphs (see section~\ref{sec:formula} for the definition) formed by the horizontal saddle connections (which also form the boundaries of the cylinders), with prescribed perimeters of the boundary components. The perimeter of any boundary component of any ribbon graph should coincide with the width of the cylinder that is glued to it. Furthermore, the vertices of these graphs correspond to the singularities of the differential, and the vertex valencies are given by $\val_i=\dg_i+2$ where $\dg_i$ is the degree of the corresponding singularity. For instance, for the first type in the example above, the (unique) ribbon graph is depicted in Figure~\ref{fig:rib}. It has genus 0, two boundary components (faces), one vertex of valency 5 and three of valency 1. The number of integer choices for the $\ell_i$ is the number of integer metrics compatible with the constraint that both boundary components have perimeter $b$.

\begin{figure}[h]
\begin{center}\begin{tabular}{cc}
\includegraphics[align=c, scale=1]{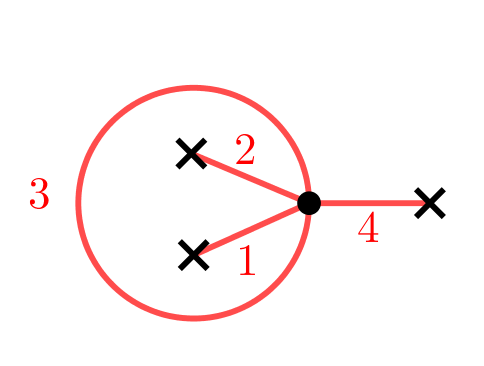} & $\begin{cases} b=2\ell_4+\ell_3\\ b=2(\ell_1+\ell_2)+\ell_3\end{cases}$
\end{tabular}\end{center}
\caption{Ribbon graph and count of integer metrics.}\label{fig:rib}
\end{figure}

In \cite{Kon}, Kontsevich relates the count of integer metrics on ribbon graphs to certain intersection numbers on $\overline{\cM}_{g,n}$. His better known result about the Witten conjecture concerns trivalent ribbon graphs (and this is the result that is used in \cite{DGZZ-vol}), but he also has a version for ribbon graphs with odd valencies. The result is formally presented in section~\ref{sec:metric} but informally states that the count of integer metrics on all ribbon graphs sharing the same combinatorics (genus $g$, valencies of vertices $\valu$, number of boundary components $n$) is an explicit polynomial $N_{g,n}^{\valu}$ in the perimeters of the boundary components, outside of a finite union of hyperplanes and up to lower degree terms. Furthermore, the coefficients of $N_{g,n}^{\valu}$ are certain intersection numbers over $\overline{\cM}_{g,n}$ or, in the case that we consider, over the Witten-Kontsevich combinatorial cycles (see section~\ref{subsec:combi_classes}).

A natural idea is then, for each type of square-tiled surfaces, to replace the computation of the sums

\[\operatorname{const} \cdot \sum_{\substack{\u b\cdot \u h\leq 2N\\b_i,h_i\in \N}}b_1\dots b_k\cdot  F(b_1, \dots, b_k)\quad \mbox{ by }\quad  \operatorname{const} \cdot \sum_{\substack{\u b\cdot \u h\leq 2N\\b_i,h_i\in \N}}b_1\dots b_k\cdot \prod_v N_{g_v,n_v}^{\valu_v}(b_1, \dots, b_{n_v}),\]
where the last product is over the vertices of the stable graph. Performing this replacement we obtain, after taking the limit as $N\to \infty$, a new quantity that we call the \emph{completed volume} $\cVol(\cQ(\dgu))$. The general formula for the completed volumes is given in Definition~\ref{def:vol}. 

To complete this example we compute in Table~\ref{tab:volc} the completed volume $\cVol(\cQ(3, -1^3))$ using the notations and the formula of Definition~\ref{def:vol} (so here $\dgu=[3, -1^3]$ and $\valu=[5,1^3]$). The stratum has complex dimension $d=4$. The constants are 
$c_d=128$, $c_{\valu}=1$ and $c_{\valu_v}=1$. We use the following Kontsevich polynomials (see also Appendix~\ref{app:tables} for the values and Appendix~\ref{app:DFIZ} that explains how these values are obtained from \cite{DFIZ}): 
\begin{equation*}
\begin{aligned}
N_{0,2}^{[5, 1^3]}(b_1,b_2)&=\frac{3}{4}(b_1^2+b_2^2)\\
N_{1,1}^{[5, 1]}(b_1)&=\frac{b_1^2}{8}
\end{aligned}
\hspace{2cm}
\begin{aligned}
N_{0,1}^{[1^2]}(b_1)&=1\\
N_{0,3}^{[5, 1]}(b_1,b_2,b_3)&=3
\end{aligned}
\end{equation*}
The second column in Table~\ref{tab:volc} accounts for the number of non-equivalent ways to label the legs of the stable graph (which is part of the data).
\begin{table}[h]
\[\begin{array}{cccccc}
\toprule
\textrm{Stable graph }  \Gamma & \textrm{Mult.} & c_{\Gamma} & P_{\Gamma} &  \cZ(P_{\Gamma}(\u b))& \textrm{Contribution}\\
\midrule
\includegraphics[align=c, scale=1]{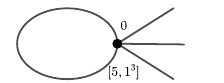}
 & 1&\frac{1}{2}&b_1N_{0,2}^{[5, 1^3]}(b_1, b_1)=\frac{3}{2}b_1^3&\frac{3}{8}\zeta(4)& 24\zeta(4)=\frac{4\pi^4}{15}\\
\includegraphics[align=c, scale=1]{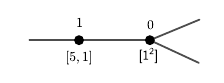} & 3&\frac{1}{2}&b_1N_{1,1}^{[5, 1]}(b_1)N_{0,1}^{[1^2]}(b_1)=\frac{b_1^3}{8}&\frac{1}{32}\zeta(4)&6\zeta(4)=\frac{\pi^4}{15}\\
\includegraphics[align=c, scale=0.3]{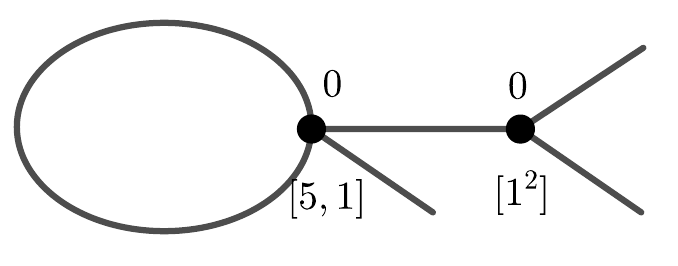} & 3&\frac{1}{4}&b_1b_2N_{0,3}^{[5,1]}(b_1,b_2,b_2)N_{0,1}^{[1^2]}(b_1)=3b_1b_2&\frac{1}{8}\zeta(2)^2& 12\zeta(2)^2=\frac{\pi^4}{3}\\
\bottomrule
\end{array}\]
\caption{Contribution to the completed volume of each type of square-tiled surface.}\label{tab:volc}
\end{table}

Summing all contributions in Table~\ref{tab:volc} we obtain $\cVol(\cQ(3, -1^3))=\frac{2}{3}\pi^4$.

Note that (see section~\ref{subsec:product} for the normalization of the volume of products of strata) \begin{align*}
	\cVol(\cQ(3, -1^3))&=\Vol(\cQ(3,-1^3))+ \frac{1}{2} \cdot \Vol\big(\cQ(-1^4) \times \cH(0)\big)\\
	&=\frac{5}{9}\pi^4+\frac{1}{2} \cdot \frac{1! \cdot 2^2 \cdot 1!}{2 \cdot 3!}\left(2\pi^2 \cdot \frac{\pi^2}{3}\right)\\
	&=\frac{5}{9}\pi^4+\frac{1}{9}\pi^4=\frac{2}{3}\pi^4.\end{align*}

To explain why the completion term appears, one should look closely at the type of square-tiled surfaces where the two computations do not match. It is the third type, where the polynomial $N_{0,3}^{[5,1]}$ appears: we obtain a contribution that is $3/2$ times bigger than the one expected.

So we will now compare $N_{0,3}^{[5,1]}$ to the actual count of integer metrics on ribbon graphs of genus $0$ with three boundary components of perimeters $b_1, b_2, b_3$, with one vertex of valency 5 and one vertex of valency 1, which is done in Table~\ref{tab:detail}.
\begin{table}[h]
\[
\begin{array}{ccc}
\toprule
    \mbox{Ribbon Graph} & \mbox{Number of integer metrics} & \mbox{Sector} \\
    \midrule
    \includegraphics[scale=0.5, align=c]{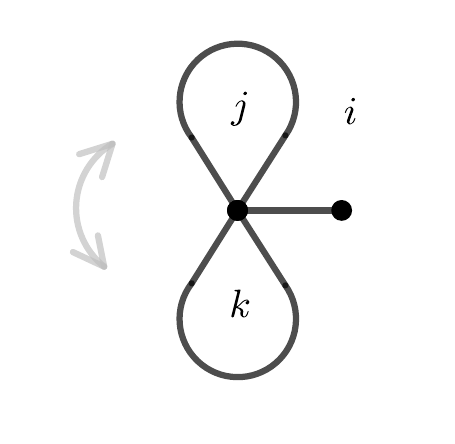} & 2 \mbox{ for } b_i>b_j+b_k & \includegraphics[scale=0.3, align=c]{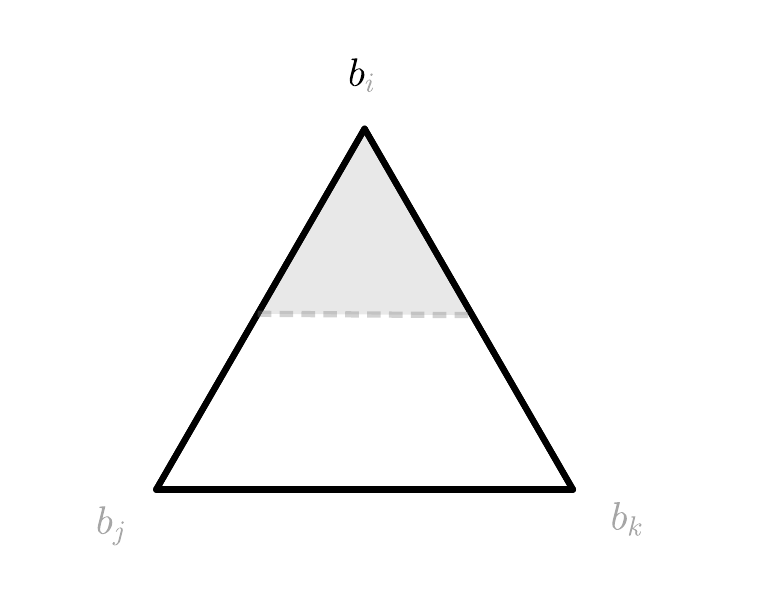}\\
    \includegraphics[scale=0.5, align=c]{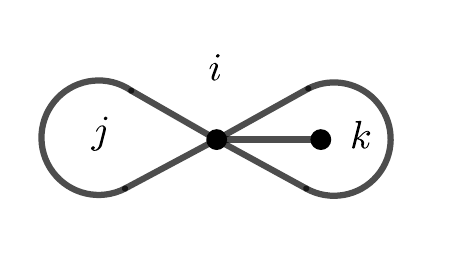} & 1 \mbox{ for } b_j+b_k>b_i>b_j & \includegraphics[scale=0.3, align=c]{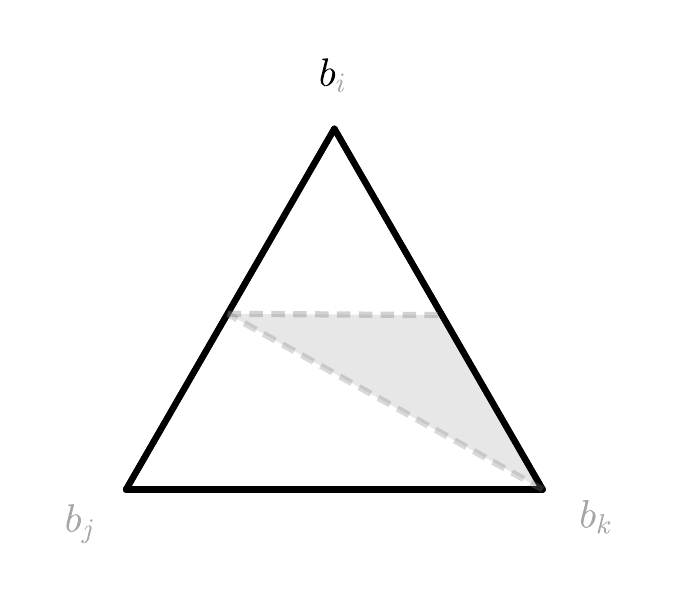}\\
    \bottomrule
\end{array}
\]
\caption{Number of integer metrics on genus 0 ribbon graphs with one vertex of valency 5 and one of valency 1.}\label{tab:detail}
\end{table}

Summing the contributions of Table~\ref{tab:detail}, we obtain a piecewise polynomial equal to 3 outside the codimension 1 walls $\{b_i=b_j\}$ and $\{b_i=b_j+b_k\}$, where it equals 2 and 1 respectively (see Figure~\ref{fig:param}). To ``complete" this piecewise polynomial into a global polynomial $N_{0,3}^{[5,1]}(b_1,b_2,b_3)=3$, we add the count of ``degenerated" metric ribbon graphs as depicted in Figure~\ref{fig:param}.

\begin{figure}[h]
\begin{center}
\includegraphics[scale=0.7]{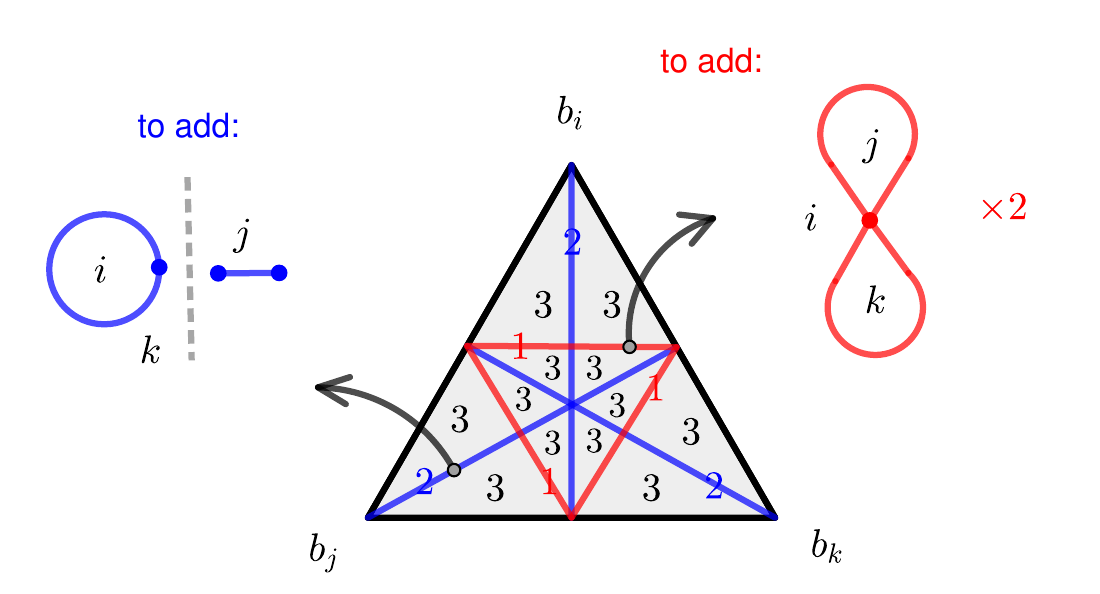}
\end{center}
\caption{Values of the counting function for integer metrics on the parameter space.}\label{fig:param}
\end{figure}
There is also a contribution of the further degenerated ribbon graphs on codimension 2 walls (intersections of the previous walls) which we will not consider here.

Since in the computation of the completed volume we evaluate the polynomial $N_{0,3}^{[5,1]}$ on the wall $b_2=b_3$, we then take into account the blue degenerated ribbon graphs. By gluing cylinders to these degenerated ribbon graphs following the global graph $\Gamma$, we obtain (non-connected) square-tiled surfaces that belong to the stratum $\cQ(-1^4) \times \cH(0)$, explaining the extra contribution in the total volume.

In the sequel we will be particularly interested in the degenerations occurring on the walls of the form $\{b_i=b_j\}$. In the example above, note that the blue ribbon graph arises as a degeneration of the second type of ribbon graphs when the length of the separating loop surrounding the univalent vertex goes to 0. In Theorem~\ref{thm:counting-funcs-on-walls} we give a general formula that expresses the additional contribution on the walls of degenerated ribbon graphs to the Kontsevich polynomials $N_{g,n}^{\valu}$. This section is the core and the novelty of this paper, and the principal ingredient of the main result, Theorem \ref{thm:coeffs}, that expresses the difference between the ``completed" volume and the usual volume as a linear combination of volumes of boundary strata.

\subsection{Outline of the paper}

In section~\ref{sec:formula} we give the main result of the paper (Theorem~\ref{thm:coeffs}) after giving all necessary definitions and in particular introducing the Kontsevich polynomials. Section~\ref{sec:metric} reviews the relation between these polynomials and the counting of integer metrics on ribbon graphs, precising all piecewise polynomiality properties of the counting functions. In section~\ref{sec:counting} we prove a formula (Proposition~\ref{prop:volsq}) for the Masur--Veech volumes of strata in terms of counting functions for metrics on ribbon graphs, via the enumeration of square-tiled surfaces. Sections~\ref{sec:compl} and~\ref{sec:count-degens} form the technical core and the novelty of the paper: they are devoted to the proof of Theorem~\ref{thm:counting-funcs-on-walls} that expresses the difference between the Kontsevich polynomials and the counting functions for metric ribbon graphs on their walls of discontinuity. The proofs in these sections are part of the PhD thesis of the third author: \cite{Yakovlev-thesis}, Chapter 7. Section~\ref{sec:loc:glob} finishes the proof of Theorem~\ref{thm:coeffs} by combining the global formula of Proposition~\ref{prop:volsq} with the local contributions obtained in Theorem~\ref{thm:counting-funcs-on-walls}. In section~\ref{sec:appl-conj} we present several conjectures on the large genus asymptotics of Masur--Veech volumes and the large genus asymptotic geometry of random square-tiled surfaces, which could be approached with our main result.

Several appendices are added to the paper. Appendix~\ref{appendix:metrics} regroups some technical proofs relevant to section~\ref{sec:metric}. Appendix~\ref{app:DFIZ} recaps briefly for completeness the algorithm developed in \cite{DFIZ, Bini} to compute the intersection numbers appearing in Kontsevich polynomials. Appendix~\ref{app:hurwitz} explains the consequences of our results in terms of shifted symmetric functions that appear in the count of ramified covers in Hurwitz theory. Finally, Appendix~\ref{app:tables} lists all volumes computed thanks to our formula as well as the corresponding Kontsevich polynomials.

\subsection{Aknowledgments} We would like to thank Alex Eskin for formulating the problem and attracting our attention to the results of~\cite[Section 3.3]{Kon}. We are extremely grateful to Adrien Sauvaget for finding the formula of Theorem~\ref{thm:coeffs} from our list of examples. We warmly thank Martin Möller and Philip Engel for numerous discussions about completed cycles from the perspective of shifted symmetric functions. We would like to thank Gabriele Mondello, Dimitri Zvonkine, Vincent Delecroix, Anton Zorich, Amol Aggarwal, Gaetan Borot and Gilberto Bini for related discussions.

We thank MSRI and MFO for their hospitality during the preparation of this paper.

\section{Volume formula}\label{sec:formula}

A \emph{ribbon graph} (also called a combinatorial map) is a graph with possible loops and multiple edges endowed with a cyclic ordering of the half-edges issued from each vertex. The cyclic orderings define in particular the \emph{boundary components} (or \emph{faces}) of the ribbon graph. A \emph{metric ribbon graph} is a ribbon graph endowed with a metric, i.e.\ an assignment of a positive number to each edge. Ribbon graphs arise naturally as graphs embedded in surfaces, and metric ribbon graphs as graphs embedded in surfaces equipped with a metric. The \emph{genus} of a ribbon graph is the genus of a surface in which it is cellularly embedded, meaning that the complementary regions (the faces) are homeomorphic to disks. 

As an interesting example, one can consider the union of all non-closed trajectories and  zeros of a Jenkins-Strebel differential, that is, a quadratic differential for which the set of non-closed trajectories has measure zero. A classical geometric construction gives a correspondence between such differentials and metric ribbon graphs. This correspondence  is an important ingredient of the celebrated results of Kontsevich about the Witten's conjecture \cite{Kon}. It is also a key idea of our count of square-tiled surfaces, see section~\ref{sec:counting}.

For a comprehensive introduction to ribbon graphs, including the detailed definitions and the correspondence with Jenkins-Strebel differentials, we refer to the book~\cite{LandoZvonkin}.

\subsection{Combinatorial moduli space and Kontsevich polynomial}\label{subsec:combi_classes}

Let $\cM_{g,n}^{comb}$ denote the combinatorial moduli space, that is the set of equivalence classes of connected metric ribbon graphs with vertices of valency greater than or equal to 3, of genus $g$ with $n$ labeled faces. This combinatorial moduli space is homeomorphic as an orbifold to $\cM_{g,n}\times (\R_+)^n$, where $\cM_{g,n}$ is the usual moduli space of complex curves of genus $g$ with $n$ punctures, provided that it is non-empty, i.e. $n\geq 1 $ and $2g-2+n>0$. There are two different ways to define such a homeomorphism. The first one uses hyperbolic geometry and is due to Penner~\cite{Penner} and Bowditch-Epstein~\cite{BE}. The second one uses meromorphic quadratic differentials: it combines the correspondence with Jenkins-Strebel differentials with the existence and uniqueness results of Strebel~\cite{Strebel}, and is due to Harer, Mumford and Thurston~\cite{Harer} (see also \cite{Strebel_quad}, \cite{HM}, \cite[Thm 2.2]{Kon}). For $m_*=(m_0, m_1, \dots)$ an infinite sequence of non-negative integers, almost all zero, we denote by $\cM_{m_*,n}$ the set of equivalence classes of metric ribbon graphs with $m_i$ vertices of valency $2i+1$ and $n$ faces. If $m_0=0$ it is a cell in the moduli space $\cM_{g,n}^{comb}$, where $g$ is computed from the Euler's relation for the corresponding ribbon graphs:
  \begin{equation} \label{eq:euler-rel}
  	2g-2+n=\frac{1}{2}\sum_i m_i(2i-1).
  \end{equation}
The cell has real codimension $2M$, where 
\begin{equation}
	\label{eq:notM} M=\sum_i m_i(i-1).
\end{equation}

If $m_0\neq 0$, $\cM_{m_*,n}$ still maps to $\cM_{g,n}\times (\R_+)^n$ by using the correspondence with Jenkins-Strebel differentials, and it also has real dimension $\frac{1}{2}\sum_{i} m_i(2i+1)=2 \cdot (\dim_{\C} \cM_{g,n}^{comb}-M)$ (the dimension is the number of edges of the corresponding ribbon graphs). 

The space $\cM_{g,n}^{comb}$ admits a natural compactification $\overline {\cM}_{g,n}^{comb}$ by considering stable ribbon graphs, see \cite[Def 3.9]{Zvon}. The previous homeomorphism of orbifolds extends to a homeomorphism  $\overline {\cM}_{g,n}^{comb}\to K\overline{\cM}_{g,n}\times(\R_+)^n$, where $K\overline{\cM}_{g,n}$ is a quotient of the Deligne-Mumford compactification $\overline{\cM}_{g,n}$ obtained by contracting all the components of stable curves that do not contain a marked point (see \cite{Loo} and \cite{Zvon}) for the detailed proofs). 

Each component of $\cM_{m_*,n}$ can be endowed with a natural orientation \cite{Kon}. When $m_0=0$, $\cM_{m_*, n}$ can be seen as a cycle with non-compact support in $\overline\cM^{comb}_{g,n}$ and hence defines a class $W_{m_*,n}\in H_{6g-6+2n-2M}(K\overline\cM_{g,n}; \Q)$ using the homeomorphism. 

This allows to define the following intersection numbers 
\[\langle  \tau_{\u d}\rangle_{m*}=\int_{\cM_{m_*,n}}\prod_i \psi_i^{d_i} \times [\R^+_n],\] 
where $\psi_i$ is the first Chern class of the $i$-th cotangent line bundle on $\overline \cM_{g,n}$, $[\R_+^n]$ denotes the standard fundamental class with compact support of $\R_+^n$, and if $m_0\neq 0$ there is an implicit pullback to $\cM_{m_*,n}$. When $m_0=0$ we can also write 
\[\langle  \tau_{\u d}\rangle_{m*}=\int_{W_{m_*,n}}\prod_i \psi_i^{d_i}.\]

We also use the following notation for the usual intersection numbers:
\[\langle  \tau_{\u d}\rangle=\int_{\overline \cM_{g,n}}\prod_i \psi_i^{d_i}.\]

The intersection numbers $\langle  \tau_{\u d}\rangle_{m*}$ can be computed in terms of the $\langle  \tau_{\u d}\rangle$ by the results of \cite{DFIZ}. For completeness we present the corresponding algorithm in Appendix~\ref{app:DFIZ}. The cycles $W_{m_*,n}$ and their relation to the Mumford $\kappa$ classes are studied in \cite{AC, Mondello, Igusa}.

\begin{Convention}\label{convpart} 
	Throughout this paper, we adopt the following terminology and notations:
	\begin{itemize}
		\item a \emph{composition} of a positive integer $n$ is an ordered sequence $(n_1,\ldots,n_k)$ of positive integers summing up to $n$;
		\item a \emph{partition} is an equivalence class of compositions, where two compositions are equivalent if one is obtained from another by permuting its terms; partitions are written in square brackets: $[n_1, \ldots, n_k]$; when an element of a partition has multiplicity, we use the multiplicative notation: e.g. $[5,1,1,1]=[5,1^3]$;
		\item compositions/partitions are denoted by underlined letters: $\valu, \u b$, etc.;
		\item when a composition is substituted where a partition is expected, just substitute the corresponding partition.
	\end{itemize}
	For a composition/partition $\valu$:
	\begin{itemize}
		\item we define its \emph{weight} as the sum of its terms, denoted by $|\valu|$;
		\item we define its \emph{length} as the number of its terms, denoted by $\ell(\valu)$;
		\item we denote by $\mu_i(\valu)$ the number of its terms equal to $i$;
		\item we define $|\Aut(\valu)| = \prod_{i=1}^\infty \mu_i(\valu)!$ and $c_{\valu} = \prod_{i=2}^\infty \mu_i(\valu)!$;
		\item for an integer $n$ we denote by $\valu+n$ the composition/partition obtained by adding $n$ to each term of $\valu$, and by $n\valu$ the one obtained by multiplying each term of $\valu$ by $n$.
	\end{itemize}
\end{Convention}

\begin{Convention}\label{convmk}
In the rest of the paper, $\valu$ denotes an odd composition/partition, $\dgu = \valu-2$, and $m_* = (m_0,m_1,\ldots)$ is such that $m_i = \mu_{2i+1}(\valu) = \mu_{2i-1}(\dgu)$. We still call $\dgu$ a composition/partition and use the Convention~\ref{convpart} for it, even though it may have some parts equal to $-1$.
\end{Convention}

Note that using Convention~\ref{convmk}, the Euler's relation \eqref{eq:euler-rel} can be rewritten as $2g-2+n = \frac{1}{2}|\valu| - \ell(\valu)$.

\begin{defi}
For each $g \ge 0$, $n \ge 1$, $\valu$ an odd partition, satisfying $2g-2+n = \frac{1}{2}|\valu| - \ell(\valu) > 0$, we define the unlabeled Kontsevich polynomials as
$$N_{g,n}^{\valu, unlab}(b_1, \dots, b_n)=\frac{1}{2^{5g-6+2n{-2M}}}\sum_{|\underline d|=3g-3+n{-M}}\frac{1}{\underline d!}\langle\tau_{\underline d}\rangle_{m_*}\underline b^{2\underline d},$$
and the labeled ones as 
\[N_{g,n}^{\valu}(b_1, \dots, b_n)= |\Aut(\valu)| \cdot N_{g,n}^{\valu, unlab}(b_1, \dots, b_n).\]
Here $\u b = (b_1,\ldots,b_n)$ and $\u b^{2\u d}$ denotes $b_1^{2d_1} \cdots b_n^{2d_n}$.
\label{def:Kont_poly}
\end{defi}
Note that $N^{\valu}_{g,n}$ is of degree $2g-2+\ell(\valu)$. Indeed, by definition, the degree is $6g-6+2n-2M = 6g-6+2n-|\valu|+3\ell(\valu)$, and it is enough to apply the Euler's relation.

\subsection{Decorated stable graphs}

We adapt the definition of stable graphs, that usually  encode the boundary classes of
the Deligne-Mumford compactification $\overline{\cM}_{g,n}$, to our case: here the stable graphs encode certain boundary classes of the compactification $\overline{\cQ(\dgu)}$.

\begin{defi}
\label{def:stable:graph}
Consider a 6-tuple
$\Gamma = (V, H, \iota, \alpha, \boldsymbol{\valu}, L)$, where
\begin{itemize}
\item $V$ is a finite set of \textit{vertices}.
\item $H$ is a finite set of \textit{half-edges}.
\item $\iota: H \to H$ is an involution. The fixed points
of $\iota$ are called the \textit{legs} and the 2-cycles of
$\iota$ are called the \textit{edges} of $\Gamma$.
\item $\alpha: H \to V$ is a map that attaches a half-edge
to a vertex. 
The number of legs incident to $v$ is denoted by $l_v$ and the number of half-edges incident to $v$ which are part of an edge is denoted by $n_v$, so that $|\alpha^{-1}(v)|=n_v+l_v$.
\item The graph is connected: for each pair of vertices
$(u,v)$ there exists a sequence of
half-edges $(h_1, h'_1, h_2, h'_2, \ldots, h_k, h'_k)$
such that $\iota(h_i) = h'_i$, $u = \alpha(h_1)$,
$v = \alpha(h'_k)$ and $\alpha(h'_i) = \alpha(h_{i+1})$.
\item $\boldsymbol{\valu} = \{\valu_v\}_{v \in V}$ is a set
of \emph{partitions}, one at each vertex,
called the \textit{decoration}. By convention $\valu_v$ has exactly $l_v$ parts equal to 1. 
\item $L$ is a bijection from the set of legs to
$\{1,\ldots,l\}$, where $l=\sum_{v\in V}l_v$.
\end{itemize}
We denote by $\u\kappa$ the partition obtained by concatenating all $\u\kappa_v$ (called \emph{total decoration}).

Such a 6-tuple $\Gamma$ is called a
\emph{decorated stable graph} for a stratum $\cQ(\dgu)$ (where $\dgu$ corresponds to $\valu$ by  Convention~\ref{convmk}) if the decoration $\boldsymbol{\valu}$ satisfies
the following conditions:
\begin{itemize}
\item
For each vertex $v$ of the graph, there exists a non-negative integer $g_v$ such that \begin{equation}\label{eq:cond_vertex}
|\valu_v|-2\ell(\valu_v)=4g_v-4+2n_v
\end{equation}
   \item At each vertex $v$, $2g_v-2+n_v>0$  (\textit{stability condition}).

\end{itemize}
We define the \textit{genus} of $\Gamma$ by the following formula:
\[g(\Gamma) = h_1(\Gamma) + \sum_{v \in V} g_v,\]
where $h_1(\Gamma)$ is the first Betti number of the graph.
\end{defi}

To a decorated stable graph $\Gamma$ we associate an underlying graph
whose vertex set is $V$ and each 2-cycle $(h,h')$ of
$\iota$ gives an edge attached to $\alpha(h)$ and
$\alpha(h')$. We denote the set of these edges by $E =
E(\Gamma)$. Such a graph can have multiple edges and loops.
The additional information carried by a stable graph is the
 decoration $\boldsymbol{\u\kappa}$ and the $l$ labeled legs.

Two decorated stable graphs $\Gamma =
(V,H,\iota, \alpha,\boldsymbol{\u\kappa}, L)$ and
$\Gamma'=(V',H',\iota', \alpha',\boldsymbol{\u\kappa'}, L')$ are
isomorphic if there exists two bijections $\phi: V \to V'$
and $\psi: H \to H'$ that preserve edges, legs and
decoration, that is
\[
\psi\circ\iota   = \iota' \circ \psi\,,
\qquad
L'(\psi(h)) = L(h)\,,
\qquad
\u\kappa'_{\phi(v)} = \u\kappa_v\,.
\]
Note that automorphisms of decorated stable graphs are allowed to
interchange edges and vertices respecting the decoration
but not the legs which are numbered by $L$.

We denote by $\cG_{g,l}^{{\u\kappa}}$ the set of isomorphism classes of decorated stable graphs with given genus $g$, number of legs $l$ and total decoration ${\u\kappa}$.

\begin{Remark} \label{rmk:odd-implies-stable}
	Note that the stability condition in the definition of a stable graph follows from condition~\eqref{eq:cond_vertex} if $\valu$ (or $\dgu$) is an odd partition. Indeed, in this case for each $v \in V(\Gamma)$ we have $|\valu_v| \ge \mu_1(\valu_v) + 3(\ell(\valu_v)-\mu_1(\valu_v))$ and so $2(2g_v-2+n_v) = |\valu_v|-2\ell(\valu_v)+2\mu_1(\valu_v) \ge \ell(\valu_v) > 0$.
\end{Remark}

\begin{figure}
    \centering
    \includegraphics[scale=0.4]{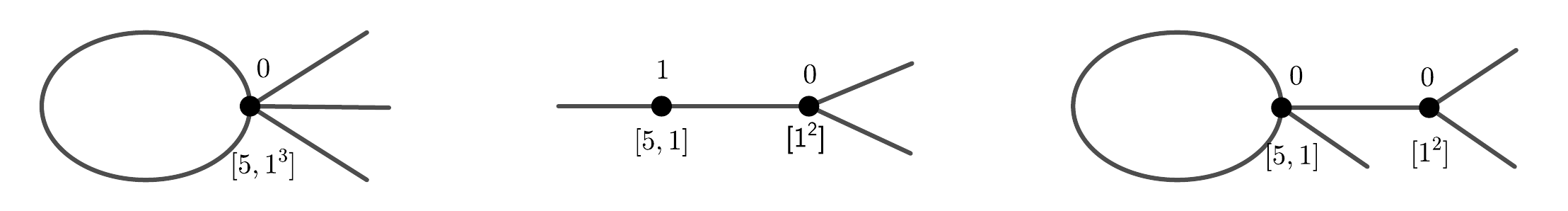}
    \caption{List of all decorated stable graphs in $\cG_{1,3}^{[5,1^3]}$. The genus $g_v$ is indicated above each vertex $v$.}
    \label{fig:stable_graphs}
\end{figure}

\subsection{Formula}\label{subsec:form} 
 
Let $\dgu= [k_1, k_2, \ldots, k_r]$ 
be an odd partition of $4g-4$, where  $k_i \ge -1$. Let $\cQ(\dgu)$ denote the stratum of meromorphic quadratic differentials with zeroes and simple poles of orders $k_i$.
Denote its dimension by $d = \dim_\C \cQ(\dgu) = 2g-2+\ell(\dgu)$. We let $\valu$ be the partition associated to $\dgu$ by Convention~\ref{convmk} ($\valu=\dgu+2$) and let $l = \mu_{-1}(\dgu) = \mu_1(\valu)$.

\begin{defi}For a decorated stable graph $\Gamma$ in $\cG_{g,l}^{\valu}$, fixing a labeling of the edges, we define the following polynomial in the variables $\u b=(b_1, \dots, b_{|E(\Gamma)|})$:
\begin{equation}\label{eq:PGamma}P_\Gamma(\u b)=\prod_{e\in E(\Gamma)}b_e\cdot
\prod_{v\in V(\Gamma)}
N^{\valu_v}_{g_v,n_v}(\u{b_v}),\end{equation}
where  $\u{b_v}$ is the $n_v$-tuple of variables $b_e$ for each $e$ such that the edge $e$ is incident to $v$ (with multiplicity 2 if it is a loop based at $v$), and $N^{\valu_v}_{g_v,n_v}$ are the labeled Kontsevich polynomials from Definition~\ref{def:Kont_poly}.

We consider the linear operator $\cZ$ on polynomials, acting on monomials as 
\[\cZ\left(\prod_i b_i^{d_i}\right)=\frac{1}{\left(\sum_i(d_i+1)\right)!}\prod_i d_i!\cdot \zeta(d_i+1).\]

Then we define the \textit{completed volume} $\cVol \cQ(\dgu)$ of the stratum $\cQ(\dgu)$ by the following formula:
\begin{equation}
\label{eq:compl:volume}
\cVol \cQ(\dgu) =
\sum_{\Gamma \in \cG_{g,l}^{\valu}}
\cVol(\Gamma),
\end{equation}     
where 
\begin{equation}\label{eq:VolGammac}
\cVol(\Gamma):=c_d\cdot \frac{c_{\valu}}{\prod_{v \in V(\Gamma)} c_{\valu_v}} \cdot c_{\Gamma}
\cdot
\cZ\left(P_\Gamma(\u b)
\right)=c_d \cdot \frac{c_{\valu}}{\prod_{v \in V(\Gamma)} c_{\valu_v}} \cdot c_{\Gamma}
\cdot \lim\limits_{N\to\infty}\frac{1}{N^d}\sum_{\substack{\u b\cdot \u h \leq N\\ \u b,\;  \u h\in \N^{|E_{\Gamma}|}}} P_\Gamma(\u b),
\end{equation}
the constants $c_{\valu}, c_{\valu_v}$ are defined in Convention~\ref{convpart} and 
\begin{equation}
\label{eq:const}
  c_d = 2^{d}\cdot 2d, \quad \quad
    c_{\Gamma} = \frac{1}{2^{|V(\Gamma)|-1}} \cdot
\frac{1}{|\operatorname{Aut}(\Gamma)|}.
\end{equation} 
The value of $\cVol(\Gamma)$ does not depend on the choice of edge labeling of $\Gamma\in\cG_{g,l}^{\u\kappa}$.
\label{def:vol}
\end{defi}

Note that $P_\Gamma(\u b)$ is of degree $d-|E(\Gamma)|=2g-2+\ell(\u k) - |E(\Gamma)|$. Indeed, \[\deg(P_{\Gamma}) = |E(\Gamma)| + \sum_v (2g_v-2+\ell(\valu_v)) = |E(\Gamma)| - 2|V(\Gamma)| + 2\cdot \sum_v g_v + \ell(\valu)\] and it is enough to use $g = |E(\Gamma)|-|V(\Gamma)|+1 + \sum_v g_v$.

Finally, the fact that for a polynomial $P_{\Gamma}$ in $|E(\Gamma)|$ variables and of degree $d-|E(\Gamma)|$, the limit $\lim\limits_{N\to\infty}\frac{1}{N^d}\sum_{\substack{\u b\cdot \u h \leq N\\ \u b,\;  \u h\in \N^{|E_{\Gamma}|}}} P(\u b)$ coincides with $\cZ\left(P_\Gamma(\u b)\right)$ is a classical computation that can be found in Lemma 3.7 of \cite{AEZ} for instance, see also Lemma~\ref{lm:evaluation:for:monomial}. 

This definition is the analog of the right-hand side of formula (1.13) in Theorem 1.5 of \cite{DGZZ-vol} in the case of principal strata (we just modified the definition of the operator $\cZ$ by the factor $d!$). The relevance of this definition and the explanation of the denomination `` completed volume'' is given by the next Theorem.

\begin{Theorem}\label{thm:coeffs}
For any odd composition $\dgu=(k_1, \dots, k_r)$ of $4g-4$ with $r\geq 3$ and $\dg_i\geq -1$, 
\[\cVol(\cQ(\dgu))=\Vol(\cQ(\dgu))+ \sum_{ \substack{\u g = \left(g_{1},\ldots, g_{r}\right) \\  0 \le g_{i} \le \frac{\dg_i+ 1}{4}}} \quad
\sum_{\substack{i\;\textrm{s.t.}\\g_i>0}}\quad \sum_{\substack{\u g_{i} = \left(g_{i}^{(1)}, g_{i}^{(2)} \ldots\right) \\ | \u g_{i}| = g_{i},\ g_{i}^{(j)} > 0}}
 C_{\u g, \u g_i} \cdot 
\Vol\left( \cQ(\dgu-4\u g) \times \prod_{i,j} \cH(2g_{i}^{(j)}-2)\right),\] where the coefficients $C_{\u g, \u g_i}$ are obtained as follows.

Let $[n], \overline{[n]}, n \ge -1$, be non-commuting variables. Declare 
\begin{equation}\label{eq:declare_vol_bar}
\overline\Vol(\cQ(\dgu))=\overline{[k_1]}\times \dots \times\overline{[k_r]}
\end{equation} and \begin{equation}\label{eq:declare_vol}
\Vol\left(\cQ(\dgu')\times \prod_{i=1}^s \cH(2g_i-2)\right)=[k'_1]\times \dots [k'_{r}]\times [2g_1-2]\times \dots \times [2g_s-2].
\end{equation}
The coefficients of the formula are obtained by applying the following change of variables in $\overline\Vol(\cQ(\dgu))$, considered as a multilinear form in the variables $\overline{[k_i]}$:
\begin{equation}\label{eq:change_var}
\overline{[k]}=[k]+\sum_{\substack{-1\leq k'<k\\k'=k\mod 4}}(k'+2)c_{k,k'}\cdot [k'], 
\end{equation}
with 
\begin{equation}\label{eq:coeffs}
c_{k,k'}=
\sum_{m>0} \frac{k!!}{(k-2m+2)!!} \cdot \frac{1}{m! \cdot 2^m} \cdot \sum_{\substack{g_1+\dots + g_m=\frac{k-k'}{4}\\ g_i>0}} \prod_i (2g_i-1)\cdot [2g_i-2].
\end{equation}

\end{Theorem}
For the definition/normalization of the volume for products of strata, see section~\ref{subsec:product} and equation~\eqref{eq:Vol-prod-as-prod-Vol}.

\begin{Remark} \label{rem:ge-3-sing}
	The condition $\ell(\dgu) \ge 3$ in Theorem~\ref{thm:coeffs} is of technical nature. It ensures that all degenerations (occurring on the walls) of the metric ribbon graphs in question are contractions of separating loops that separate face-bipartite ribbon graphs (this produces the abelian strata in the final formula), see Lemma~\ref{lem:static-edges-bridges}. In terms of dual ribbon graphs, these degenerations correspond to zero-length static edges which are bridges, see section~\ref{sec:metric}.

	The alternative to these degenerations is the contraction of one or several edges between two distinct vertices, which produces one or several face-bipartite graphs. In terms of dual ribbon graphs, these degenerations correspond to zero-length static edges which are not bridges, again see section~\ref{sec:metric}. Such degenerations can only occur when the ribbon graph has exactly two vertices and the defining equations of the wall involve the perimeters of all boundary components. In our setup this can only happen when we consider square-tiled surfaces with exactly two singularities lying on the same singular level (i.e.\ the corresponding stable graph has one vertex). We exclude the case $\ell(\dgu)=2$ in order to avoid dealing with this type of degenerations: they make the formulas more cumbersome. We plan to address this case in a future work.
\end{Remark}

The proof of Theorem~\ref{thm:coeffs} occupies sections \ref{sec:counting} to \ref{sec:loc:glob}.


\section{Metric ribbon graphs and counting functions}\label{sec:metric}

\subsection{Counting functions for metric ribbon graphs}


A metric on a ribbon graph is called \emph{integer} if the lengths of all edges are integer. Given a metric ribbon graph, the \emph{perimeter of a boundary component} is the sum of lengths of edges incident to this boundary component (if an edge is incident to the boundary component twice, then its length contributes twice to the perimeter).

Denote by $\RGc^{\valu}_{g,n}$ the set of isomorphism classes of ribbon graphs of genus $g$, with $n$ labeled boundary components and the degrees of vertices given by the partition $\valu$.

\begin{defi}
For a ribbon graph $G$ in $\RGc^{\valu}_{g,n}$ and for $\u b = (b_1, \dots, b_n) \in \Z_+^n$, we denote by 
   $ F_G(\u b)$
the number of integer metrics on $G$ with perimeters of the corresponding boundary components equal to $b_1, \dots, b_n$, respectively. We then define the following counting functions for all $g,n,\valu$: 
\begin{equation}
    F_{g,n}^{\valu}(\u b) = \sum_{G\in \RGc^{\valu}_{g,n}} \frac{F_G(\u b)}{|\Aut(G)|}.
\end{equation}
\label{def:counting_fct}
\end{defi}

It is known (see \cite{Blakley}, \cite{Sturmfels}, \cite{Brion-Vergne}, \cite{Norbury-cell} for instance) that for each ribbon graph $G$, $F_G$ is a piecewise quasi-polynomial in $\u b\in\Z^n$ (meaning in this case that it is a piecewise polynomial on each coset of $2\Z^n\subset \Z^n$). However, for our applications we need to make this statement more precise. In particular, we need to specify how the discontinuities of $F_G$ are related to the structure of the ribbon graph $G$. This will be done in the following sections.

\subsection{Passing to the dual ribbon graphs}

It will be more convenient for us to work with dual ribbon graphs.  Recall that the \emph{dual} $G^*$ of a ribbon graph $G$ is constructed as follows (for simplicity, we assume that $G$ is cellularly embedded in a surface). Put a new vertex inside each face of $G$. Then, for each edge $e$ of $G$, join the two new vertices corresponding to the two faces of $G$ incident to $e$ by a new edge $e^*$ which only intersects $e$. Note that if $e$ was incident twice to the same face, then $e^*$ is a loop. The ribbon graph formed by the new vertices and the new edges is the dual ribbon graph $G^*$.
 For any $G \in \RGc^{\valu}_{g,n}$ the dual ribbon graph $G^*$ has genus $g$, $n$ labeled vertices, and face degrees given by $\valu$. Moreover, $|\Aut(G^*)| = |\Aut(G)|$. Denote by $\RGc^{\valu,*}_{g,n}$ the set of isomorphism classes of duals of ribbon graphs in $\RGc^{\valu}_{g,n}$. 

The metrics on any $G \in \RGc^{\valu}_{g,n}$ with perimeters of the boundary components given by $\u b$ are in bijective correspondence with the metrics on $G^*$ with sums of edge lengths around each vertex given by $\u b$ (by assigning to any edge of $G^*$ the length of its dual edge in $G$). By analogy, we call these numbers the \emph{vertex perimeters} of $G^*$. For any dual ribbon graph $G \in \RGc^{\valu,*}_{g,n}$ we denote by $F^*_G(\u b)$ the number of integer metrics on $G$ with vertex perimeters equal to $b_1,\ldots,b_n$. By the remarks above, the counting function $F_{g,n}^{\valu}$ can also be defined as
\begin{equation}
\label{eq:F-dual-def}
        F_{g,n}^{\valu}(\u b) = \sum_{G\in \RGc^{\valu,*}_{g,n}} \frac{F^*_G(\u b)}{|\Aut(G)|}.
\end{equation}

\subsection{Weight functions on non-bipartite ribbon graphs}

We call a \emph{weight function} on a ribbon graph $G$ any function $w:E(G) \rightarrow \R$. For $e \in E(G)$, $w(e)$ is called the \emph{weight} of $e$. Note that the metrics on $G$ are exactly the positive weight functions on $G$. The \emph{vertex perimeters} of a weight function are defined analogously to the case of metrics. The space of all weight functions on $G$ is the vector space $\R^{E(G)}$. For $G \in \RGc^{\valu,*}_{g,n}$, denote by $\vp_G:\R^{E(G)} \rightarrow \R^n$ the map which assigns to a weight function its vertex perimeters. It is a linear map given by the incidence matrix $\big(a_{ve}\big)_{\substack{1\leq v\leq n \\ e \in E(G)}}$, with 
\[ a_{ve}=
\begin{cases} 
2, \mbox{ if } e \mbox{ is a loop based at } v,\\
1, \mbox{ if } e \mbox{ is a not a loop and is incident to } v,\\
0, \mbox{ otherwise.}
\end{cases}
\]

Recall that a graph is bipartite if and only if all of its cycles have even length. Note that if $\valu$ is an odd partition, than the face cycles of any $G \in \RGc^{\valu,*}_{g,n}$ are odd. In particular, $G$ is non-bipartite.

We now give several elementary statements about weight functions on non-bipartite ribbon graphs.

\begin{defi}
    Let $G$ be a non-bipartite ribbon graph. An edge $e \in E(G)$ is called \emph{static} if at least one connected component of $G-e$ is bipartite. The set of static edges of a non-bipartite ribbon graph $G$ is denoted by $S(G)$.
\end{defi}

Recall that a \emph{bridge} of a graph is an edge whose deletion disconnects the graph. Note that if a static edge $e$ is a bridge, then exactly one connected component of $G-e$ is bipartite (because otherwise $G$ would be bipartite). If $e$ is not a bridge, than $G-e$ is connected and bipartite, and $e$ is incident to two vertices from the same part of $G-e$ (again, because otherwise $G$ would be bipartite).

The terminology is explained by the following Lemma \ref{lem:odd-weight-static-edge}, which says that the weights of static edges are uniquely determined by the vertex perimeters.

\begin{lemma}[Weight of a static edge]
\label{lem:odd-weight-static-edge}
    Let $G \in \RGc^{\valu,*}_{g,n}$ be a non-bipartite ribbon graph and let $w$ be a weight function on $G$ with $\vp_G(w)=\u b$. Let also $e \in S(G)$ be a static edge. 
    \begin{itemize}
        \item If $e$ is a bridge, let $G'$ be the connected component of $G-e$ which is bipartite. Color the parts of $G'$ in black and white in such a way that $e$ is adjacent to a black vertex. Let $I, J \subset \{1,\ldots,n\}$ be the labels of black and white vertices in $G'$. Then
    \begin{equation}
    \label{eq:prelim-static-edge-weight-1}
        w(e) = \sum_{i \in I} b_i - \sum_{j \in J} b_j.
    \end{equation}
    \item If $e$ is not a bridge, $G-e$ is connected and bipartite. Color its parts in black and white in such a way that $e$ is adjacent to two black vertices. Let $I, J \subset \{1,\ldots,n\}$ be the labels of black and white vertices in $G-e$.       Then $I \cup J = \{1,\ldots, n\}$ and 
    \begin{equation}
    \label{eq:prelim-static-edge-weight-2}
        w(e) = \frac{1}{2}\cdot \left(\sum_{i \in I} b_i - \sum_{j \in J} b_j \right).
    \end{equation}
    \end{itemize}
\end{lemma}
\begin{proof}
    Compute the sum of all edge weights of the bipartite component of $G-e$ in two ways. On the one hand it is $\sum_{j \in J} b_j$. On the other hand, it is $\sum_{i \in I} b_i - w(e)$ if $e$ is a bridge, and  $\sum_{i \in I} b_i - 2w(e)$ otherwise.
\end{proof}

For a static edge $e \in S(G)$ we denote by $f_e(\u b)$ the linear function on $\R^n$ giving the weight of $e$ as a function of vertex perimeters, (\ref{eq:prelim-static-edge-weight-1}) or (\ref{eq:prelim-static-edge-weight-2}). We do not specify the dependency on $G$ in the notation $f_e$ as it will always be clear from the context.

The proof of the following properties is elementary and is postponed to Appendix \ref{appendix:metrics}.

\begin{lemma}
\label{lem:properties-weight-funcs}
    Let $G \in \RGc^{\valu, *}_{g,n}$ be a non-bipartite ribbon graph. Then:
    \begin{enumerate}
        \item $\operatorname{Im}(\operatorname{vp}_G) = \mathbb{R}^n$;
        \item for every $\u b \in \mathbb{R}^n$, $\operatorname{vp}_G^{-1}(\u b)$ is an affine subspace of $\mathbb{R}^{E(G)}$ of dimension $|E(G)| - |V(G)|$;
        \item regard the coordinates $w(e)$ of $\R^{E(G)}$ as linear functions on $\vp_G^{-1}(\u b)$; then for all $e\in S(G)$, $w(e)$ is constant with value $f_e(\u b)$, and all other functions $w(e), e\in E(G) \setminus S(G)$ are non-constant;
        \item if $\u b \in \Z^n$ and $b_1+\ldots+b_n = 0 \pmod{2}$, then $\vp_G^{-1}(\u b) \cap \Z^{E(G)}$ is a lattice in $\vp_G^{-1}(\u b)$; otherwise, $\vp_G^{-1}(\u b) \cap \Z^{E(G)}$ is empty.
    \end{enumerate}
\end{lemma}
Note that the condition $b_1+\ldots+b_n = 0 \pmod{2}$ is clearly necessary for the existence of an integer weight function with vertex perimeters $\u b$, because this sum is twice the sum of weights of all edges.

\subsection{Polytopes of metrics}

Let $G \in \RGc^{\valu, *}_{g,n}$ be a non-bipartite ribbon graph and let $\u b \in \R^n$. Define the following polytope in $\vp_G^{-1}(\u b)$:
\begin{equation}
\label{eq:odd-one-graph-polytope}
    P_G(\u b) = \{w \in \vp_G^{-1}(\u b) : w(e) \ge 0, e \in E(G) \setminus S(G)\}.
\end{equation}
Note that by point 3 of Lemma \ref{lem:properties-weight-funcs}, for each non-static edge $e \in E(G) \setminus S(G)$ the weight $w(e)$ is a non-constant linear function on $\vp_G^{-1}(\u b)$, and so each condition $w(e) \ge 0$ in (\ref{eq:odd-one-graph-polytope}) defines a half-space.

The significance of this polytope for our problem is that the counting function $F^*_G$ is closely related to the count of integer points inside $P_G$. More precisely, for all $\u b \in \Z^n$ we have:
\begin{equation}
\label{eq:odd-NG-contribution-explicit}
    F^*_G(\u b) = \left( \prod_{e \in S(G)} \mathbf{1}_{f_e(\u b)>0} \right) \cdot |\operatorname{int} P_G(\u b) \cap \Z^{E(G)}|,
\end{equation}
where $\mathbf{1}$ denotes the indicator function and $\operatorname{int}$ denotes the interior relative to $\vp_G^{-1}(\u b)$. 

Define for all $\u b \in \R^n$
\begin{equation}
\label{eq:VG-def}
    V^*_G(\u b) = \left( \prod_{e \in S(G)} \mathbf{1}_{f_e(\u b)>0} \right) \cdot \Vol P_G(\u b),
\end{equation}
where the volume is with respect to the volume form on $\vp^{-1}_G(\u b)$ which is the quotient of standard Euclidean volume forms on $\R^{E(G)}$ and $\R^n$.

Define also for all $g,n,\valu$
\begin{equation}
\label{eq:V-def}
    V^{\valu}_{g,n}(\u b) = \sum_{G \in \RGc^{\valu, *}_{g,n}} \frac{V^*_G(\u b)}{|\Aut(G)|}.
\end{equation}

\subsection{Walls and piecewise (quasi-)polynomiality}

\begin{defi}\label{def:wall}
    A \emph{wall} is a hyperplane or an intersection of hyperplanes in $\R^n$ of the form
    \begin{equation}
    \label{eq:wall-eq}
        \sum_{i \in I} b_i - \sum_{j \in J} b_j = 0,
    \end{equation}
    where $I,J \subset \{1,\ldots,n\}$ and $I \cap J = \emptyset$. The walls generate a subdivision of $\R^n$ into (relatively) open polyhedral pieces of various dimensions which we call \emph{cells}.
\end{defi}
Note that the cells are stable by dilation, and so they are in fact relatively open polyhedral cones.


\begin{prop}
\label{prop:quasi-poly}
    Let $G \in \RGc^{\valu, *}_{g,n}$ be a non-bipartite ribbon graph, let $C$ be a cell in $\R^n$ and let $\ov C$ be its closure. 
    Then:
    \begin{enumerate}
        \item $|\operatorname{int} P_G(\u b) \cap \Z^{E(G)}|$ is a quasi-polynomial for $\u b \in C \cap \Z^n$ (more precisely, a polynomial on each coset of $2\Z^n \subset \Z^n$ intersected with $C$) of degree at most $|E(G)|-|V(G)|$;
        \item $\Vol P_G(\u b)$ is a polynomial for $\u b \in \ov C$, which is either of degree $|E(G)|-|V(G)|$ or identically zero.
    \end{enumerate}
\end{prop}

We give here the idea of the proof of Proposition \ref{prop:quasi-poly}, the formal proof is postponed to Appendix \ref{appendix:metrics}.

When $\u b$ changes inside $C$, the combinatorial structure of the polytope $P_G(\u b)$ remains the same, only the hyperplanes defining its faces change by parallel translations. Moreover, for $\u b \in C \cap \Z^n$, the vertices of $P_G(\u b)$ have rational (half-integer, in fact) coordinates. By a general result from the theory of integer points in polyhedra, the number of integers points in the interior of $P_G(\u b)$ is then a quasi-polynomial in the coordinates of the vertices of $P_G(\u b)$, which are in turn linear functions of $\u b$. Similarly, for $\u b \in C$ the volume of $P_G(\u b)$ is a polynomial in these coordinates. The polynomial expression for the volume is valid on the boundary $\ov C \setminus C$ as well, by continuity (the hyperplanes defining $P_G(\u b)$ depend continuously on $\u b$).


\begin{lemma}
\label{lem:top-is-volume}
    Let $G \in \RGc^{\valu, *}_{g,n}$ be a non-bipartite ribbon graph and let $C$ be a cell in $\R^n$. 
    Then the terms of degree $|E(G)|-|V(G)|$ of the polynomials giving $|\operatorname{int} P_G(\u b) \cap \Z^{E(G)}|$ on 
    \[
     (\Z^n \cap C) \cap \left \{ b_1+\ldots+b_n = 0 \pmod{2} \right \},
    \]
    are all equal to $2\cdot \Vol P_G(\u b)$.
\end{lemma}
\begin{proof}
    If $b_1+\ldots+b_n = 0 \pmod{2}$, then $\vp^{-1}_G(\u b) \cap \Z^n$ is a lattice in $\vp^{-1}_G(\u b)$ by point (4) of Lemma \ref{lem:properties-weight-funcs}. Consider the asymptotics of $|\operatorname{int} P_G(N \cdot \u b) \cap \Z^{E(G)}|$ when $N \rightarrow \infty$. On the one hand, it grows as $N^{|E(G)|-|V(G)|}$ times the term of degree $|E(G)|-|V(G)|$ of the corresponding polynomial. On the other hand, it grows as $N^{|E(G)|-|V(G)|}$ times the volume of $P_G(\u b)$ with respect to the Lebesgue measure on $\vp_G^{-1}(\u b)$ normalized so that the covolume of the integer lattice is one. This volume is twice the quotient volume, since the integer lattice of $\R^{E(G)}$ is sent to the lattice $b_1+\ldots+b_n = 0 \pmod{2}$ in $\R^n$, which is of index 2.
\end{proof}

\begin{Corollary}
\label{cor:top-deg-term}
    Fix $g,n,\valu$ and a cell $C$ in $\R^n$. Then:
    \begin{enumerate}
        \item $F_{g,n}^{\valu}(\u b)$ is a quasi-polynomial for $\u b \in C \cap \Z^n$ (polynomial on each coset of $2\Z^n \subset \Z^n$ intersected with $C$) of degree at most $2g-2+\ell(\valu)$;
        \item $V_{g,n}^{\valu}(\u b)$ is a polynomial for $\u b \in C$;
        \item the terms of degree $2g-2+\ell(\valu)$ of the polynomials giving $F_{g,n}^{\valu}(\u b)$ on 
        \[
        (\Z^n \cap C) \cap \left \{ b_1+\ldots+b_n = 0 \pmod{2} \right \}
        \]
        are all equal to $2\cdot V^{\valu}_{g,n}(\u b)$.
    \end{enumerate}
\end{Corollary}

\begin{proof}
    When $\u b \in C$, the sign ($+$, $-$ or $0$) of any linear function of the form (\ref{eq:wall-eq}) is constant; in particular, the product of indicator functions in (\ref{eq:odd-NG-contribution-explicit}) and (\ref{eq:VG-def}) is constant. Hence the statement follows directly from Proposition \ref{prop:quasi-poly} and  Lemma \ref{lem:top-is-volume}. Note that for any $G \in \RGc^{\valu,*}_{g,n}$, Euler's formula gives $|E(G)|-|V(G)| = 2g-2+\ell(\valu)$.
\end{proof}

\subsection{A result of Kontsevich}

\begin{prop}[ \S 3.3 of \cite{Kon} ] \label{prop:Kon}
For $\u b \in \Z_+^n$ such that $b_1+\ldots+b_n = 0 \pmod{2}$ and $\u b$ is outside of the walls, the top-degree term $2\cdot V^{\valu}_{g,n}(\u b)$ of $F_{g,n}^{\valu}(\u b)$ is the unlabelled Kontsevich \emph{polynomial} $N_{g,n}^{\valu, unlab}(\underline b)$.
\end{prop}

\begin{proof}
This proposition follows from \S 3.3 of \cite{Kon} (see also \cite{Loo} and \cite{Zvon} for the complete proof). Namely 
Kontsevich computes the volumes $\Vol_{symp}(\pi_{m_*}^{-1}(\u b))$ of the fibers of the map $\pi_{m_*}: \cM_{m_*,n}\to \R^n_+$ with respect to the natural symplectic volume form (induced by the 2-form $\Omega$ in the notations of \cite{Kon}), and gets the following formula:
\[\Vol_{symp}(\pi_{m_*}^{-1}(\u b))=\sum_{|\u d|=3g-3+n-M}\frac{\langle\tau_{\u d}\rangle_{m_*}}{\u d!}\u b^{2\u d}.\] 
He then shows that on cells, the ratio of the symplectic volume form $\Omega^d/d!$ and the euclidean quotient volume form $\frac{\prod|d\ell_j|}{\prod |db_i|}$ (where $\ell_j$ denote the lengths of edges of the ribbon graphs) is constant equal to  $\rho=2^{5g-5+2n-2M}$ (\cite[Lemma 3.1 and Appendix C]{Kon}). Note that, up to identifying lower-dimensional faces and factorizing by symmetries, $\pi_{m_*}^{-1}(\u b) = \sqcup_{G \in \RGc^{\valu,*}_{g,n}} P_G(\u b)$. Its euclidean quotient volume is then $V^{\valu}_{g,n}(\u b)$. Using Corollary \ref{cor:top-deg-term}, the top-degree term of $F_{g,n}^{\valu}(\u b)$ is then equal to $2 \cdot V^{\valu}_{g,n}(\u b) = 2 \cdot \frac{\Vol_{symp}(\pi_{m_*}^{-1}(\u b))}{\rho} = N_{g,n}^{\val, unlab}(\u b)$.

\end{proof}

\begin{Remark}Note that the proof of Kontsevich contains some gaps that are addressed in \cite{Loo} and \cite{Zvon}, namely a precise construction of the compactification of the moduli space $\cM_{m_*,n}$ and the possibility to extend some of the introduced cohomology classes on this compactification. \end{Remark}

\subsection{Face-bicolored metric ribbon graphs}

\begin{defi}
    A ribbon graph is \emph{face-bicolored} if its boundary components (faces) are colored in black and white in such a way that any two adjacent boundary components have different colours.
\end{defi}

Denote by $\RGc^{\valu}_{g,(n^\bl,n^\wh)}$ the set of isomorphism classes of face-bicolored ribbon graphs of genus $g$, with $n^\bl$ black and $n^\wh$ white labeled boundary components, and with the degrees of vertices given by the partition $\valu$.

\begin{defi}
For a ribbon graph $G$ in $\RGc^{\valu}_{g,(n^\bl,n^\wh)}$ and for $\u b^\bl \in \Z_+^{n_\bl}$, $\u b^\wh \in \Z_+^{n_\wh}$, we denote by 
\begin{equation*}
    F_G(\u b^\bl; \u b^\wh)
\end{equation*}
the number of integer metrics on $G$ with perimeters of the corresponding black and white boundary components equal to $b^\bl_1, \dots, b^\bl_{n^\bl}$ and $b^\wh_1, \dots, b^\wh_{n^\wh}$, respectively. We then define the following counting functions for all $g,n^\bl,n^\wh,\valu$: 
\begin{equation}
    F_{g,(n^\bl,n^\wh)}^{\valu}(\u b^\bl; \u b^\wh) = \sum_{G\in \RGc^{\valu}_{g,(n^\bl,n^\wh)}} \frac{F_G(\u b^\bl; \u b^\wh)}{|\Aut(G)|}.
\end{equation}
\label{def:counting_fct_ab}
\end{defi}

In this paper, we will only need the counting functions for one-vertex face-bicolored ribbon graphs, i.e. $\ell(\valu)=1$. By Euler's formula we get $\valu = [4g - 2 + 2n^\bl + 2n^\wh]$. These functions were studied in \cite{Yakovlev}. Some results can be recovered in earlier works \cite{OP}, since counting metrics on face-bicolored ribbon graphs is equivalent to counting covers of the sphere ramified over three points, which can be achieved via Hurwitz theory. We recall here a result of \cite{Yakovlev} 
using our notations.

Note that the dual graphs of face-bicolored ribbon graphs are \emph{vertex-bicolored} and, in particular, \emph{bipartite}. Similarly to (\ref{eq:F-dual-def}), one can define $F_{g,(n^\bl,n^\wh)}^{\valu}$ as the count of metrics on dual graphs in $\RGc^{\valu,*}_{g,(n^\bl,n^\wh)}$ with given vertex perimeters. For each $G \in \RGc^{\valu,*}_{g,(n^\bl,n^\wh)}$ one can also define the polytope $P_G(\u b^\bl; \u b^\wh)$ as in (\ref{eq:odd-one-graph-polytope}), except the set of static edges should be replaced by the set of bridges, see \cite[section 3.2]{Yakovlev}. One then defines $V^*_G(\u b^\bl, \u b^\wh)$ and $V_{g,(n^\bl,n^\wh)}^{\valu}(\u b^\bl, \u b^\wh)$ similarly to (\ref{eq:VG-def}) and (\ref{eq:V-def}).

Define the following hyperplane in $\R^{n_\bl} \times 
\R^{n_\wh}$:
\[
W_{n^\bl, n^\wh} = \left \{ \sum_i b_i^\bl = \sum_j b_j^\wh \right \}.
\]
In the context of face-bicolored ribbon graphs, we define the \emph{walls} in $W_{n^\bl, n^\wh}$ as its intersections with one or several hyperplanes of the form
\[
\sum_{i \in I} b^\bl_i - \sum_{j \in J} b^\wh_j = 0,
\]
where $I \subset \{1,\ldots,n^\bl\}$, $J \subset \{1,\ldots,n^\wh \}$.

\begin{prop}[\cite{Yakovlev}]
\label{prop:face-bicolored}
\begin{enumerate}
    \item $F_{g,(n^\bl,n^\wh)}^{[4g - 2 + 2n^\bl + 2n^\wh]}(\u b^\bl; \u b^\wh)$ is a piecewise polynomial of degree at most $2g$. It is identically zero outside of $W_{n^\bl, n^\wh}$.
    \item The term of degree $2g$ of $F_{g,(n^\bl,n^\wh)}^{[4g - 2 + 2n^\bl + 2n^\wh]}$ is equal to $V_{g,(n^\bl,n^\wh)}^{[4g - 2 + 2n^\bl + 2n^\wh]}$.
    \item $V_{g,(n^\bl,n^\wh)}^{[4g - 2 + 2n^\bl + 2n^\wh]}$ is \emph{polynomial} on $W_{n^\bl, n^\wh} \cap (\R_+^{n^\bl} \times \R_+^{n^\wh})$ and outside of the walls. Likewise, it is \emph{polynomial} on any wall intersected with $\R_+^{n^\bl} \times \R_+^{n^\wh}$ and outside of the lower-dimensional walls.
\end{enumerate}
\end{prop}
The corresponding polynomials can be recursively computed, see \cite{Yakovlev}. For our applications, however, we will not need the explicit expressions for these polynomials. For a formula in terms of characters of the symmetric group, see Appendix~\ref{app:hurwitz}.


\section{Counting square-tiled surfaces}\label{sec:counting}
In this section we make the first step towards Theorem~\ref{thm:coeffs} by expressing the Masur--Veech volumes of odd strata in terms of the counting functions for metric ribbon graphs. We will mimic the proof of \cite[Theorem 1.5]{DGZZ-vol}.  The additional contributions coming from the difference between the counting functions $F_{g,n}^{\valu}$ and the Kontsevich polynomials $N_{g,n}^{\valu}$ will be computed in section~\ref{sec:compl}.

\subsection{Volume normalization}\label{subsec:norm}

In this section we recall the canonical construction of the Masur--Veech measure on $\cQ(\dgu)$ and its link with the integral structure given by square-tiled surfaces, as well as the chosen normalization for the volumes in this paper (following \cite{AEZ}, \cite{Gou}, \cite{DGZZ-vol}).

Any pair $(X,q)\in\cQ(\dgu)$, where $X$ is a Riemann surface of genus $g$ and $q$ is a quadratic differential on $X$ with singularities of orders $\dg_i$ (the zeros of $q$ are the singularities of order $\dg_i>0$ and the poles are the singularities of order $-1$) defines a canonical ramified
double cover $\pi:\hat X\to X$ such that $\pi^\ast
q=\hat\omega^2$, where $\hat\omega$ is an Abelian
differential on the double cover $\widehat
X$. The ramification points of $\pi$ are exactly the zeros
and poles of $q$. The double cover $\widehat X$ is endowed
with the canonical involution $\iota$ interchanging the two
preimages of every regular point of the cover. The stratum
$\cQ(\dgu)$ of such differentials is modeled
on the subspace of the relative cohomology of the double
cover $\widehat X$, anti-invariant with respect to the
involution $\iota$, denoted
by $H^1_-(\widehat S,\{\widehat P_1,\dots,\widehat
P_{\zeroes}\};\C)$, where $\{\widehat P_1,\dots,\widehat
P_{\zeroes}\}$ are zeroes of the induced Abelian
differential $\widehat\omega$ and $\widehat S$ is the underlying topological surface. Note that if $p:H^1(\hat S,\{\hat P_1, \dots \hat P_l\}; \C)\to H^1(\hat S; \C)$ is the projection onto absolute periods, then for strata $\cQ(\u k)$ with only odd $k_i$ we have $ker(p)\cap H^1_-(\widehat S,\{\widehat P_1,\dots,\widehat
P_{\zeroes}\};\C)=\{0\}$. This explains the analogy between the odd strata of quadratic differentials and the minimal strata of Abelian differentials that we mentioned in section~\ref{subsec:context}.

Following previous conventions on the normalization of Masur--Veech measures, we define a lattice in $H^1_-(\widehat
S,\{\widehat{P}_1,\ldots,\widehat{P}_{\zeroes}\};\C)$ as
the subset of those linear forms which take values in
$\Z\oplus i\Z$ on $H^-_1(\widehat S,\{\widehat
P_1,\dots,\widehat P_{\zeroes}\};\Z)$. The integer points
in $\cQ(\dgu)$ are exactly those quadratic differentials
for which the associated flat surface with the metric $|q|$
can be tiled with $1/2 \times 1/2$ squares. In this way the
integer points in $\cQ(\dgu)$ are represented by
\textit{square-tiled surfaces}, obtained by gluing isometric euclidean squares of size $1/2\times 1/2$ along sides (vertical to vertical and horizontal to horizontal).

We define the Masur--Veech volume
element $d\!\Vol$ on $\cQ(\dgu)$ as the linear volume
element in the vector space
$H^1_-(\widehat S,\{\widehat{P}_1,\ldots,\widehat{P}_{\zeroes}\};\C{})$ normalized in such a
way that the fundamental domain of the above lattice has unit
volume. The Masur--Veech volume element $d\Vol$ in $\cQ(\dgu)$
induces a volume element on the level sets of the $\Area$ function. In particular
on the level hypersurface $\cQ^{\Area=\frac{1}{2}}(\dgu)$ we get
\begin{equation}
\label{eq:def:Vol}
\Vol \cQ(\dgu)
:= \Vol_1 \cQ^{\Area=\frac{1}{2}}(\dgu)
= 2 d \cdot \Vol \cQ^{\Area\le\frac{1}{2}}(\dgu)\,,
\end{equation}
where $d =  \dim_{\C}\cQ(\dgu)=2g-2+\ell(\dgu)=\frac{|\dgu|}{2}+\ell(\dgu)$.

We denote by $\cST(\cQ(\dgu), N)$ the set of square-tiled surfaces in the
stratum $\cQ(\dgu)$ made of at most $N$ squares. By construction:

\begin{equation}
\label{eq:def:Vol:sq}\Vol \cQ(\dgu)= 2 d \cdot
\lim_{N\to+\infty}
\frac{\card(\cST(\cQ(\dgu), 2N))}{N^{d}}\,.
\end{equation}

\subsection{Volume evaluation}\label{subsec:vol_eval}

\begin{prop}
For a decorated stable graph $\Gamma$ in $\cG_{g,l}^{\valu}$, fixing a labelling of the edges, we define the following piecewise quasi-polynomial in the variables $\u b=(b_1, \dots, b_{|E(\Gamma)|})$:
\begin{equation}\label{eq:FGamma}F_\Gamma(\u b):=\prod_{e\in E(\Gamma)}b_e\cdot
\prod_{v\in V(\Gamma)}
F^{\valu_v}_{g_v,n_v}(\u{b_v}),\end{equation}
where  $\u{b_v}$ is the $n_v$-tuple of variables $b_e$ for each $e$ such that the edge $e$ is incident to $v$ (with multiplicity 2 if it is a loop based at $v$), and $F_{g,n}^{\valu}$ is the counting function defined in Definition~\ref{def:counting_fct}.

Then, for an odd partition $\dgu= \valu-2$ of $4g-4$, the Masur--Veech volume $\Vol \cQ(\dgu)$ of the stratum $\cQ(\dgu)$ is given by the following formula:
\begin{equation}
\label{eq:square:tiled:volume}
\Vol \cQ(\dgu) =
\sum_{\Gamma \in \cG_{g,l}^{\valu}}
\Vol(\Gamma),
\end{equation}     
where $l = \mu_{-1}(\dgu) = \mu_1(\valu)$, 
\begin{equation}\label{eq:VolGamma}
\Vol(\Gamma):=c_d \cdot c_{\valu} \cdot \prod_{v \in V(\Gamma)} \mu_1(\valu_v)! \cdot \frac{1}{|\Aut(\Gamma)|}
\cdot \lim\limits_{N\to\infty}\frac{1}{N^d}\sum_{\substack{\u b\cdot \u h \leq N\\ \u b,\;  \u h\in \N^{|E_{\Gamma}|}}} F_\Gamma(\u b),
\end{equation}
$d=2g-2+\ell(\valu)$, the constant $c_d$ is defined in \eqref{eq:const} and $c_{\valu}$ is defined in Convention~\ref{convpart}.
\label{prop:volsq}
\end{prop}

\begin{proof}
The proof of this proposition is totally parallel to the proof of the main formula in \cite[Theorem 1.5]{DGZZ-vol}. We recap all the main steps here, for completeness and precision on some normalization factors.

A square-tiled surface admits a decomposition into
maximal horizontal cylinders filled with isometric closed
regular flat geodesics. Every such maximal horizontal
cylinder has at least one conical singularity on each of
the two boundary components. 

Let $S$ be a square-tiled surface and let $S=\mathit{cyl}_1\cup \ldots\cup\mathit{cyl}_k$ be its decomposition into the set of maximal horizontal cylinders. To each cylinder $\mathit{cyl}_i$ we associate the corresponding
waist curve $\gamma_i$ considered up to free homotopy.
The curves $\gamma_i$ are non-peripheral (i.e. none of them
bound a disc containing a single pole) and pairwise
non-homotopic. We consider the reduced multicurve $\gamma=\cup \gamma_i$. To $\gamma\subset S$ we associate a decorated stable graph $\Gamma$ (see Definition \ref{def:stable:graph}), defined as the decorated dual
graph to $\gamma$. More precisely, $\Gamma$ is the decorated graph whose vertices represent the components of $S\setminus\gamma$ and are decorated with the degrees of singularities in this component of the square-tiled surface. The edges of $\Gamma$ represent the components $\gamma_i$ of $\gamma$, where the endpoints of the edge associated to $\gamma_i$ are the two vertices corresponding to the two components of $S\setminus\gamma$ adjacent to $\gamma_i$ (that might be the same one). Finally, $\Gamma$ is endowed with $l$ ``legs`` (or half-edges) labeled from $1$ to $l$. The leg with label $i$ is attached to the vertex that represents the component that contains the $i$-th pole in $S$. Note that $\Gamma$ thus constructed is indeed stable: the genus decoration at each vertex is clearly non-negative and the stability condition follows because $\valu$ (or $\dgu$) is an odd partition, see Remark~\ref{rmk:odd-implies-stable}.

Given  a stable graph $\Gamma$ in $\cG^{\valu}_{g,l}$, let us consider
the subset $\cST_{\Gamma}(\cQ(\dgu))$ of those square-tiled
surfaces for which the associated stable graph is $\Gamma$. Let us define 
$\Vol(\Gamma)$ to be the contribution to
$\Vol \cQ(\dgu)$ of square-tiled surfaces from the subset $\cST_\Gamma(\cQ(\dgu))$:
\begin{align}
\label{eq:Vol:gamma}
\Vol(\Gamma)
&:= 2d\cdot
\lim_{N\to+\infty}
\frac{\card(\cST_\Gamma(\cQ(\dgu), 2N)}{N^{d}}\,,
\end{align}
The results in~\cite{DGZZ-meanders}
imply that for any $\Gamma$ in $\cG^{\valu}_{g,l}$ the above limits
exist, are strictly positive, and that
\begin{equation}
\label{eq:Vol:Q:as:sum:of:Vol:gamma}
\Vol \cQ(\dgu)
= \sum_{\Gamma \in \cG^{\valu}_{g,l}} \Vol(\Gamma)\,.
\end{equation}

We now evaluate the contribution of each graph $\Gamma\in\cG^{\valu}_{g,l}$ separately. 

For such a graph $\Gamma$, let $k=|E(\Gamma)|$ be the number of
maximal cylinders filled with closed horizontal
trajectories and denote by $w_1, \dots, w_k$ the lengths of
the waist curves of these cylinders. Since every edge of
any singular layer $v$ is followed by the boundary of the
corresponding ribbon graph twice, the sum of the lengths of
all boundary components of each singular layer $V$ is
integral (and not only half-integral).

Let us consider the collection of linear
forms $f_v=\sum_{e\in E_v(\Gamma)}w_e$ in variables $w_1,\dots, w_k$, where $v$ runs over the vertices $V(\Gamma)$, and $E_v(\Gamma)$ is
the set of edges adjacent to the vertex $v$ (ignoring legs).
It is immediate to see that the $(\Z/2\Z)$-vector space
spanned by all such linear forms has dimension $|V(\Gamma)|-1$.

Let us make a change of variables passing from half-integer
to integer parameters
$b_i:=2w_i$ where $i=1,\dots,k$. Consider the integer sublattice
$\mathbb{L}_\Gamma\subset\Z^k$ defined by the linear relations
\begin{equation}
\label{eq:sublattice:L}
f_{v}(b_1,\dots,b_k)=\sum_{e\in E_v(\Gamma)}b_{e}=0\ (\operatorname{mod} 2)
\end{equation}
for all vertices $v\in V(\Gamma)$. By the above
remark, the sublattice $\mathbb{L}_\Gamma$ has index $2^{|V(\Gamma)|-1}$
in $\Z^k$. We summarize the observations of this section in the following
criterion, that will be useful later in the paper.

\begin{Corollary}
\label{cor:criterion}
A collection of positive numbers
$w_1,\dots,w_{k}$,
where $w_i\in\tfrac{1}{2}\N$ for $i=1,\dots,k$,
corresponds to a square-tiled surface
realized by a stable graph $\Gamma\in\cG_{g,n}$ if and only if $k=|E(\Gamma)|$ and the
corresponding vector $\boldsymbol{b}=2\boldsymbol{w}$ belongs to the sublattice
$\mathbb{L}_\Gamma$. This sublattice has index
$|\Z^{k}:\mathbb{L}_\Gamma|=2^{|V(\Gamma)|-1}$
in the integer lattice $\Z^k$.
\end{Corollary}

   %
   %


Let us now review the different parameters that describe a square-tiled surface of type $\Gamma$. These parameters can be sorted in three independent groups. The parameters in the first group are
responsible for the configurations and the lengths of horizontal saddle connections. In this
group we fix only the lengths $w_1, \dots, w_k$ of the waist curves
of the cylinders filled with closed horizontal trajectories, and we want to estimate the number of choices of the configurations and the lengths of all horizontal saddle connections. The criterion of admissibility of a given
collection $\u{w}=(w_1,\dots,w_k)$ is given by
Corollary~\ref{cor:criterion}. The key observation is that the union of the boundaries of the cylinders coincides with the union of all horizontal saddle connections of the square-tiled surface. These saddle connections form a collection of graphs embedded in the surface, i.e.\ ribbon graphs. Each connected ribbon graph corresponds exactly to a vertex of $\Gamma$. For each vertex $v$, the number of choices for the configuration and the lengths of the saddle connections corresponding to this vertex $v$ is exactly the number of half-integer metric ribbon graphs of genus $g_v$ with vertex degrees $\valu_v$ and with boundary components of perimeters $\u{w}_v$ equal to the perimeters (lengths of the waist curves) of the adjacent cylinders.
After the change of variables $\u b= 2\u w$, this number coincides with the number of integer metric ribbon graphs with perimeters $\u{b_v}$, that is $F^{\valu_v}_{g_v,n_v}(\u{b_v})$. Applying this count to each vertex $v$, we obtain that the total number of choices for the configurations and the lengths of horizontal saddle connections in square-tiled surfaces of type $\Gamma$ with fixed cylinder perimeters $\u w$ is 
\[\prod_{v\in V(\Gamma)}F^{\valu_v}_{g_v,n_v}(\u{b_v}).\]

There are no restrictions on the choice of positive integer
or half-integer heights $h_1,\dots, h_k$ of the cylinders.

Having chosen the widths $w_1, \dots, w_k$ of all maximal cylinders
and the heights $h_1, \dots, h_k$ of the cylinders, the flat area of the
entire surface is already uniquely determined as the sum
$\u{w}\cdot\u{h}=w_1 h_1+\dots +w_k h_k$ of flat areas of
individual cylinders.

However, when the configurations and the lengths of all horizontal saddle connections and
the heights $h_i$ of all cylinders are fixed, there is still a freedom
in the third independent group of parameters. Namely, we can twist
each cylinder by some twist $\phi_i\in\frac{1}{2}\N$ before attaching
it to the singular layer/ribbon graph. Applying, if necessary, an appropriate Dehn twist we
can assume that $0\le\phi_i<w_i$, where $w_i$ is the perimeter of the corresponding cylinder. Thus, having fixed the $w_i$, the choice of configurations and lengths of horizontal saddle connections and the $h_i$, the number of choices for the twists is equal to $(2w_1)\cdot\ldots\cdot(2w_k)=b_1\cdot \ldots b_k$.

We are ready to write a formula for the leading term in
the number of all square-tiled surfaces tiled with at most $2N$
squares represented by the stable graph $\Gamma$ when the integer
bound $N$ becomes sufficiently large:
\begin{multline*}
\card(\cST_\Gamma(\cQ(\dgu), 2N))
\sim
c_{\valu}\cdot \prod_{v \in V(\Gamma)} \mu_1(\valu_v)! \cdot 
\frac{1}{|\operatorname{Aut}(\Gamma)|}\cdot
\sum_{\substack{\u w\cdot\u h \le N/2\\w_i, h_i\in\frac{1}{2}\N\\ 2\u{w}\in\mathbb{L}_\Gamma}}
b_1\cdots b_k\cdot
\prod_{v\in V(\Gamma)}
F^{\valu_v}_{g_v,n_v}(\u{b_v})\,\\
=
c_{\valu}\cdot \prod_{v \in V(\Gamma)} \mu_1(\valu_v)! \cdot
\frac{1}{|\operatorname{Aut}(\Gamma)|}\cdot
\sum_{\substack{\u{b}\cdot\u{H}\le 2N\\b_i, H_i\in\N\\ \u{b}\in\mathbb{L}_\Gamma}}
F_{\Gamma}(\u b)\,,
\end{multline*}
where we have made the change of variables $b_i:=(2w_i)\in\N$ and
$H_i:=(2h_i)\in\N$ in the second line.
The factor $c_{\valu}$ represents the number of ways
to label the $\val_i$-valent vertices of the ribbon graphs for $\val_i \ge 3$. Note that by convention, for each ribbon graph we already know the labels of the univalent vertices (leaves) (corresponding to simple poles of $q$ and also to $n$ marked points): they are given by the labels of the legs incident to the corresponding vertex of $\Gamma$. The factor $\prod_{v \in V(\Gamma)} \mu_1(\valu_v)!$ represents the number of ways to distribute for each ribbon graph these given labels among its leaves.

Now, since the limit as $N\to\infty$ of $\frac{1}{N^d}\sum_{\substack{\u{b}\cdot\u{H}\le 2N\\b_i, H_i\in\N\\ \u{b}\in\mathbb{L}_\Gamma}}
F_{\Gamma}(\u b)$  exists and is positive, and since $F_\Gamma$ is 0 outside $\L_\Gamma$ by the results of section~\ref{sec:metric} and the definition \eqref{eq:FGamma}, we have \[\lim_{N\to\infty}\frac{1}{N^d}\sum_{\substack{\u{b}\cdot\u{H}\le 2N\\b_i, H_i\in\N\\ \u{b}\in\mathbb{L}_\Gamma}}
F_{\Gamma}(\u b)=2^d\lim_{N\to\infty}\frac{1}{N^d}\sum_{\substack{\u{b}\cdot\u{H}\le N\\b_i, H_i\in\N}}
F_{\Gamma}(\u b)\] which finishes the proof of Proposition~\ref{prop:volsq}.

\end{proof}

\subsection{Product of strata}\label{subsec:product}

To prove Theorem~\ref{thm:coeffs}, we need to adapt Proposition~\ref{prop:volsq} to products of strata, and specifically, products of type $\cQ(\dgu)\times\prod_i \cH(2g_i-2)$. To this end, we introduce a few new definitions. 

\begin{defi}A \emph{decorated abelian stable graph} for a stratum $\cH(2g-2)$ is a decorated stable graph $\Gamma$ in the sense of Definition~\ref{def:stable:graph} with a few modifications and additional constraints:
\begin{itemize}
\item $\Gamma$ has only one vertex;
\item $\Gamma$ has no legs (so $|\alpha^{-1}(v)|=n_v$);
\item the decoration at the unique vertex $v$ is $\valu_v=[4g-2]$ (thus not an odd partition anymore);
\item condition \eqref{eq:cond_vertex} is replaced by $2g=2g_v+n_v$.
\end{itemize}
Note that these conditions imply that $\Gamma$ is made of loops only, so $n_v$ is even, and that the genus of $\Gamma$ is $g(\Gamma)=g$.
\end{defi}
The set isomorphism classes of decorated abelian stable graphs associated to a stratum $\cH(2g-2)$ is still denoted by $\cG^{\valu}_g$: the difference with the set of  stable graphs in Definition~\ref{def:stable:graph} can be read of the parity of the decoration $\valu$ and the absence of leg count $l$. It encodes boundary divisors in the compactification $\overline{\cH(2g-2)}$.

\begin{Convention}\label{convAutAbelian} 
	For a decorated abelian stable graph $\Gamma$ with $n$ loops we define $|\Aut(\Gamma)|=n!$, instead of the expected $n!\cdot 2^n$ (by analogy to usual decorated stable graphs). This is explained in the proof of Proposition~\ref{prop:volsq:prod}.
\end{Convention}

For a decorated abelian stable graph $\Gamma$ in $\cG_{g}^{\valu}$, fixing a labeling of the edges, we define the following piecewise quasi-polynomial in the variables $\u b=(b_1, \dots, b_{|E(\Gamma)|})$ analogously to~\eqref{eq:FGamma}:
\begin{equation}\label{eq:FGamma_ab}F_\Gamma(\u b)=\prod_{e\in E(\Gamma)}b_e \cdot F^{\valu_v}_{g_v,(n_v/2, n_v/2)}(\u{b};\u{b}),
\end{equation}
where $v$ is the unique vertex of $\Gamma$, and $F_{g_v,(n_v/2, n_v/2)}^{\valu_v}$ is the counting function defined in Definition~\ref{def:counting_fct_ab}.

The definition and normalization of the Masur--Veech volume can be extended to products of strata as follows: the Masur--Veech measure on the product is the product of the Masur--Veech measures of the strata and we keep the normalization of \eqref{eq:def:Vol}, $d$ being replaced by the dimension of the product, and the stratum $\cQ(\dgu)$ being replaced by the product.  Let  $\cST\big(\cQ(\dgu)\times\prod_i \cH(2g_i-2), N\big)$ be the set of disconnected square-tiled surfaces made of at most $N$ squares in total, such that the connected components belong to the strata $\cQ(\dgu)$ and $\cH(2g_i-2)$ respectively. Equivalently, the Masur--Veech volume of the product is given analogously to~\eqref{eq:def:Vol:sq} by:

\begin{equation}
\label{eq:def:Vol:sq:prod}\Vol \Big(\cQ(\dgu)\times\prod_i \cH(2g_i-2)\Big)= 2 d \cdot
\lim_{N\to+\infty}
\frac{\card\big(\cST\big(\cQ(\dgu)\times\prod_i \cH(2g_i-2), 2N\big)\big)}{N^{d}}\,,
\end{equation}
where $d=\dim_\C\big(\cQ(\dgu)\times\prod_{i=1}^r \cH(2g_i-2)\big)$.

Note that this normalization implies (see \cite[\S 6.2]{EMZ}, \cite[\S 4.4]{AEZ_right} and \cite[\S 3.4]{Gou_SV})

\begin{equation} \label{eq:Vol-prod-as-prod-Vol}
	\Vol \Big(\cQ(\dgu)\times\prod_{i=1}^r \cH(2g_i-2)\Big)=\frac{1}{2^r}\frac{(d'-1)!\Vol\cQ(\dgu)\prod_{i=1}^r 2^{2g_i}(2g_i-1)!\Vol\cH(2g_i-2)}{(d-1)!},
\end{equation}
where $d'=\dim\cQ(\dgu)$, $\dim\cH(2g_i-2)=2g_i$ so $d=d'+2\sum_{i=1}^r g_i$, and $\Vol\cH(2g_i-2)$ is computed in the standard normalization of Masur--Veech volume for Abelian strata, that is: \[\Vol \left(\cH(2g_i-2)\right)= 2g_i \cdot
\lim_{N\to+\infty}
\frac{\card(\cST(\cH(2g_i-2), N))}{N^{2g_i}}\,.\]
Now we are ready to state the analogue of Proposition~\ref{prop:volsq} for products of strata.

\begin{prop}Let $g\geq 0$, $r \ge 1$ and $g_1,\ldots,g_r \geq 1$ be integers, let $\dgu= \valu-2$ be an odd partition of $4g-4$ and let $l = \mu_{-1}(\dgu) = \mu_1(\valu)$.

For a disconnected decorated stable graph $\Gamma=(\tilde\Gamma, \Gamma_1, \dots, \Gamma_r)$ in $\cG_{g,l}^{\valu}\times \prod_i \cG_{g_i}^{[4g_i-2]}$, fixing a labeling of all the edges, we define the following piecewise quasi-polynomial in the variables $\u b=(b_1, \dots, b_{|E(\Gamma)|})$ where $E(\Gamma)=E(\tilde\Gamma)\sqcup\bigsqcup_{i=1}^r E(\Gamma_i)$:
\begin{equation}\label{eq:FGamma:prod}F_\Gamma(\u b):=F_{\tilde\Gamma}(\u{b_{\tilde\Gamma}})\cdot\prod_{i=1}^r
F_{\Gamma_i}(\u{b_{\Gamma_i}}),
\end{equation}
where  $\u{b_{\tilde\Gamma}}$ (resp.  $\u{b_{\Gamma_i}}$) is the set of edge variables associated to $\tilde\Gamma$ (resp. $\Gamma_i$), and the counting functions $F$ for the connected components of $\Gamma$ are defined in~\eqref{eq:FGamma} and~\eqref{eq:FGamma_ab}.

Then the Masur--Veech volume of the product $\cQ(\dgu)\times\prod_{i=1}^r\cH(2g_i-2)$ is given by the following formula:
\begin{equation}
\label{eq:square:tiled:volume:prod}
\Vol \Big(\cQ(\dgu)\times\prod_{i=1}^r\cH(2g_i-2)\Big) =
\sum_{\Gamma \in \cG_{g,l}^{\valu}\times \prod_i \cG_{g_i}^{[4g_i-2]}}
\Vol(\Gamma),
\end{equation}     
where
\begin{equation}\label{eq:VolGamma:prod}
\Vol(\Gamma):=c_d \cdot c_{\valu} \cdot \prod_{v \in V(\tilde\Gamma)} \mu_1(\valu_v)! \cdot 
\frac{1}{|\Aut(\Gamma)|}
 \cdot \lim\limits_{N\to\infty}\frac{1}{N^d}\sum_{\substack{\u b\cdot \u h \leq N\\ \u b,\;  \u h\in \N^{|E_{\Gamma}|}}} F_\Gamma(\u b),
\end{equation}
$d=\dim\big(\cQ(\dgu)\times\prod_{i=1}^r\cH(2g_i-2)\big)$, $c_d$ is as in \eqref{eq:const}, $c_{\valu}$ is defined in Convention~\ref{convpart}, $|\Aut(\Gamma)|=|\Aut(\tilde\Gamma)|\cdot \prod_{i=1}^r |\Aut(\Gamma_i)|$, and $|\Aut(\Gamma_i)|$ are defined in Convention \ref{convAutAbelian}.
\label{prop:volsq:prod}
\end{prop}

\begin{proof}The proof follows exactly the lines of the proof of Proposition~\ref{prop:volsq}. We emphasize only the key points that differ from the first proof. For a disconnected stable graph $\Gamma=(\tilde\Gamma, \Gamma_1, \dots, \Gamma_r)$ in $\cG^\dgu_{g,l}\times \prod_{i=1}^r \cG_{g_i}^{[4g_i-2]}$, as before we denote by $\cST_\Gamma\big(\cQ(\dgu)\times\prod_i \cH(2g_i-2), 2N\big)$ the set of disconnected square-tiled surfaces of type $\Gamma$, i.e. square-tiled surfaces having the decomposition into horizontal cylinders of their connected components encoded by the graphs $\tilde\Gamma$ and $\Gamma_i$ respectively. Then the existence and positivity of the limit $\lim_{N\to+\infty}
\frac{\card\big(\cST_\Gamma\big(\cQ(\dgu)\times\prod_i \cH(2g_i-2), 2N\big)\big)}{N^{d}}$ is derived from the disintegration of Masur--Veech measure along the different area functions (preformed in \cite[\S 6.2]{EMZ} and \cite[\S 4.4]{AEZ_right}) and the existence and positivity of the corresponding limits for the individual connected components  (see \cite{DGZZ-meanders}). We get then the analogue of \eqref{eq:Vol:Q:as:sum:of:Vol:gamma} where $\Vol(\Gamma)$ is defined as $2d$ times this limit. 

Now to count square-tiled surfaces of type $\Gamma$ we introduce as before the linear forms $f_v$ and the related integer sublattice $\L_{\Gamma}\subset\Z^k$ where $k=|E(\Gamma)|$ as in \eqref{eq:sublattice:L}. Note that here, since each graph $\Gamma_i$ has only one vertex, the corresponding linear forms $f_v$ are trivial modulo 2, so  the $(\Z/2\Z)$-vector space spanned by all $f_v$ has dimension $|V(\Gamma')|-1$, and the lattice $\L_{\Gamma}$ has index $2^{|V(\Gamma')|-1}$ in $\Z^k$ and we get a new version of Corollary~\ref{cor:criterion} by replacing the index $2^{|V(\Gamma)|-1}$ by $2^{|V(\Gamma')|-1}$.

The parameters that describe the square-tiled surfaces of type $\Gamma$ are exactly the same as before. The only difference here is that to get a connected component of the square-tiled surface in the stratum $\cH(2g_i-2)$ of Abelian differentials, we need to glue cylinders on a face-bicolored metric ribbon graph: the ``bottom'' and ``top'' boundary components of the cylinders should be glued to the black and white boundary components of the ribbon graph, respectively. Thus the total number of choices for the configurations and the lengths of horizontal saddle connections in square-tiled surfaces of type $\Gamma$ times the number of choices for the twist parameters is now given by $F_\Gamma(\u b)$ defined in \eqref{eq:FGamma:prod}. As before we get 
\[\card(\cST_\Gamma(\cQ(\dgu), 2N))
\sim
c_{\valu} \cdot \prod_{v \in V(\tilde\Gamma)} \mu_1(\valu_v)! \cdot 
\frac{1}{|\operatorname{Aut}(\Gamma)|}\cdot
\sum_{\substack{\u{b}\cdot\u{H}\le 2N\\b_i, H_i\in\N\\ \u{b}\in\mathbb{L}_\Gamma}}
F_{\Gamma}(\u b)\,,\]
where the factor $c_{\valu} \cdot \prod_{v \in V(\tilde\Gamma)} \mu_1(\valu_v)!$ is the number of ways to label the vertices of the ribbon graphs in $\tilde\Gamma$. Note that the term $1/|\Aut(\Gamma)|$ arises because the counting functions $F_{\tilde\Gamma}$, $F_{\Gamma_i}$ count metric ribbon graphs with labeled faces, and we have to forget this labeling. In particular, note that in face-bicolored ribbon graphs, we can only exchange labels of faces of the same color. This means that for each abelian stable graph $\Gamma_i$, we overcount each corresponding metric ribbon graph $|E(\Gamma_i)|!$ times, and not $|E(\Gamma_i)|! \cdot 2^{|E(\Gamma_i)|}$ times. This explains Convention~\ref{convAutAbelian}.

The rest of the proof follows as before: we still get a factor $2^d$ passing from the summation over $\u b\cdot\u H\leq 2N$ to the summation over $\u b\cdot \u H \leq N$.
\end{proof}


\section{Counting functions on the walls}\label{sec:compl}

\subsection{Top-degree terms of $F_{g,n}^{\valu}$ on the walls}

We now present our extension of Kontsevich's Proposition \ref{prop:Kon} to certain walls. We will see later in section~\ref{sec:loc:glob} that in our context we actually only need this computation on very specific walls.

This computation is of independent interest, see Appendix~\ref{app:hurwitz} for the interpretation in terms of characters of the symmetric group.

\begin{Theorem}
\label{thm:counting-funcs-on-walls}
    Fix $g, n, \valu$, with $\valu = (\val_1,\ldots,\val_{\ell(\valu)})$ an odd composition. Let $p \ge 1$ and let $\Pi = (I_0, I^0_1, I^1_1, \ldots, I^0_p, I^1_p)$ be a sequence of $2p+1$ non-empty sets forming a partition of the set $\{1,\ldots,n\}$. If $\ell(\valu) \ge 3$, we also allow $I_0 = \emptyset$. Let $W_{\Pi}$ be the wall in $\R^n$ defined by the $p$ independent equations
    \begin{equation}
    \label{eq:odd-wpi-equations}
        \sum_{i \in I^0_s} b_i = \sum_{j \in I^1_s} b_j,\ s=1,\ldots,p.
    \end{equation}
    Then the top-degree term $2\cdot V_{g,n}^{\valu}$ of $F_{g,n}^{\valu}$ is \emph{polynomial} on $W_{\Pi} \cap \R_+^n$ and outside of the lower-dimensional walls. Moreover, the following relation holds for all $\u b \in W_{\Pi} \cap \R_+^n$ outside of the lower-dimensional walls:
    
    \begin{multline}
    \label{eq:V-recursion}
        N^{\valu, unlab}_{g,n}(\u b) = 2\cdot V_{g,n}^{\valu}(\u b) +
        \sum_{m \ge 1} \ \sum_{\substack{g_0+g_1+\ldots+g_m=g \\ g_i\ge 0}} \  
        \sum_{\substack{A_0 \sqcup \ldots \sqcup A_m = \{1,\ldots,p\} \\ A_i \neq \emptyset}} \ \sum_{\epsilon_1,\ldots,\epsilon_p \in \{0,1\}} \ \sum_{a : \{1,\ldots,m\} \rightarrow \{1,\ldots,\ell(\valu)\}} \\
        \frac{1}{m! \cdot 2^{m+|A_0|}} \cdot \frac{|\Aut(\valu^0)|}{|\Aut(\valu)|} \cdot C_{\u g, \u A, \u \epsilon, a} \cdot 2 \cdot V^{\valu^0}_{g_0,n_0}(\u b_0) \cdot \prod_{i=1}^m V^{[4g_i-2+2n^\bl_i + 2n^\wh_i]}_{g_i, (n^\bl_i,n^\wh_i)}(\u b^\bl_i, \u b^\wh_i),
    \end{multline}
    where:
    \begin{itemize}
        \item $\u b_0 = \left( b_i, i \in I_0 \cup \bigcup_{s \in A_0} (I^0_s \cup I^1_s) \right)$, $\u b^\bl_i = \left( b_j, j \in \bigcup_{s \in A_i} I^{\epsilon_s}_s\right)$, $\u b^\wh_i = \left( b_j, j \in \bigcup_{s \in A_i} I^{1-\epsilon_s}_s\right)$;
        \item $n_0 = |I_0| + \sum_{s \in A_0} (|I^0_s|+|I^1_s|)$, $n^\bl_i = \sum_{s \in A_i} |I^{\epsilon_s}_s|$, $n^\wh_i = \sum_{s \in A_i} |I^{1-\epsilon_s}_s|$;
        \item $\valu^0$ is a composition of length $\ell(\valu)$ with parts $\val^0_i = \val_i - \sum_{j \in a^{-1}(i)} (4g_j + 2n^\bl_j + 2n^\wh_j)$, which should be positive (otherwise the term is considered to be zero),
    \end{itemize}
    and the coefficients $C_{\u g, \u A, \u \epsilon, a}$ are given by:
        \begin{equation}
    \label{eq:odd-degen-coeff}
        C_{\u g, \u A, \u \epsilon, a} = \prod_{i=1}^{\ell(\valu^0)} \left( \val^0_i \cdot \prod_{j\in a^{-1}(i)} (2g_j-1+n^\bl_j+n^\wh_j) \cdot \frac{(\val_i - 2)!!}{(\val_i - 2|a^{-1}(i)|)!!}\right).
    \end{equation}
\end{Theorem}

Note that the polynomiality property of Theorem \ref{thm:counting-funcs-on-walls} actually follows from (\ref{eq:V-recursion}) by induction. Indeed, $N^{\valu, unlab}_{g,n}$ is the Kontsevich polynomial, $V^{[4g_i-2+2n^\bl_i + 2n^\wh_i]}_{g_i, (n^\bl_i,n^\wh_i)}$ are polynomial by Proposition \ref{prop:face-bicolored}, and $V^{\valu^0}_{g_0,n_0}$ can be assumed to be polynomial by the inductive hypothesis. The base cases of this induction are the functions $V^{\valu}_{g,n}(\u b)$ outside of the walls or on the walls of the form $\sum_{i \in I} b_i = \sum_{i \in I^c} b_i$, where $I^c$ is the complement of $I$ in $\{1,\ldots,n\}$. The first ones are polynomial by Proposition \ref{prop:Kon}. For the second ones, recursion (\ref{eq:V-recursion}) simplifies to $N^{\valu, unlab}_{g,n}(\u b) = 2\cdot V_{g,n}^{\valu}(\u b)$, because the third sum is over an empty set ($p=1$).

\begin{Remark}\label{rk:unimaps} 
The polynomial $2 \cdot V^{\valu}_{g,1}(b) = N^{\valu, unlab}_{g,1}(b)$ can be computed explicitly as follows. It is the top-degree term of the counting function $F^{\valu}_{g,1}(b)$, which is, for even $b$, a product of $|\RGc^{\valu}_{g,1}|$ with the polynomial 
\[
\binom{b/2 +E-1}{E-1} = \left | \left \{ (l_1,\ldots,l_E) \in \Z_+^E : 2(l_1+\ldots+l_E) = b \right \} \right |,
\]
where $E=\ell(\valu)-1+2g$ is the number of edges in any graph from $\RGc^{\valu}_{g,1}$.

The enumeration of unicellular maps (ribbon graphs with only one face) with prescribed vertex degrees and unlabelled vertices is performed in \cite[Proposition 12]{Chapuy-Feray-Fusy}. For completeness, we give here the content of \cite[Proposition 12]{Chapuy-Feray-Fusy} in terms of the corresponding Kontsevich polynomials:
\[\sum_{\val \vdash 2E}N_{g,1}^{\val, unlab}(b)\cdot p_{\val}(\u x) = 
\sum_{\rho \vdash 2E}\frac{(2E-\ell(\rho))!}{(E-\ell(\rho)+1)!(E-1)!}2^{\ell(\rho)-2E}m_{\rho}(\u x)\cdot b^{E-1},\]
where $m_\rho(\u x)$ is the monomial symmetric polynomial, and $p_{\val}(\u x)$ is the power sum symmetric  polynomial.
This formula was used at the early stage of this project to clean up normalization issues.
\end{Remark}

The following Proposition can be obtained by adapting the proof of Theorem \ref{thm:counting-funcs-on-walls} to the case of the codimension 1 walls of the form $b_i=0$. This simple proof is given in section \ref{subsec:proof-walls-bi-equal-0}.

\begin{prop}
\label{prop:walls-bi-equal-0}
The labelled Kontsevich polynomials $N_{g,n}^{\valu}$  satisfy the following recursion:
\[N_{g,n+1}^{[\kappa_1, \dots \kappa_r]}(b_1, \dots, b_{n}, 0)=\sum_{\substack{1 \le i \le r \\ \kappa_i\geq 3}}  (\kappa_i-2) N_{g,n}^{[\kappa_1, \dots, \kappa_i-2, \dots, \kappa_r]}(b_1, \dots, b_n).\]
\end{prop}

\begin{Corollary}
Intersection numbers for higher valencies satisfy the following ``string'' equation:
\[\langle \tau_0\tau_{d_1}\dots \tau_{d_n}\rangle_{m_*}=\sum_{i>0}  (m_{i-1}+1)(2i-1)\langle\tau_{d_1}\dots\tau_{d_n}\rangle_{m_*^{(j)}}\] where $m_*^{(i)}=(\dots  m_{i-1}+1, m_i-1, \dots)$.
\end{Corollary}


\subsection{Proof of Theorem \ref{thm:counting-funcs-on-walls}}

This section is devoted to the proof of the recursion (\ref{eq:V-recursion}).

Let $C$ be any cell contained in $W_{\Pi}$ of the same dimension as $W_{\Pi}$ (equivalently, it is a connected component of $W_{\Pi}$ minus the lower-dimensional walls). Let also $C_0$ be any highest-dimensional cell in $\R^n$ adjacent to $C$ (for now the cell $C_0$ is arbitrary, but later we will choose a particular one).

By Corollary \ref{cor:top-deg-term}, $V_{g,n}^{\valu}$ is polynomial on both $C$ and $C_0$. However, these polynomials are (in general) different. To prove the recursion (\ref{eq:V-recursion}), we will compute in two ways the jump 
\[ \left( \lim_{\u b' \rightarrow \u b, \ \u b' \in C_0} 2\cdot V_{g,n}^{\valu}(\u b') \right) - 2\cdot V_{g,n}^{\valu}(\u b),\  \u b \in C.\]

On the one hand, by Proposition \ref{prop:Kon}, for $\u b' \in C_0$ we have $2\cdot V_{g,n}^{\valu}(\u b') = N^{\valu, unlab}_{g,n}(\u b')$, and so this jump is simply 
\begin{equation}
\label{eq:jump-1}
    N^{\valu, unlab}_{g,n}(\u b) - 2 \cdot V_{g,n}^{\valu}(\u b).
\end{equation}

On the other hand, recall from (\ref{eq:VG-def}) and (\ref{eq:V-def}) that for any $\u b'$
\[
V_{g,n}^{\valu}(\u b') = \sum_{G \in \RGc^{\valu, *}_{g,n}} \frac{1}{|\Aut(G)|} \cdot \left( \prod_{e \in S(G)} \mathbf{1}_{f_e(\u b')>0} \right) \cdot \Vol P_G(\u b').
\]
By claim 2 of Proposition \ref{prop:quasi-poly}, $\Vol P_G(\u b')$ is continuous on the closure $\ov C_0$ of $C_0$. Hence the jump comes from the discontinuity of the indicator functions. More precisely, the jump is equal to the sum of contributions of graphs $G \in \RGc^{\valu, *}_{g,n}$ for which the weight of at least one static edge goes from positive to zero when $\u b' \rightarrow \u b$ (i.e.\ $f_e(\u b') \rightarrow 0$ from above). We say that these graphs \emph{degenerate} when $\u b' \rightarrow \u b$. Let us now describe the structure of degenerating ribbon graphs $G$, explaining along the way the combinatorial meaning of the parameters $m, g_i, A_i, \epsilon_i, a$ of the sums in (\ref{eq:V-recursion}).

\begin{lemma}
\label{lem:static-edges-bridges}
    All of the static edges whose weights become zero are bridges of $G$.
\end{lemma}
\begin{proof}
    If one of these edges $e$ was not a bridge, by Lemma \ref{lem:odd-weight-static-edge} its length would be given by a linear function of the form $\frac{1}{2}(\sum_{i \in I} b_i - \sum_{i \in I^c} b_i)$ for some $I \subset \{1,\ldots,n\}$, where $I^c$ is the complement. Thus we must have $\sum_{i \in I} b_i = \sum_{i \in I^c} b_i$ on $C$, and so on $W_{\Pi}$. This equation must be a linear combination of the defining equations (\ref{eq:odd-wpi-equations}) of $W_{\Pi}$. This is only possible if $I_0 = \emptyset$ and the equation is the sum of all defining equations of $W_{\Pi}$. By the assumptions of Theorem \ref{thm:counting-funcs-on-walls}, $I_0 = \emptyset$ implies $\ell(\valu) \ge 3$. This last condition implies that there is a face $f$ of $G$ which is not incident to $e$. This face will also be a face of the graph $G-e$, which is connected and bipartite (because $e$ is static and not a bridge). However, $f$ is of odd degree, and so $G-e$ contains an odd cycle, a contradiction.
\end{proof}

Let $m \ge 1$ be the number of static bridges of $G$ which become zero-weight when $\u b' \rightarrow \u b$. It follows from Lemma \ref{lem:static-edges-bridges} that $G$ consists of several ribbon graphs $G_0,G_1, \ldots, G_m$ connected together into a tree-like structure with these bridges. 
We now describe the properties of the pieces $G_i$, see Figure \ref{fig:odd-degeneration-example} for an illustration.

\begin{figure}
    \centering
    \def\svgwidth{0.95\textwidth}
\begingroup%
  \makeatletter%
  \providecommand\color[2][]{%
    \errmessage{(Inkscape) Color is used for the text in Inkscape, but the package 'color.sty' is not loaded}%
    \renewcommand\color[2][]{}%
  }%
  \providecommand\transparent[1]{%
    \errmessage{(Inkscape) Transparency is used (non-zero) for the text in Inkscape, but the package 'transparent.sty' is not loaded}%
    \renewcommand\transparent[1]{}%
  }%
  \providecommand\rotatebox[2]{#2}%
  \newcommand*\fsize{\dimexpr\f@size pt\relax}%
  \newcommand*\lineheight[1]{\fontsize{\fsize}{#1\fsize}\selectfont}%
  \ifx\svgwidth\undefined%
    \setlength{\unitlength}{306.61795068bp}%
    \ifx\svgscale\undefined%
      \relax%
    \else%
      \setlength{\unitlength}{\unitlength * \real{\svgscale}}%
    \fi%
  \else%
    \setlength{\unitlength}{\svgwidth}%
  \fi%
  \global\let\svgwidth\undefined%
  \global\let\svgscale\undefined%
  \makeatother%
  \begin{picture}(1,0.36365148)%
    \lineheight{1}%
    \setlength\tabcolsep{0pt}%
    \put(0,0){\includegraphics[width=\unitlength,page=1]{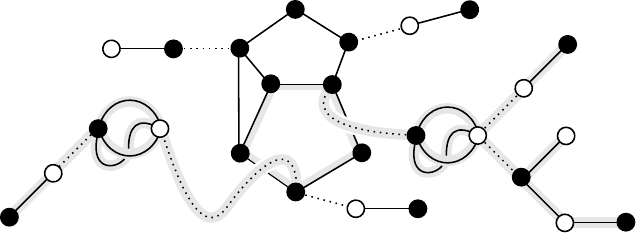}}%
    \put(0.4314083,0.14663674){\makebox(0,0)[lt]{\lineheight{1.25}\smash{\begin{tabular}[t]{l}$G_0$\end{tabular}}}}%
  \end{picture}%
\endgroup%

    \caption{A degeneration of a ribbon graph $G$ in $\RGc^{\valu,*}_{g,n}$ into $9$ components. The zero-weight static bridges are dotted. The unique non-bipartite component $G_0$ has 4 faces of degrees $7,5,5,3$. Two ``branches'' are glued to the corners of one of the faces of degree $5$. Their boundaries (in grey) form together a face of the initial graph $G$.}
    \label{fig:odd-degeneration-example}
\end{figure}

\begin{lemma}
\label{lem:odd-Gi-properties}
    There is exactly one non-bipartite graph among the $G_i$. The bipartite graphs $G_i$ each have one face.
    
\end{lemma}

\begin{proof}
Recall that for any static bridge $e$ one component of $G-e$ is bipartite and the other is not. If among the $G_i$ there were two non-bipartite ribbon graphs, any static bridge separating them would violate the above property. If all the $G_i$ were bipartite, then the initial graph $G$ would also be bipartite, a contradiction. Hence, there is exactly one non-bipartite graph among the $G_i$. 

We now prove that each bipartite graph $G_i$ has one face. Consider a bipartite graph $G_i$ which is a ``leaf'' of the tree-like structure. The unique zero-weight static edge incident to $G_i$ is incident to a corner of some face of $G_i$. If $G_i$ had another face, it would also be a face of $G$, thus of odd degree, which is impossible since $G_i$ is bipartite. Hence $G_i$ has one face and we can remove it from $G$ with the adjacent zero-weight static edge. This decreases the degree of some face of $G$ by an even number and so this face will still be of odd degree. We can continue removing the bipartite graphs $G_i$ in the same manner.
\end{proof}

Label the $G_i$ in such a way that $G_0$ is the non-bipartite graph.
Clearly there are $m!$ different labelings satisfying this condition. Label also the faces of $G$ in such a way that the degree of face $i$ is equal to $\val_i$ for $i=1,\ldots,\ell(\valu)$. This can be done in $|\Aut(\valu)|$ different ways. We say that a degenerating ribbon graph is \emph{well-labeled} if the choices of these two labelings are made.

Let $g_i$ denote the genus of $G_i$.

\begin{lemma}
    There is a partition $A_0 \sqcup A_1 \sqcup \ldots \sqcup A_m$ of $\{1,\ldots,p\}$ and a vector $(\epsilon_1,\ldots,\epsilon_p) \in \{0,1\}^p$ such that:
        \begin{itemize}
            \item the labels of vertices in the bipartite graphs $G_i,\ i=1,\ldots,m$ are $\bigcup_{s \in A_i} I^{\epsilon_s}_s$ for one part and $\bigcup_{s \in A_i} I^{1-\epsilon_s}_s$  for another part;
            \item the labels of vertices in $G_0$ are $I_0 \cup \bigcup_{s \in A_0} I^0_s \cup \bigcup_{s \in A_0} I^1_s$.
        \end{itemize} 
    The partition is defined uniquely, while the vector $(\epsilon_1,\ldots,\epsilon_p)$ is unique up to the $2^{m+|A_0|}$ possible compositions of $m$ involutions $(\epsilon_s, s \in A_i) \leftrightarrow (1-\epsilon_s, s \in A_i), \ i=1,\ldots,m$ and $|A_0|$ involutions $(\epsilon_s) \leftrightarrow (1-\epsilon_s), \ s \in A_0$.
\end{lemma}
\begin{proof}
    Let $e$ be any zero-weight static bridge of $G$. Recall the formula for its weight (Lemma \ref{lem:odd-weight-static-edge}). On $C$ this weight is zero, which implies a linear relation on the $b_i$ of the form $\sum_{i \in I} b_i = \sum_{j \in J} b_j$, where $I,J$ are the labels of vertices in the two parts of the bipartite component of $G-e$. But the only possible relations of this form on $C$ are the defining equations (\ref{eq:odd-wpi-equations}) of $W_{\Pi}$ and the sums thereof. This implies that $I = I^{j_1}_{i_1} \cup \ldots \cup I^{j_s}_{i_s}$ and $J=I^{1-j_1}_{i_1} \cup \ldots \cup I^{1-j_s}_{i_s}$ for some $i_1,\ldots,i_s \in \{1,\ldots,p\}$, $j_1,\ldots,j_s \in \{0,1\}$. Since the choice of a zero-weight static bridge $e$ was arbitrary, the same property is satisfied by the vertex labels in each bipartite $G_i$. Indeed, first apply the observation to the bridges $e$ separating a single bipartite graph $G_i$ (which is at a ``leaf'' of the tree-like structure), then to the ones separating several bipartite graphs $G_i$ whose distance from the ``leaves'' is at most 2, and so on.
\end{proof}


Since for each $i=1,\ldots,m$, the graph $G_i$ has one face, each ``branch'' of the tree-like structure emanating from $G_0$ is also a graph with one face. Take now a face of $G_0$ and all the branches of the tree-like structure that are glued to some corner of this face via a zero-weight static edge. Together all these faces form one face of the initial graph $G$ (see Figure \ref{fig:odd-degeneration-example}). For $i=1,\ldots,m$, let the unique face of $G_i$ be part of the the face of $G$ with label $a(i)$. We call $a:\{1,\ldots,m\} \rightarrow \{1,\ldots,\ell(\valu)\}$ the \emph{attachment map}. 

For $i=1,\ldots, \ell(\valu)$, label the unique face of $G_0$ which is part of the face of $G$ with label $i$ by the same label $i$. By the argument in the previous paragraph, this face of $G_0$ has (odd) degree 
\begin{equation}
\label{eq:odd-G0-face-degrees}
    \val^0_i = \val_i - \sum_{j \in a^{-1}(i)} (2|E(G_j)|+2).
\end{equation}
Note that by Euler's formula 
\[
\val^0_i = \val_i - \sum_{j \in a^{-1}(i)} \left( 4g_j + 2\sum_{s \in A_i} |I^{\epsilon_s}_s|+ 2\sum_{s \in A_i}|I^{1-\epsilon_s}_s| \right).
\]

To sum up, starting from a \emph{well-labeled} degenerating ribbon graph $G$, we obtained the following \emph{degeneration data}:
\begin{itemize}
    \item $m \ge 1$;
    \item $g_0,\ldots,g_m \ge 0$ such that $g_0+\ldots+g_m=g$;
    \item a partition $A_0 \sqcup A_1 \sqcup \ldots \sqcup A_m$ of $\{1,\ldots,p\}$ with $A_i \neq \emptyset$;
    \item a (non-unique) vector $(\epsilon_1,\ldots,\epsilon_p) \in \{0,1\}^p$;
    \item an attachment function $a:\{1,\ldots,m\} \rightarrow \{1,\ldots,\ell(\valu)\}$,
    \item ribbon graphs $G_0, G_1,\ldots,G_m$,
\end{itemize}
satisfying the following conditions (with $\valu^0$ and $n_0$ defined as in Theorem \ref{thm:counting-funcs-on-walls}):
\begin{itemize}
    \item $\val^0_i >0$ for $i=1,\ldots,\ell(\valu)$;
    \item $G_0 \in \RGc^{\valu^0,*}_{g_0,n_0}$, with vertex labels in $I_0 \cup \bigcup_{s \in A_0} I^0_s \cup \bigcup_{s \in A_0} I^1_s$ and, additionally, with a labeling of faces so that the degree of face $i$ is $\val^0_i$, $i=1,\ldots,\ell(\valu)$;
    \item for $i=1,\ldots, m$, $G_i$ is a bipartite ribbon graph of genus $g_i$, with one face, with vertex labels in $\bigcup_{s \in A_i} I^{\epsilon_s}_s$ for one part and $\bigcup_{s \in A_i} I^{1-\epsilon_s}_s$  for another part.
\end{itemize}



The following proposition is the key to the proof of Theorem \ref{thm:counting-funcs-on-walls}.

\begin{prop}
\label{prop:odd-degen-coeff}
    Choose $C_0$ to be a highest-dimensional cell adjacent to $C$ which contains a path $\u b'(\varepsilon),\ \varepsilon \in (0,1]$ such that $\lim_{\varepsilon \rightarrow 0} \u b'(\varepsilon) = \u b$ and
    \begin{equation}
    \label{eq:C0-def}
        \sum_{i \in I^0_j} b'_i(\varepsilon) - \sum_{i \in I^1_j} b'_i(\varepsilon) = 2^j \varepsilon,\ j=1,\ldots,p.
    \end{equation}
    
    The number of ribbon graphs $G$ which degenerate when $\u b'(\varepsilon) \rightarrow \u b$, are well-labeled and have a fixed degeneration data (as above), is equal to $C_{\u g, \u A, \u \epsilon, a} \cdot \left( |\Aut_F(G_0)| \cdot \prod_{i=1}^m |\Aut(G_i)| \right)^{-1}$, where $C_{\u g, \u A, \u \epsilon, a}$ is defined in (\ref{eq:odd-degen-coeff}) and $\Aut_F(G_0)$ is the subgroup of automorphisms of $G_0$ fixing every face.
\end{prop}
Proposition \ref{prop:odd-degen-coeff} is proved in section \ref{sec:count-degens}. We are now ready to finish the proof of Theorem \ref{thm:counting-funcs-on-walls}.
\begin{proof}[End of proof of Theorem \ref{thm:counting-funcs-on-walls} assuming Proposition~\ref{prop:odd-degen-coeff}]
    Choose the highest-dimensional cell $C_0$ adjacent to $C$ as in Proposition \ref{prop:odd-degen-coeff}.

    Suppose $G$ degenerates when $\u b'(\epsilon) \rightarrow \u b$ with a given degeneration data as above. For $i=1,\ldots,m$, color the vertices of $G_i$ with labels in $\bigcup_{s \in A_i} I^{\epsilon_s}_s$ into black, and the ones with labels in $\bigcup_{s \in A_i} I^{1-\epsilon_s}_s$ into white. This makes each $G_i$ an element of $\RGc_{g_i,(n^\bl_i,n^\wh_i)}^{[4g_i - 2 + 2n^\bl_i + 2n^\wh_i],*}$.
    
    For $\u b \in C$ the connecting bridges have zero weights, and so the polytope $P_G(\u b)$ is the product 
    \[P_{G_0}(\u b_0) \times \prod_{i=1}^m P_{G_i}(\u b^\bl_i;\u b^\wh_i),\]
    and the contribution of $G$ to the jump is equal to
    \[
        2 \cdot V^*_{G_0}(\u b_0) \cdot \prod_{i=1}^m V^*_{G_i}(\u b^\bl_i;\u b^\wh_i).
    \]
	Note that there is no factor $1/|\Aut(G)|$, because $|\Aut(G)|=1$. Indeed, if $G$ degenerates, then, as we have seen, it has at least one bridge. Since the vertices of $G$ are labeled, any automorphism of $G$ must fix this edge and its endpoints, and so it is trivial.
	
    By Proposition \ref{prop:odd-degen-coeff}, there are $C_{\u g, \u A, \u \epsilon, a} \cdot \left( |\Aut_F(G_0)| \cdot \prod_{i=1}^m |\Aut(G_i)| \right)^{-1}$ well-labeled degenerating ribbon graphs $G$ with the same degeneration data. 
    Thus, summing the contributions of all $G_i$ compatible with the parameters $g_i, A_i, \epsilon_i, a$, we obtain
    \[
         |\Aut(\valu^0)| \cdot C_{\u g, \u A, \u \epsilon, a} \cdot 2 \cdot V^{\valu^0}_{g_0,n_0}(\u b_0) \cdot \prod_{i=1}^m V^{[4g_i-2+2n^\bl_i + 2n^\wh_i]}_{g_i, (n^\bl_i,n^\wh_i)}(\u b^\bl_i, \u b^\wh_i),
    \]
    where we have used the fact that each $G_0$ is counted $|\Aut(\valu^0)| / |\Aut(G_0) : \Aut_F(G_0)|$ times (recall that $G_0$ is equipped with the additional labeling of faces such that the face with label $i$ has degree $\val^0_i$).
    
    Because of the well-labeling, the contribution of each (non-well-labeled) degenerating ribbon graph is counted $m! \cdot |\Aut(\valu)|$ times in this sum. So the actual contribution to the jump is
    \[
    \frac{1}{m!} \cdot \frac{|\Aut(\valu^0)|}{|\Aut(\valu)|} \cdot C_{\u g, \u A, \u \epsilon, a} \cdot 2 \cdot V^{\valu^0}_{g_0,n_0}(\u b_0) \cdot \prod_{i=1}^m V^{[4g_i-2+2n^\bl_i + 2n^\wh_i]}_{g_i, (n^\bl_i,n^\wh_i)}(\u b^\bl_i, \u b^\wh_i).
    \]

    Summing this over all possible parameters $g_i, \epsilon_i, a$, we overcount by $2^{m+|A_0|}$ because of the non-uniqueness of the vector $\epsilon_i$ in degeneration data. Hence, multiplying by $\frac{1}{2^{m+|A_0|}}$ and summing over all $A_i$ and all $m\ge 1$, we get the total jump. Equating with (\ref{eq:jump-1}) we obtain the desired recursion (\ref{eq:V-recursion}).
\end{proof}

\subsection{Proof of Proposition \ref{prop:walls-bi-equal-0}}
\label{subsec:proof-walls-bi-equal-0}

\begin{proof}[Proof of Proposition \ref{prop:walls-bi-equal-0}]
    Let $C$ be a cell contained in the wall $b_{n+1}=0$ and of the same dimension. Let also $C_0$ be the (unique) highest-dimensional cell adjacent to $C$ and contained in the half space $b_{n+1}>0$.
    
    As in the proof of Theorem \ref{thm:counting-funcs-on-walls}, consider the jump of $V^{\valu}_{g,n}$ when passing from $C_0$ to $C$. Since $F^{\valu}_{g,n}$ and $V^{\valu}_{g,n}$ are both identically zero on $C$ (by definition), this jump is equal to $N^{\valu, unlab}_{g,n}(b_1,\ldots,b_n,0)$. On the other hand, this jump is equal to the sum of contributions of all graphs that degenerate when passing from $C_0$ to $C$. As before, each degenerating graph $G$ has at least one static edge whose weight becomes zero on $C$. However, using Lemma \ref{lem:odd-weight-static-edge}, we see that the only possibility for such static edge is to be an edge incident to the leaf with label $n+1$. Clearly, any graph $G$ with such a leaf does degenerate when passing from $C_0$ to $C$. 

    Label the faces of $G$ in such a way that the degree of face $i$ is equal to $\val_i$. This can be done in $|\Aut(\valu)|$ different ways. The contribution of $G$ to the jump is clearly equal to $V^*_{G'}(b_1,\ldots,b_n)$, where $G'$ is obtained from $G$ by removing the leaf $n+1$. Note that the faces of $G'$ are still labeled. If the leaf was incident to the face $i$ of $G$, then degree of face $i$ of $G'$ is equal to $\val_i-2$. Conversely, starting from $G'$, there are $\val_i-2$ corners of the face $i$ where we can glue the leaf to reconstruct $G$. Hence
    \begin{multline*}
        |\Aut(\valu)|\cdot N_{g,n+1}^{[\kappa_1, \dots \kappa_r], unlab}(b_1, \dots, b_{n}, 0)= \\ 
        \sum_{\substack{1 \le i \le r \\ \kappa_i\geq 3}}  (\kappa_i-2) \cdot |\Aut(\kappa_1, \dots, \kappa_i-2, \dots, \kappa_r)| \cdot N_{g,n}^{[\kappa_1, \dots, \kappa_i-2, \dots, \kappa_r], unlab}(b_1, \dots, b_n),
    \end{multline*}
    which is equivalent to the desired equality.
\end{proof}

\section{Counting degenerations of ribbon graphs}
\label{sec:count-degens}

In this section we prove Proposition \ref{prop:odd-degen-coeff}.

Suppose the degeneration data $m, g_i, A_i, \epsilon_i, a, G_i$ is fixed. We will count the number of ways to join together into a tree-like structure the graphs $G_i$ with $m$ bridges, in such a way as to produce a ribbon graph in $\RGc^{\valu,*}_{g,n}$ which degenerates when $\u b'(\epsilon) \rightarrow \u b$, with the given degeneration data. Equivalently, the weights of these $m$ bridges in the constructed graph should be positive for $\epsilon>0$ and zero for $\epsilon=0$. We call such joinings \emph{admissible}.

Note that, because of the possible symmetries of the pieces $G_i$, the count of admissible joinings is $ |\Aut_F(G_0)| \cdot \prod_{i=1}^m |\Aut(G_i)|$ times the actual count of well-labeled degenerating ribbon graphs (the graph $G_0$ has labeled faces, so we only allow symmetries fixing the faces). Hence we need to show that the number of admissible joinings is equal to $C_{\u g, \u A, \u \epsilon, a}$.

Note that since the attachment map $a$ is fixed, we know which of the graphs $G_i$ are in the branches of the tree-like structure which are joined to the given face of $G_0$. Moreover, by Lemma \ref{lem:odd-weight-static-edge}, the weights of the joining bridges (as functions of $\u b'(\epsilon)$) in any such branch only depend on the branch itself. It means that we can count separately for each $i=1,\ldots,\ell(\valu_0)$ the number of admissible joinings of the face of $G_0$ with label $i$ with the corresponding graphs, and than take the product over $i$. This explains the product structure of the coefficient (\ref{eq:odd-degen-coeff}).

Hence, it is sufficient to consider the case $i=1$. Without loss of generality, assume that $a^{-1}(1)=\{1,2,\ldots,m'\}$, for some $m'\ge 1$. Denote $e_i=|E(G_i)|$. We will prove that the number of admissible joinings of $G_0, G_1, \ldots, G_{m'}$, where the joining bridges incident to $G_0$ are only incident to its face with label $1$, is equal to
\[\val^0_1 \cdot \prod_{i=1}^{m'} e_i\cdot \frac{(2\sigma_1 + \val^0_1 + 2(m'-1))!!}{(2\sigma_1 + \val^0_1)!!},\]
where $\sigma_1 = \sum_{i=1}^{m'} e_i$. This will imply Proposition \ref{prop:odd-degen-coeff}, because by Euler's formula $e_i = 2g_i + |V(G_i)|-1 = 2g_i + \sum_{s \in A_i} |I^{\epsilon_s}_s|+ \sum_{s \in A_i}|I^{1-\epsilon_s}_s| -1 = 2g_i + n^\bl_i + n^\wh_i - 1$, and by (\ref{eq:odd-G0-face-degrees}) we have $2\sigma_1 + \val^0_1 + 2(m'-1) = \val_1 - 2$. 

Up to relabeling the $G_i$, we can assume that $\max A_1 < \max A_2 < \ldots < \max A_{m'}$.

For $i=1,\ldots,m'$, color the vertices of $G_i$ into black and white in such a way that $G_i$ becomes vertex-bicolored and the vertices with labels in $I^0_{\max A_i}$ are colored black (the ones with labels in $I^1_{\max A_i}$ are then automatically white). Denote by $D_i(\u b')$ the linear function of $\u b'$ which is the difference between the sums of black and white vertex perimeters of the graph $G_i$. 

\begin{lemma}
\label{lem:odd-Di-ineq}
    For each $i=2,\ldots,m'$ we have for $\varepsilon>0$
    \begin{equation*}
        D_i(\u b'(\varepsilon)) > D_{i-1}(\u b'(\varepsilon)) + \ldots + D_1(\u b'(\varepsilon)).
    \end{equation*}
\end{lemma}
\begin{proof}
    For any $i=1,\ldots,m'$ we have
    \[
        D_i(\u b'(\varepsilon)) =  \sum_{p \in A_i} \pm \left( \sum_{j \in I^0_p} b'_j(\varepsilon) - \sum_{j \in I^1_p} b'_j(\varepsilon) \right).
    \]
    Using (\ref{eq:C0-def}) and the choice of the vertex-bicoloring of $G_i$, we see that 
    \[
         D_i(\u b'(\varepsilon)) \ge ( 2^{\max A_i} - \sum_{\substack{p \in A_i \\ p \neq \max A_i}} 2^p ) \cdot \epsilon,
    \]
    and 
    \[
        D_{i-1}(\u b'(\varepsilon)) + \ldots + D_1(\u b'(\varepsilon)) \le (\sum_{p \in \bigsqcup_{j=1}^{i-1} A_j} 2^p ) \cdot \epsilon.
    \]
    We conclude by noticing that 
    \[
        2^{\max A_i} > \sum_{p=1}^{\max A_i -1} 2^p \ge \sum_{p \in \bigsqcup_{j=1}^{i-1} A_j} 2^p + \sum_{\substack{p \in A_i \\ p \neq \max A_i}} 2^p.
    \]
\end{proof}

To every possible joining of the $G_i$ we associate a tree $T$ on $m'+1$ vertices labeled from $0$ to $m'$ (the vertices correspond to the $G_i$ and the edges correspond to the joining bridges).

Root the tree $T$ at the vertex $0$. Then every vertex $i\ge 1$ has a well-defined parent vertex and descendant vertices (those for which the unique path to the root passes through this vertex). We will also say that $G_i$ is a parent/descendant of $G_j$, if $i$ is a parent/descendant of $j$ in $T$.

\begin{lemma}
\label{lem:structureAdmissibleBranches}
A joining is admissible if and only if for every $i=1,\ldots, m$ the following holds:
\begin{itemize}
    \item if all of the descendants of $G_i$ have labels smaller then $i$, then the bridge joining $G_i$ to its parent is in the black corner of $G_i$;
    \item otherwise, the bridge joining $G_i$ to its parent and the bridge joining $G_i$ to the subtree containing the descendant of $G_i$ of maximal label are in the corners of $G_i$ of different colors.
\end{itemize}
\end{lemma}


\begin{proof}

In the first case, let $s_p$ be $1$ if the bridge joining $G_i$ to its parent is in the black corner of $G_i$, and $-1$ otherwise. Then, by Lemma \ref{lem:odd-weight-static-edge}, the weight of this bridge at $\u b'(\varepsilon)$ is equal to 
\[s_p D_i(\u b'(\varepsilon)) + \sum_j \pm D_j(\u b'(\varepsilon)),\]
where the sum is over the labels $j$ of all descendants of $G_i$. 

Note that $j<i$ for all descendants $G_j$ of $G_i$. Hence, if $s_p=1$, this is at least $D_i(\u b'(\varepsilon)) - \sum_j D_j(\u b'(\varepsilon)) >0$ when $\varepsilon>0$, by Lemma \ref{lem:odd-Di-ineq}. On the contrary, if $s_p=-1$, this is at most $-D_i(\u b'(\varepsilon)) + \sum_j D_j(\u b'(\varepsilon)) <0$ when $\varepsilon>0$, again by Lemma \ref{lem:odd-Di-ineq}. Hence the weight of this bridge is positive for $\varepsilon>0$ if and only if $s_p=1$, as desired. 

In the second case, let $i'$ be the descendant of $i$ of maximal label and let $i'=i_0, i_1, \ldots, i_r=i$ be the vertices on the unique path from $i'$ to $i$. Let $s_p(i_j)$ be $1$ if the bridge joining $G_{i_j}$ to $G_{i_{j+1}}$ is in the black corner of $G_{i_j}$, and $-1$ otherwise. Similarly, let $s_d(i_j)$ be $1$ if the bridge joining $G_{i_j}$ to $G_{i_{j-1}}$ is in the black corner of $G_{i_j}$, and $-1$ otherwise.

Note that the label $i'=i_0$ is bigger than the labels of all of its descendants, so by the first case $s_p(i_0)=1$. Then for each $1 \leq j \leq r$ the length at $\u b'(\varepsilon)$ of the bridge joining $G_{i_j}$ and $G_{i_{j+1}}$ is equal to 
\begin{equation}
    \label{eq:loop_length}
    \left(\prod_{k=1}^j (-1)\cdot s_p(i_k)\cdot s_d(i_k) \right) \cdot D_{i'}(\u b'(\varepsilon)) + \sum_k \pm D_k(\u b'(\varepsilon)),
\end{equation}
where the sum is over the labels $k$ equal to $i_j$ and to all of descendants of $i_j$ except $i_0=i'$.

Again, $k<i'$ for all $k$, so by the same reasoning as above, the weight of this bridge is positive for $\varepsilon>0$ if and only if the product in (\ref{eq:loop_length}) is equal to $1$. We conclude that $\prod_{k=1}^j (-1)\cdot s_p(i_k)\cdot s_d(i_k)=1$ for all $1 \leq j \leq r$, and so $s_p(i_j) s_d(i_j)=-1$ for all $1 \leq j \leq r$. In particular, $s_p(i_r) s_d(i_r)=s_p(i) s_d(i)=-1$, as desired.

\end{proof}

We now establish a bijection between admissible joinings and certain sequences of markers in the corners of the ribbon graphs $G_i$. Each marker will correspond to a place were one of the joining bridges is glued. Thus each joining bridge produces two markers. If there are several markers in the same corner of $G_i$, their relative order around the vertex is part of the data.

For a marker $a$, denote by $l(a)$ the label of the component in which $a$ is contained. If $l(a)>0$, let $s(a)$ be $1$ if $a$ is contained in a black corner, and $-1$ otherwise.

\begin{lemma}
\label{lem:branchesToArrangements}
There is a bijection between admissible joinings and sequences $(a_1, b_1, \ldots, a_{m'}, b_{m'})$ of $2m'$ markers in the corners of $G_i$ satisfying the following conditions:
\begin{itemize}
    \item $l(a_1)=0$;
    \item for all $1 \leq i \leq m'$, either
    \begin{equation}
    \label{eq:conditionLoops1}
    \begin{cases}
    l(a_i) \in \{0, l(a_1),l(b_1),\ldots,l(a_{i-1}),l(b_{i-1})\}, \\
    l(b_i) = \max{\{0, l(a_1),l(b_1),\ldots,l(a_{i-1}),l(b_{i-1})\}^c}, \\
    s(b_i)=1,
    \end{cases}    
    \end{equation}
    or
    \begin{equation}
    \label{eq:conditionLoops2}
    \begin{cases}
    l(a_i)=l(b_i) \notin \{0, l(a_1),l(b_1),\ldots,l(a_{i-1}),l(b_{i-1})\}, \\
    s(a_i) \neq s(b_i),
    \end{cases}    
    \end{equation}
    where the superscript $c$ stands for the complement in $\{0,1,\ldots,m'\}$.
\end{itemize}
\end{lemma}

\begin{proof}
The sequence of markers corresponding to an admissible joining consists of the $2m'$ places where the joining bridges are glued to the $G_i$, written down in a particular order, which we describe in Algorithm \ref{alg:branchToSequence}. In this algorithm $S$ is interpreted as the set of \emph{non-visited} components. Note that $0 \notin S$ from the start, so component $0$ is considered to be visited from the start. When the algorithm terminates, $res$ contains the corresponding sequence of markers. See Figure \ref{fig:odd-joining-to-markers} for an example computation.

\begin{figure}
    \centering
    \def\svgwidth{0.95\textwidth}
    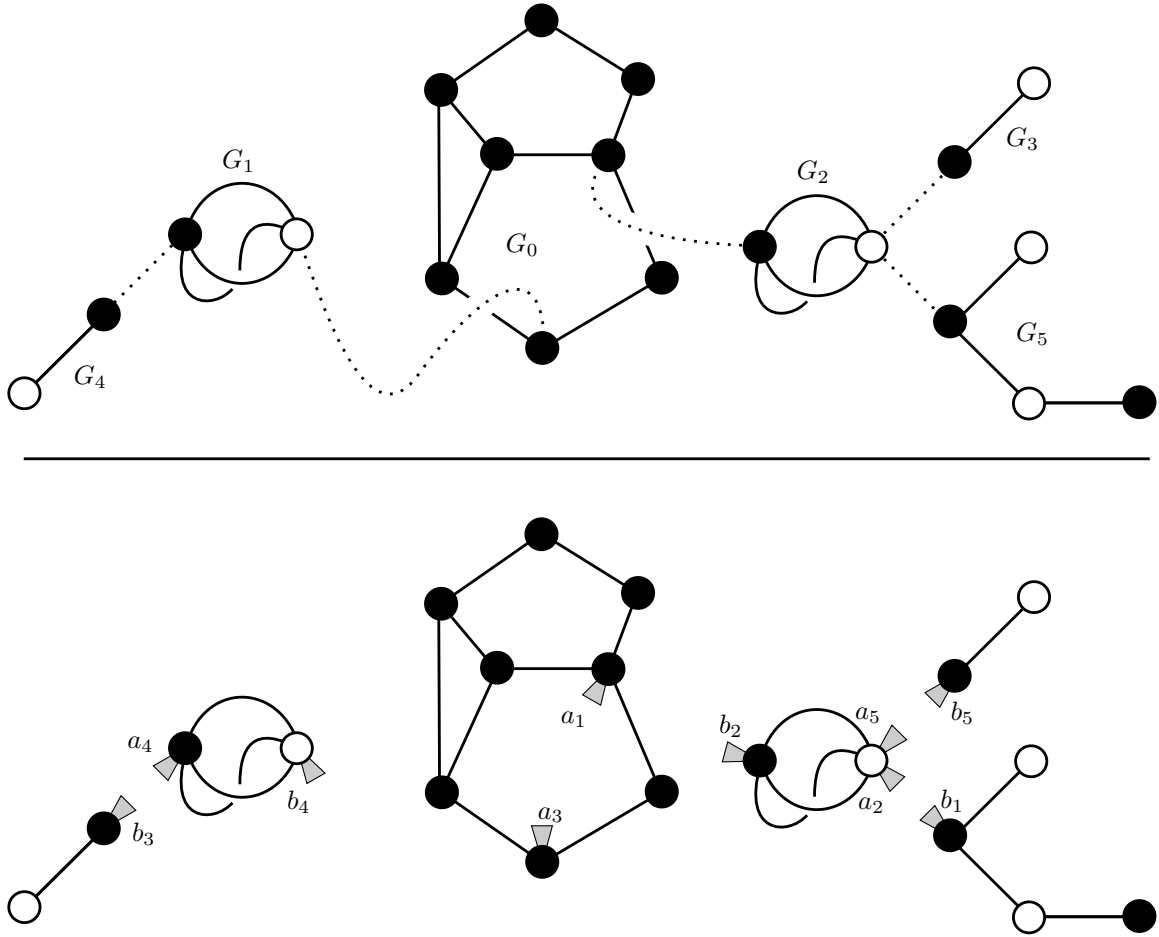
    \caption{An admissible joining (top) and the corresponding sequence of markers (bottom) produced by Algorithm \ref{alg:branchToSequence}.}
    \label{fig:odd-joining-to-markers}
\end{figure}


\begin{algorithm}

\caption{From an admissible joining to a sequence of markers}
\label{alg:branchToSequence}
\begin{algorithmic}[1]
\State $S:=\{1,\ldots,m'\}$
\State $res:=()$
\While{$S\neq \emptyset$}
\State $\gamma = (i_0=\max{S},i_1,\ldots,i_r=0) :=$ path from $\max{S}$ to 0 in $T$ 
\State $i_{r'} :=$ first component in $\gamma$ with label not in $S$ (i.e. already visited)
\State append to $res$ a marker at a place where the bridge joining \newline\hspace*{\algorithmicindent}$G_{i_{r'-1}}$ and $G_{i_{r'}}$ is glued to $G_{i_{r'}}$ 
\label{alg:step-append-first}
\For{$j=0,\ldots,r'-1$} \label{alg:step-for}
\State append to $res$ a marker at a place where the bridge joining \newline\hspace*{\algorithmicindent}\hspace*{\algorithmicindent}$G_{i_{j}}$ and $G_{i_{j+1}}$ is glued to $G_{i_{j}}$
\State append to $res$ a marker at a place where the bridge joining \newline\hspace*{\algorithmicindent}\hspace*{\algorithmicindent}$G_{i_{j}}$ and $G_{i_{j+1}}$ is glued to $G_{i_{j+1}}$
\EndFor
\State append to $res$ a marker at a place where the bridge joining \newline\hspace*{\algorithmicindent}$G_{i_{r'-1}}$ and $G_{i_{r'}}$ is glued to $G_{i_{r'-1}}$ \label{alg:step-append-last}
\State remove from $S$ the components $i_0, \ldots,i_{r'-1}$
\EndWhile
\end{algorithmic}
\end{algorithm}

The algorithm terminates since at each step of the while-loop we remove from $S$ at least one element ($\max{S}$).

In  words, the Algorithm~\ref{alg:branchToSequence} traverses the tree $T$ corresponding to the joining by starting with the $0$-component, then at each iteration of the while-loop it appends to the already traversed subtree the path joining it to the not yet visited component of maximal label. It first writes in $res$ the marker at a place where this path is glued (with a joining bridge) to the already traversed subtree. Then, for each edge along this path, starting from the not yet visited component of maximal label, it appends to $res$ the markers of the joining bridge corresponding to this edge. 
In particular, it is clear that all of the edges of $T$ will be traversed, and so the resulting sequence $res$ will contain all of the $2m'$ markers corresponding to the $m'$ joining bridges.

We now show that the sequence of markers thus constructed satisfies the conditions of the Lemma.

Clearly $l(a_1)=0$. Let $a_i,b_i,\ldots,a_{i+r},b_{i+r}$ be a sequence of markers added to $res$ during one iteration of the while-loop. The marker $a_i$ is added in step \ref{alg:step-append-first} of the Algorithm. In particular, it is in one of the already visited components, so the first condition of (\ref{eq:conditionLoops1}) is satisfied for $a_i$. The marker $b_i$ is the first one added in the for-loop \ref{alg:step-for}. By design, it is in the not yet visited component of maximal label, so the second condition of (\ref{eq:conditionLoops1}) is satisfied for $b_i$. The label of the component of $b_i$ is in particular bigger than the labels of all of its descendants, so by the first point of Lemma~\ref{lem:structureAdmissibleBranches}, $s(b_i)=1$, which is the third point of (\ref{eq:conditionLoops1}).

The markers $a_{i+1},b_{i+1},\ldots,a_{i+r},b_{i+r}$ are added in steps \ref{alg:step-for} to \ref{alg:step-append-last} of the Algorithm. For each $j=1,\ldots,r$, $a_{i+j}$ and $b_{i+j}$ are markers in the corners of $G_{i_{j}}$, and so $l(a_{i+j})=l(b_{i+j})$. Moreover, all of the components $G_{i_j}$ are not yet visited, so the first condition in (\ref{eq:conditionLoops2}) is satisfied for all $a_{i+j}, b_{i+j}$. The second condition of (\ref{eq:conditionLoops2}) follows from the second point of Lemma~\ref{lem:structureAdmissibleBranches}, since for all $j=1,\ldots,r$, the component $\max{S}$ is the descendant of maximal label of the component with label $l(a_{i+j})=l(b_{i+j})$ (because all of its descendants are in $S$ at this stage of the Algorithm).

Finally, we have to show that Algorithm~\ref{alg:branchToSequence} establishes a bijection between admissible joinings and sequences of markers satisfying conditions of the Lemma. We do this by describing the inverse algorithm.

To reconstruct the joining from the sequence of markers $(a_1, b_1, \ldots, a_{m'}, b_{m'})$ it is enough to find which pairs of markers should be joined by a joining bridge. 
Subdivide the sequence of markers into intervals of the form $(a_i,b_i,\ldots,a_{i+r},b_{i+r})$ with the pair $(a_i,b_i)$ satisfying condition (\ref{eq:conditionLoops1}) and the pairs $(a_{i+1},b_{i+1}), \ldots, (a_{i+r},b_{i+r})$ satisfying condition (\ref{eq:conditionLoops2}) (such subdivision is clearly unique). Then treat the intervals from left to right by gluing (for each interval) $b_i$ with $a_{i+1}$, $b_{i+1}$ with $a_{i+2}$, $...$, $b_{i+r-1}$ with $a_{i+r}$, and finally $b_{i+r}$ with $a_{i}$ (which joins this branch with what has been constructed from the previous intervals). 
It is straightforward to check that the two algorithms are inverses of each other.
\end{proof}

Before we conclude the proof of Proposition~\ref{prop:odd-degen-coeff}, we need the following elementary counting lemma.

\begin{lemma}
    \label{lem:leaf_counting}
    Let $n \ge 1$ and $d_1,\ldots, d_n \ge 1$. Suppose that in a rooted tree $T$ any path from the root to a leaf passes successively through $n$ vertices having $d_{\pi(1)},\ldots,d_{\pi(n)}$ children respectively, for some permutation $\pi$ of the set $\{1,\ldots,n\}$ depending on the leaf. Then $T$ has $d_1 d_2 \cdots d_n$ leaves.
\end{lemma}
\begin{proof}
    The proof is by induction on $n$. For $n=1$ the statement is trivial. For $n\ge 2$, suppose the root of $T$ has $d_k$ children for some $k \in \{1,\ldots,n\}$. Each of the root subtrees satisfies the conditions of the lemma with parameters $\{d_i, i \neq k\}$ and so, by induction hypothesis, has $\prod_{1 \leq i \leq n, i \neq k} d_i$ leaves. Since there are $d_k$ such subtrees, $T$ has $d_1 d_2\cdots d_n$ leaves in total, as desired.
\end{proof}

\begin{proof}[Conclusion of the proof of Proposition~\ref{prop:odd-degen-coeff}]
It remains to count the sequences of markers satisfying the conditions of Lemma~\ref{lem:branchesToArrangements}. We do this by successively choosing the markers $a_i, b_i$ and keeping track of how many choices we are left with after each step. However, the number of choices we have at each step \emph{depends on the choices we made before}, so we cannot directly apply the combinatorial product rule. Instead, we will represent all possible choices as a (rooted) decision tree (with leaves corresponding to the final sequences of markers) and then apply Lemma~\ref{lem:leaf_counting} to this tree.

Clearly, there are $\val^0_1$ choices for the location of $a_1$ (since our branches are only glued to the face of $G_0$ of degree $\val^0_1$). By (\ref{eq:conditionLoops1}) we necessarily have $l(b_1)=m'$ and $s(b_1)=1$, for which there are $e_{m'}$ choices. At each next step $i=2,\ldots,m'$, we can freely choose the location of the marker $a_i$ among all of the $\val^0_1+2e_1+\ldots+2e_{m'}+2(i-1)$ places available at this stage (before stage $i$ we have already chosen the places for $a_1,b_1,\ldots,a_{i-1},b_{i-1}$, which creates $2(i-1)$ additional places). 

\emph{Depending on the choice of $a_i$}, there are two possibilities: 
\begin{itemize}
    \item if $l(a_i) \in \{0, l(a_1),l(b_1),\ldots,l(a_{i-1}),l(b_{i-1})\}$, the $l(b_i)$ and $s(b_i)$ are now fixed by (\ref{eq:conditionLoops1}), and there are $e_{l(b_i)}$ choices for the place of $b_i$, since at this stage none of the markers has been placed in the component with label $l(b_i)$;
    \item if $l(a_i) \notin \{0, l(a_1),l(b_1),\ldots,l(a_{i-1}),l(b_{i-1})\}$, the $l(b_i)$ and $s(b_i)$ are now fixed by (\ref{eq:conditionLoops2}), and there are $e_{l(b_i)}$ choices for the place of $b_i$, since at this stage only one marker (namely, $a_i$) has been placed in the component with label $l(b_i)$, but it was placed in a corner of color different to that of $b_i$, so this does not affect the number of choices.
\end{itemize}

When we have finally chosen the positions of all of the markers $a_1,b_1,\ldots,a_m,b_m$, the numbers of choices we had along the way were equal to 
\begin{equation}
    \label{eq:decision_tree_degrees}
    \val^0_1, e_{l(b_1)}, \val^0_1+2\sigma_1+2, e_{l(b_2)}, \val^0_1+2\sigma_1+4, e_{l(b_3)}, \ldots, \val^0_1+2\sigma_1+2(m'-1), e_{l(b_{m'})}.
\end{equation}

Note that $l(b_1),\ldots,l(b_{m'})$ is a permutation of the set $\{1,\ldots,m'\}$ (this follows from (\ref{eq:conditionLoops1}) and (\ref{eq:conditionLoops2}); $l(b_i)$ is always among the labels we have not used before). It means that our decision tree satisfies the conditions of Lemma~\ref{lem:leaf_counting} with parameters (\ref{eq:decision_tree_degrees}), and so it has
\[\val^0_1 e_1 \cdots e_{m'} (\val^0_1+2\sigma_1+2) (\val^0_1+2\sigma_1+4) \ldots (\val^0_1+2\sigma_1+2(m'-1))\]
leaves, which completes the proof.
\end{proof}


\section{Local to global}\label{sec:loc:glob}

In this section, we finish the proof of Theorem \ref{thm:coeffs} using Theorem~\ref{thm:counting-funcs-on-walls}. That is, we study which local contributions matter in the global computation of volumes, and how these contributions arrange together to give the formula of Theorem \ref{thm:coeffs}.

Note that in Section \ref{sec:compl} we computed the higher order completion terms on all walls (in the sense of Definition \ref{def:wall}), outside of the walls of higher codimension. We will see that we actually only need this computation on very specific walls. 

We will need the following lemma that extends Lemma 2.3 of \cite{DGZZ-vol} to piecewise quasi-polynomials.
\begin{lemma}
\label{lm:evaluation:for:monomial}
Let $\mathbb{L}$ be a sublattice of finite index
$|\Z^k:\mathbb{L}|$ in the integer lattice $\Z^k$
and let $m_1,\dots,m_k\in\N$ be any positive integers.
The following formula holds
\begin{eqnarray}
\label{eq:3.7}
\lim_{N\to+\infty}
\frac{1}{N^{|m|+k}}
\sum_{\substack{\u{b}\cdot\u{h}\le N\\b_i, h_i\in\N\\ \u{b}\in\mathbb{L}}}
b_1^{m_1}\cdots b_k^{m_k}& =\frac{1}{|\Z^k:\mathbb{L}|}\cdot
\lim_{N\to+\infty}
\frac{1}{N^{|m|+k}}\sum_{\substack{\u{b}\cdot\u{h}\le N\\b_i, h_i\in\N}}
b_1^{m_1}\cdots b_k^{m_k}.
\\ & =
\frac{1}{|\Z^k:\mathbb{L}|} \cdot \cZ(b_1^{m_1}\cdots b_k^{m_k}) \,.\notag
\end{eqnarray}
Furthermore, if $W$ is a wall of codimension $r$ in $\R^k$, not contained in any of the walls $\{b_i=0\}$, and $C$ is an open cell of $W\cap \R_+^k$ (i.e. a polyhedral cone open in $W\cap \R_+^k$, see Definition~\ref{def:wall}),  then
\begin{equation}\label{eq:asympt_wall}
\exists c\geq 0 :\quad 
\sum_{\substack{\u{b}\cdot\u{h}\le N\\b_i, h_i\in\N\\ \u b\in C\cap \L}}b_1^{m_1}\dots b_k^{m_k} \sim  c\, N^{|m|+k-r}\quad \mbox{ as }N\to\infty,
\end{equation}
with $c>0$ if and only if $W\cap \L$ is a lattice in $W$.
\end{lemma}

\begin{proof}The first statement is exactly Lemma 2.3 of \cite{DGZZ-vol} (see also \cite[Lemma 3.7]{AEZ} for the proof). The second statement follows from similar arguments that we detail here for completeness.

 We first assume that $W\cap \L$ is a lattice in $W$. Consider the simplex $\Delta^k=\{(x_1, \dots, x_k)\in\R^k_+, x_1+\dots + x_k\leq 1\}$ and introduce the variables $y_i:=\frac{b_i}{N}$. For a fixed $\u h\in\N^k$ and $N>0$, we consider the maps $f_{N}:(b_i)_i\mapsto (y_i)_i$ and $f_{\u h}:(x_i)_i\mapsto (x_ih_i)_i$. By hypothesis the set $\Delta_{\u h, C}:= f_{\u h}^{-1}(\Delta^k)\cap C\subset W\cap \R_+^k$ is non-empty and it is a polyhedron in $W$ of dimension $k-r$.
 We write \[\sum_{\substack{\u{b}\cdot\u{h}\le N\\b_i, \u h\in\N^k\\ \u b\in C\cap \L}}b_1^{m_1}\dots b_k^{m_k}=N^{|m|}\sum_{h_i\in \N}\sum_{\u y\in \Delta_{\u h, C}\cap f_N(W\cap\L)}y_1^{m_1}\dots y_k^{m_k} \] and use the equidistribution of the lattice $f_N(W\cap\L)$ in $W$ for large $N$  to approximate the last sum by the integral \[\int_{\u y\in \Delta_{\u h,C}} y_1^{m_1}\dots y_k^{m_k} \frac{d\lambda_W}{\Covol(f_N(W\cap \L),W)}\] where $\lambda_W$ is a normalization of the Lebesgue measure on $W$ and $\Covol(f_N(W\cap \L),W)$ is the covolume of the lattice $f_N(W\cap \L)$ in $W$ for that normalization, namely \[\Covol(f_N(W\cap \L))=\det(\restriction{f_N}{W})\Covol(W\cap\L,W)=\frac{\Covol(W\cap\L,W)}{N^{k-r}}.\] Now, for all $\u h \in \N^k$ we have \[0<\int_{\u y\in \Delta_{\u h,C}} y_1^{m_1}\dots y_k^{m_k} d\lambda_W\leq \int_{\u y\in f_{\u h}^{-1}(\Delta^k)\cap W} y_1^{m_1}\dots y_k^{m_k} d\lambda_W\leq \prod_{i=1}^k \frac{1}{h_i^{m_i}}\int_{\u y\in f_{\u h}^{-1}(\Delta^k)\cap W} d\lambda_W.\]
 To estimate the last volume, we let $v^{(1)}, \dots v^{(r)}$ be the generators of the cone $W\cap\{y_i\geq 0\}$. Without loss of generality we assume that they are linearly independent, otherwise we decompose the cone into simplicial ones. Denote $\ell_{\u h}(\u y)=\sum_i y_i h_i$. Then $f_{\u h}^{-1}(\Delta^k)\cap W$ is the simplex with vertices $0$ and $\frac{v^{(i)}}{\ell_{\u h}(v^{(i)})}$ for $i=1, \ldots, r$. Its volume is then bounded by 
 \[\prod_{i=1}^r\left\Vert\frac{v^{(i)}}{\ell_{\u h}(v^{(i)})}\right\Vert=\frac{c_W}{\prod_{i=1}^r \ell_{\u h}(v^{(i)})}\leq \frac{c_W'}{\prod_i\sum_{v^{(i)}_j\neq 0}h_j}\leq \frac{c_W''}{\sum_i\sum_{v^{(i)}_j\neq 0}h_j}\] for some constants $c_W$, $c'_W$ and $c''_W$, where we use the inequality $ab\geq \frac{1}{2}(a+b)$ which holds for $a,b\geq 1$ in the last estimate. Since $W$ is not contained in any of the hyperplanes $\{y_i=0\}$, every $h_j$ for $j=1, \dots k$ appears in the denominator at least once. So the last expression is bounded by $\frac{c_W''}{\sum_{j=1}^k h_j}$. In the end, we have 
 \begin{align*}
\sum_{\u h\in\N^k} \frac{1}{(\sum h_i)h_1^{m_1}\dots h_k^{m_k}}\leq\sum_{\u h\in\N^k} \frac{1}{(\sum h_i)h_1\dots h_k}=\sum_{\u h\in\N^k}\frac{\sum h_i}{(\sum h_i)^2h_1\dots h_k}\\=\sum_i\sum_{\u h}\frac{1}{(\sum h_i)^2h_1\dots h_{i-1}h_{i+1}\dots h_k}<\infty\end{align*}
 because each sum over $\u h$ is bounded by some linear combination of multiple zeta values (see \cite[Proof of Proposition 2.1]{Kaneko_MZV}). We use the dominated convergence theorem to justify the inversion of the sum on $\u h$ and the limit in $N\to\infty$ and the existence of the limit $c$: indeed the sum approximates the integral from below.
 
Finally, if $W\cap\L$ is not a lattice in $W$, it is a lattice in a subspace of higher codimension and the arguments above apply verbatim to give a polynomial growth in $N$ with smaller exponent, hence $c=0$.
\end{proof}

The proof of Theorem~\ref{thm:coeffs} below is conceptually straightforward, however accounting for all of the symmetry/automorphism factors is delicate. Before passing to the proof, we introduce the following simple modification of decorated stable graphs, which will be useful for dealing with symmetries.

\begin{defi}
	For a composition $\valu$, an \emph{ordered decorated stable graph with total decoration $\valu$} is a decorated stable graph for a stratum $\cQ(\dgu)$ as in Definition~\ref{def:stable:graph} ($\valu = \dgu + 2$), with the following additional data:
	\begin{itemize}
		\item for each $v \in V(\Gamma)$, an ordering of $\valu_v$, which we still denote by $\valu_v$ (i.e.\ $\valu_v$ is now a composition);
		\item a bijection $\phi: \{(i,v) : v \in V(\Gamma), 1 \le i \le \ell(\valu_v)\} \rightarrow \{1, \ldots, \ell(\valu)\}$ satisfying the following two properties:
			\begin{itemize}
				\item the $i$-th element of $\valu_v$ is equal to $\val_{\phi(i,v)}$;
				\item suppose a vertex $v \in V(\Gamma)$ has $l_v$ legs with labels $L_1, \ldots, L_{l_v}$; let also $i_1, \ldots, i_{l_v}$ be the indices of elements in $\valu_v$ which are equal to $1$; then we require that $\phi(i_1,v), \ldots, \phi(i_{l_v},v)$ are the indices (in some order) of the $L_1$-th, ..., $L_{l_v}$-th elements of $\valu$ which are equal to $1$.
			\end{itemize}
	\end{itemize}
\end{defi}

Two ordered decorated stable graphs with the same total decoration $\valu$ are isomorphic if there exists an isomorphism of the underlying unordered decorated stable graphs which preserve the \emph{ordered} versions of $\valu_v$ and the values of $\phi$, that is, if $f: \Gamma \rightarrow \Gamma'$ is such an isomorphism and $\phi, \phi'$ are the corresponding bijections, we require $\phi'(i,f(v)) = \phi(i,v)$ for all $v \in V(\Gamma)$. Note that the automorphisms of an ordered decorated stable graph are exactly the automorphisms of the corresponding unordered stable graph which fix every vertex.

\begin{lemma}
	Let $\valu$ be a composition and let $\Gamma \in \cG^{\valu}_{g,l}$ be a decorated stable graph. Then the number of distinct \emph{ordered} decorated stable graphs with total decoration $\valu$ whose underlying unordered graph is $\Gamma$ is equal to 
	\begin{equation} \label{eq:num-orderings}
		\left( \prod_{v \in V(\Gamma)} \frac{\ell(\valu_v)!}{|\Aut(\valu_v)|} \right) \cdot \left( \frac{|\Aut(\valu)|}{\mu_1(\valu)!} \cdot \prod_{v \in V(\Gamma)} \mu_1(\valu_v)! \right) \cdot \frac{|\Aut_V(\Gamma)|}{|\Aut(\Gamma)|},
	\end{equation}
	where $\Aut_V(\Gamma)$ is the subgroup of automorphisms of $\Gamma$ fixing every vertex.
\end{lemma}
\begin{proof}
	The first term represents the number of distinct orderings of $\valu_v$. The second term represents the number of distinct bijections $\phi$. Because of the symmetries of $\Gamma$ each ordered graph is overcounted $|\Aut(\Gamma):\Aut_V(\Gamma)|$ times, which gives the third term.
\end{proof}

\begin{proof}[Proof of Theorem~\ref{thm:coeffs}]
Recall the expression \eqref{eq:compl:volume} for $\cVol(\cQ(\dgu))$ and fix a decorated stable graph $\Gamma\in\cG^{\valu}_{g,l}$. Recall also the lattice $\L_{\Gamma}$ defined by the linear relations \eqref{eq:sublattice:L}, that is
\[\L_{\Gamma}=\bigcap_{v\in V(\Gamma)} \L_v, \mbox{ with } \L_v=\{|\u{b_v}|=0 \mbox{ mod }2\},\]
where $\u{b_v}$ are the variables corresponding to the edges adjacent to $v$. By Corollary~\ref{cor:criterion}, it has index $2^{|V(\Gamma)|-1}$. Applying \eqref{eq:3.7} to the definition \eqref{eq:VolGammac}, we see that $\cVol(\Gamma)$ is equal to 
\begin{equation} \label{eq:limitN}
c_d \cdot \frac{c_{\valu}}{\prod_{v \in V(\Gamma)} c_{\valu_v}} \cdot \frac{1}{|\Aut(\Gamma)|} \cdot \lim\limits_{N\to\infty}\frac{1}{N^d}\sum_{\substack{\u{b}\cdot\u{h}\le N\\b_i, h_i\in\N}}P_\Gamma(\u b)\chi_{\L_\Gamma}(\u b),
\end{equation} 
where $\chi_{\L_\Gamma}$ is the characteristic function of $\L_{\Gamma}$. Writing $P_{\Gamma}$ as in \eqref{eq:PGamma} along $\L_{\Gamma}$ we get 
\begin{align} \label{eq:prodV}
	P_\Gamma(\u b)\chi_{\L_\Gamma}(\u b) &= \prod_{e\in E(\Gamma)}b_e \cdot \prod_{v\in V(\Gamma)} N^{\valu_v}_{g_v, n_v}(\u{b_v})\chi_{\L_v(\u{b_v})} \nonumber \\
	&= \prod_{v \in V(\Gamma)} |\Aut(\valu_v)| \cdot \prod_{e\in E(\Gamma)}b_e \cdot \prod_{v\in V(\Gamma)} N^{\valu_v, unlab}_{g_v, n_v}(\u{b_v})\chi_{\L_v(\u{b_v})}. 
\end{align}

Recall that in Theorem~\ref{thm:counting-funcs-on-walls} $\valu$ is a composition, while the $\valu_v$ coming from the stable graph $\Gamma$ are only partitions. To be able to apply Theorem~\ref{thm:counting-funcs-on-walls} (and to keep track of the symmetries), choose any \emph{ordered} decorated stable graph $(\Gamma_{ord},\phi)$ with total decoration $\valu$ whose underlying unordered graph is $\Gamma$ (we will later account for this arbitrary choice by dividing the count by the number of such orderings). Hence the $\valu_v$ now refer to compositions.

Note that in the term $N^{\valu_v, unlab}_{g_v, n_v}(\u{b_v})$ we are evaluating $N^{\valu_v, unlab}_{g_v, n_v}$ on the wall (which we denote by $W_v$) defined by the equations of the form $b_i=b_j$ for each edge forming a loop at $v$ in $\Gamma$. Observe that either there is at least one non-loop edge of $\Gamma$ incident to $v$, or $v$ is the unique vertex of $\Gamma$, in which case $\ell(\valu_v) = \ell(\valu) \ge 3$. In both cases we can apply Theorem~\ref{thm:counting-funcs-on-walls} to replace each $N^{\valu_v, unlab}_{g_v, n_v}(\u{b_v})$ by the right-hand side of \eqref{eq:V-recursion}, which is equal to $2\cdot V^{\valu_v}_{g_v, n_v}(\u{b_v})$ plus the additional completion terms.

By Corollary~\ref{cor:top-deg-term}, the terms $2\cdot V^{\valu_v}_{g_v, n_v}(\u{b_v})\chi_{\L_v}$ are the top-degree terms of the counting functions $F^{\valu_v}_{g_v, n_v}(\u{b_v})$ on the walls $W_v$ and outside of lower-dimensional walls (where they are given by piecewise polynomials of at most the same degree). Lemma \ref{lm:evaluation:for:monomial} implies then that the large $N$ limit above is not affected by the replacement of $2\cdot V^{\valu_v}_{g_v, n_v}(\u{b_v})\chi_{\L_v}$ by $F^{\valu_v}_{g_v, n_v}(\u{b_v})$. Performing this replacement, we get
\begin{align*}
	& c_d \cdot \frac{c_{\valu}}{\prod_{v \in V(\Gamma)} c_{\valu_v}} \cdot \frac{1}{|\Aut(\Gamma)|} \cdot \prod_{v \in V(\Gamma)} |\Aut(\valu_v)| \cdot \lim\limits_{N\to\infty}\frac{1}{N^d}\sum_{\substack{\u{b}\cdot\u{h}\le N\\b_i, h_i\in\N}}F_\Gamma(\u b) \\
	& = c_d \cdot c_{\valu} \cdot \frac{1}{|\Aut(\Gamma)|} \cdot \prod_{v \in V(\Gamma)} \mu_1(\valu_v)! \cdot \lim\limits_{N\to\infty}\frac{1}{N^d}\sum_{\substack{\u{b}\cdot\u{h}\le N\\b_i, h_i\in\N}}F_\Gamma(\u b)
\end{align*}
which is equal to $\Vol(\Gamma)$ by definition. Summing over all $\Gamma_{ord}$ and dividing by the number of such orderings we obviously still get $\Vol(\Gamma)$. Summing over all $\Gamma$, we get $\Vol(\cQ(\valu))$, which is the first term in the desired formula of Theorem~\ref{thm:coeffs}. It is now left to show that the additional completion terms of \eqref{eq:V-recursion} arrange together to give the rest of the terms in Theorem~\ref{thm:coeffs}. 

To this end, replace in \eqref{eq:prodV} each term  $N^{\valu_v, unlab}_{g_v, n_v}(\u{b_v})$ by the sum of corresponding completion coefficients of \eqref{eq:V-recursion}. Multiplying out over $v \in V(\Gamma)$, we obtain a sum of terms of the form
\[
\operatorname{const} \cdot \prod_{e \in E(\Gamma)} b_e \cdot \prod_{v \in V(\Gamma)} \chi_{\L_v}(\u b_v) \left( 2\cdot V^{\valu^0_v}_{g_{0,v}, n_{0,v}}(\u b_{0,v}) \cdot \prod_{i=1}^{m_v} V^{[4g_{i,v}-2+2n^\bl_{i,v} + 2n^\wh_{i,v}]}_{g_{i,v}, (n^\bl_{i,v},n^\wh_{i,v})}(\u b^\bl_{i,v}, \u b^\wh_{i,v}) \right),
\]
where, for each $v \in V(\Gamma)$, we have adjoined a label $v$ to the indices of the corresponding variables in \eqref{eq:V-recursion}. Note that each such term corresponds to a certain choice of the parameters $m_v, g_{i,v}, A_{i,v}, \varepsilon_{i,v}, a_v$ appearing in \eqref{eq:V-recursion}. Fix a choice of these parameters.

Now, for each $v \in V(\Gamma)$, we have the following in \eqref{eq:V-recursion}:
\begin{itemize}
    \item $p_v$ is the number of loops of $\Gamma$ based at $v$;
    \item $|I^0_{i,v}|=|I^1_{i,v}|=1$ for all $i=1,\ldots,p_v$;
    \item $\u b^\bl_{i,v} = \u b^\wh_{i,v}$ and $n^\bl_{i,v} = n^\wh_{i,v} = |A_{i,v}|$ for all $i=1,\ldots,m_v$.
\end{itemize}
In particular, if we denote $\L'_v = \{| \u b_{0,v}| = 0 \pmod{2}\}$, then $\chi_{\L_v}(\u b_v) = \chi_{\L'_v}(\u b_{0,v})$, and we can rewrite the last expression as
\[
\operatorname{const} \cdot \prod_{e \in E(\Gamma)} b_e \cdot \left( \prod_{v \in V(\Gamma)} 2\cdot V^{\valu^0_v}_{g_{0,v}, n_{0,v}}(\u b_{0,v}) \cdot \chi_{\L'_v}(\u b_{0,v}) \right) \cdot \prod_{\substack{v \in V(\Gamma) \\ 1 \le i \le m_v}} V^{[4(g_{i,v}+n^\bl_{i,v})-2]}_{g_{i,v}, (n^\bl_{i,v},n^\bl_{i,v})}(\u b^\bl_{i,v}, \u b^\bl_{i,v}).
\]
By Corollary~\ref{cor:top-deg-term} and Proposition~\ref{prop:face-bicolored}, this is the top-degree term (on the corresponding wall) of
\begin{equation}\label{eq:prod_count_funcs}
\operatorname{const} \cdot \prod_{e \in E(\Gamma)} b_e \cdot \left( \prod_{v \in V(\Gamma)} F^{\valu^0_v}_{g_{0,v}, n_{0,v}}(\u b_{0,v}) \right) \cdot \prod_{\substack{v \in V(\Gamma) \\ 1 \le i \le m_v}} F^{[4(g_{i,v}+n^\bl_{i,v})-2]}_{g_{i,v}, (n^\bl_{i,v},n^\bl_{i,v})}(\u b^\bl_{i,v}, \u b^\bl_{i,v}).
\end{equation}
Let $s_{i,v} = |a_v^{-1}(i)|$ and define
\[
\Gamma'_{ord} = \Gamma'_{0,ord} \sqcup \bigsqcup_{a=1}^{\ell(\valu)} \bigsqcup_{b=1}^{s_{\phi^{-1}(a)}} \Gamma'_{a,b},
\] 
where:
\begin{itemize}
	\item $\Gamma'_{0,ord}$ is an \emph{ordered} decorated stable graph obtained from $(\Gamma_{ord},\phi)$ by:
	\begin{itemize}
		\item removing from each vertex $v$ the loops corresponding to variables $\u b^\bl_{i,v}$, $1 \le i \le m_v$;
		\item changing the decoration of $v$ from $\valu_v$ to $\valu^0_v$;
		\item keeping $\phi$ the same (this uniquely determines the total decoration $\valu^0$ of $\Gamma'_{0, ord}$);
		\item adding $\mu_1(\valu^0_v) - \mu_1(\valu_v)$ new legs to $v$;
		\item relabeling all the legs in such a way that, if $i_1, \ldots, i_{l_v}$ are the indices in $\valu^0_v$ of elements equal to $1$, and if $\phi(i_1,v), \ldots, \phi(i_{l_v},v)$ are the indices of the $L_1$-th, ..., $L_{l_v}$-th elements of $\valu^0$ which are equal to $1$, then legs at $v$ have labels $L_1,\ldots,L_{l_v}$;
	\end{itemize}   
	\item for each $a = 1,\ldots,\ell(\valu)$, let $\phi^{-1}(a) = (i,v)$ and let $a_v^{-1}(i) = \{j_1,\ldots,j_{s_{i,v}}\}$ with $j_1 < \ldots < j_{s_{i,v}}$; then for each $b = 1,\ldots,s_{i,v}$, $\Gamma'_{a,b}$ is a decorated abelian stable graph for a stratum ${\cH(2(g_{j_b,v}+n^\bl_{j_b,v})-2)}$ with a single vertex and $n^\bl_{j_b,v}$ loops corresponding to the variables $\u b^\bl_{j_b,v}$.
\end{itemize}
Note that $\Gamma'_{0,ord}$ is indeed stable: the genus decoration at each vertex $v$ is equal to $g_{0,v} \ge 0$ by construction and the stability condition follows because $\valu^0$ is an odd partition, see Remark~\ref{rmk:odd-implies-stable}.

Let $\Gamma'$ be the disconnected decorated stable graph obtained from $\Gamma'_{ord}$ by forgetting the ordering of $\Gamma'_{0,ord}$. Then \eqref{eq:prod_count_funcs} is clearly proportional to $F_{\Gamma'}(\u b)$.

Note that by construction, the total decoration $\valu^0$ of $\Gamma'_{0, ord}$ satisfies $\ell(\valu^0) = \ell(\valu)$ and 
\begin{equation} \label{eq:k0-condition}
	\valu^0_a + 4 \cdot \sum_b g(\Gamma'_{a,b}) = \valu_a 
\end{equation}
for each $a=1,\ldots,\ell(\valu^0)$.

Conversely, fix an arbitrary $\Gamma'_{ord} = \Gamma'_{0,ord} \sqcup \bigsqcup_{a=1}^{\ell(\valu)} \bigsqcup_{b=1}^{s_a} \Gamma'_{a,b}$ with $s_a \ge 0$, $\Gamma'_{0,ord}$ an ordered stable graph, $\Gamma'_{a,b}$ abelian stable graphs, the total decoration $\valu^0$ of $\Gamma'_{0,ord}$ satisfying \eqref{eq:k0-condition}. We now identify how many distinct choices of the parameters $(\Gamma_{ord}, \phi)$ and $m_v, g_{i,v}, A_{i,v}, \varepsilon_{i,v}, a_v$ produce the given $\Gamma'_{ord}$ (note that we consider the indexation of abelian components of $\Gamma'_{ord}$ to be part of the data).

Given $\Gamma'_{ord}$ we can recover $(\Gamma_{ord}, \phi)$ (and so $\Gamma$ as well) as follows. Denote by $\phi'$ the $\phi$-bijection of $\Gamma'_{0,ord}$. Then add at each vertex $v$ of $\Gamma'_{0,ord}$ as many loops as there are in total in the graphs $\Gamma'_{a,b}$ with $a = \phi'(i, v)$ for some $i \in \{1, \ldots, \ell(\valu^0_v)\}$. Change the decoration of each vertex from $\valu^0_v = \left(\valu^0_{\phi'(i,v)}, i=1,\ldots, \ell(\valu^0_v)\right)$ to $\valu_v = \left( \valu_{\phi'(i,v)}, i=1,\ldots, \ell(\valu^0_v) \right)$. Keep the same $\phi'$. Remove from each vertex $\mu_1(\valu^0_v) - \mu_1(\valu_v)$ legs. Relabel all the legs in such a way that, if $i_1, \ldots, i_{l_v}$ are the indices in $\valu_v$ of elements equal to $1$, and if $\phi'(i_1,v), \ldots, \phi'(i_{l_v},v)$ are the indices of the $L_1$-th, ..., $L_{l_v}$-th elements of $\valu$ which are equal to $1$, then legs at $v$ have labels $L_1,\ldots,L_{l_v}$. Condition~\eqref{eq:k0-condition} ensures that the constructed $\Gamma_{ord}$ is indeed an ordered decorated stable graph with total decoration $\valu$.

Now for each $v \in V(\Gamma)$ and each $i=1,\ldots,\ell(\valu_v)$, the sizes of the sets $a_v^{-1}(i)$ are fixed by $|a_v^{-1}(i)| = s_{\phi'(i,v)}$. In particular, this fixes $m_v = \sum_{i=1}^{\ell(\valu_v)} |a_v^{-1}(i)| = \sum_{i=1}^{\ell(\valu_v)} s_{\phi'(i,v)}$. Choose the $a_v$ in one of 
\[\prod_{v\in V(\Gamma)} \binom{m_v}{s_{\phi'(1,v)}, \ldots, s_{\phi'(\ell(\valu_v),v)}}\] 
possible ways. Once $a_v$ are fixed, the values of $g_{i,v}$ and $n^\bl_{i,v}$ are uniquely determined: if $a_v^{-1}(i) = \{j_1,\ldots,j_{s_{\phi'(i,v)}}\}$ with $j_1 < \ldots < j_{s_{\phi'(i,v)}}$, then
\begin{align*}
	 \left( (g_{j_1,v}+n^\bl_{j_1,v}), \ldots, (g_{j_{s_{\phi'(i,v)}},v}+n^\bl_{j_{s_{\phi'(i,v)}},v}) \right) &= \left( g(\Gamma'_{\phi'(i,v),1}), \ldots, g(\Gamma'_{\phi'(i,v),s_{\phi'(i,v)}}) \right), \\
	 \left( n^\bl_{j_1,v}, \ldots, n^\bl_{j_{s_{\phi'(i,v)}},v} \right) &= \left( |E(\Gamma'_{\phi'(i,v),1})|, \ldots, |E(\Gamma'_{\phi'(i,v),s_{\phi'(i,v)}})| \right). 
\end{align*}
The sizes of the sets $A_{i,v}$ are now uniquely determined by $|A_{i,v}| = n^\bl_{i,v}$.
Thus there are 
\[\prod_{v \in V(\Gamma)} \binom{p_v}{n^\bl_{1,v}, \ldots, n^\bl_{m_v,v}, n_{0,v} }\]
 choices for the $A_{i,v}$, where $p_v$ and $n_{0,v}$ are the total numbers of loops at the vertex $v$ in $\Gamma_{ord}$ and $\Gamma'_{0,ord}$ respectively. Finally, note that the values of $n^\bl_{i,v}, n^\wh_{i,v}, \u b^\bl_{i,v}, \u b^\wh_{i,v}$ do not actually depend on the choice of $\varepsilon_{i,v}$. Thus there are $\prod_{v \in V(\Gamma)} 2^{p_v}$ choices for the $\varepsilon_{i,v}$. It is easy to check that any such choice of parameters indeed produces $\Gamma'_{ord}$.

We can now rewrite $\cVol(\cQ(\dgu)) - \Vol(\cQ(\dgu))$ as
\begin{align*}
	& \sum_{\Gamma} c_d \cdot \frac{c_{\valu}}{\prod_{v \in V(\Gamma)} c_{\valu_v}} \cdot \frac{1}{|\Aut(\Gamma)|} \cdot \prod_{v \in V(\Gamma)} |\Aut(\valu_v)| \cdot \\
	& \left( \prod_{v \in V(\Gamma)} \frac{\ell(\valu_v)!}{|\Aut(\valu_v)|} \cdot \frac{|\Aut(\valu)|}{\mu_1(\valu)!} \cdot \prod_{v \in V(\Gamma)} \mu_1(\valu_v)! \cdot  \frac{|\Aut_V(\Gamma)|}{|\Aut(\Gamma)|} \right)^{-1} \cdot \\
	& \Bigg[ \sum_{\Gamma'_{ord}} \prod_{v\in V(\Gamma)} \binom{m_v}{s_{\phi'(1,v)}, \ldots, s_{\phi'(\ell(\valu_v),v)}} \cdot \prod_{v \in V(\Gamma)} \binom{p_v}{n^\bl_{1,v}, \ldots, n^\bl_{m_v,v}, n_{0,v} } \cdot \prod_{v \in V(\Gamma)} 2^{p_v} \cdot  \\
	& \prod_{v \in V(\Gamma)} \frac{1}{m_v! \cdot 2^{m_v+|A_{0,v}|}} \cdot \frac{|\Aut(\valu^0_v)|}{|\Aut(\valu_v)|} \cdot \\
	& \prod_{i=1}^{\ell(\valu)} 
	\left(
	\val^0_i \cdot \prod_{\substack{ (j,v) :\\ \phi'(a_v(j),v) = i }} (2(g_{j,v} + n^\bl_{j,v})-1) \cdot \frac{(\val_i - 2)!!}{(\val_i - 2 \cdot |\{ (j,v) : \phi'(a_v(j),v) = i \}|)!!}
	\right) \cdot \\
	& \lim\limits_{N\to\infty}\frac{1}{N^d}\sum_{\substack{\u{b}\cdot\u{h}\le N\\b_i, h_i\in\N}} F_{\Gamma'}(\u b) \Bigg],
\end{align*}
where in the first line the sum is over all $\Gamma\in\cG^{\valu}_{g,l}$ and the terms come from \eqref{eq:limitN} and \eqref{eq:prodV}; the second line accounts for the number of choices of orderings of $\Gamma$ from \eqref{eq:num-orderings}; in the third line the sum is over all $\Gamma'_{ord}$ obtained from all orderings of $\Gamma$ by the procedure described above, and the terms account for the number of times each $\Gamma'_{ord}$ is obtained; the forth and the fifth lines contain the multiplicative terms from \eqref{eq:V-recursion}; in the sixth line, as before, $\Gamma'$ is obtained from $\Gamma'_{ord}$ by forgetting the ordering of $\Gamma'_{0,ord}$.

Using \eqref{eq:VolGamma:prod}, we can rewrite the last line as
\[
c_{d'}^{-1} \cdot c_{\valu^0}^{-1} \cdot \left( \prod_{v \in V(\Gamma'_0)} \mu_1(\valu^0_v)! \right)^{-1} \cdot |\Aut(\Gamma')| \cdot \Vol(\Gamma'),
\]
where $d'$ is the dimension of the product of strata corresponding to $\Gamma'$ and 
\[
|\Aut(\Gamma')| = |\Aut(\Gamma'_0)| \cdot \prod_{a,b} |\Aut(\Gamma'_{a,b})| = |\Aut(\Gamma'_0)| \cdot \prod_{i,v} n^\bl_{i,v}!
\]
(see Convention~\ref{convAutAbelian}, the numbers of loops of $\Gamma'_{a,b}$ are given by $n^\bl_{i,v}$).

The double sum over $\Gamma$ and $\Gamma'_{ord}$ can be replaced by a double sum over all possible compositions $\valu^0$ and all possible disconnected stable graphs $\Gamma' = \Gamma'_0 \sqcup \bigsqcup_{a=1}^{\ell(\valu)} \bigsqcup_{b=1}^{s_a} \Gamma'_{a,b}$, where $\Gamma'_0$ has total decoration $\valu^0$ and $\Gamma'_{a,b}$ are abelian, all together satisfying condition~\eqref{eq:k0-condition}. Note that an additional factor 
\[
\left( \prod_{v \in V(\Gamma)} \frac{\ell(\valu^0_v)!}{|\Aut(\valu^0_v)|} \right) \cdot \left( \frac{|\Aut(\valu^0)|}{\mu_1(\valu^0)!} \cdot \prod_{v \in V(\Gamma)} \mu_1(\valu^0_v)! \right) \cdot \frac{|\Aut_V(\Gamma'_0)|}{|\Aut(\Gamma'_0)|}
\]
must be introduced, accounting for the forgetting of the ordering when passing from $\Gamma'_{ord}$ to $\Gamma'$.

Note the following relations:
\begin{itemize}
	\item $c_d = c_{d'}$ since $d=d'$ (because $\dim \cQ(\dg^0) = 2(h_1(\Gamma'_0) + \sum_v g_{0,v})-2+\ell(\valu^0)$, $h_1(\Gamma'_0) = h_1(\Gamma')-\sum_{i,v} n^\bl_{i,v}$ and $\dim \cH(2(g_{i,v}+n^\bl_{i,v})-2) =  2(g_{i,v}+n^\bl_{i,v})$, hence $d' = 2(h_1(\Gamma') + \sum_{i,v} g_{i,v}) - 2 + \ell(\valu) = 2(h_1(\Gamma') + \sum_{v} g_{v}) - 2 + \ell(\valu) = d$);
	\item $|\Aut(\valu)| = c_{\valu} \cdot \mu_1(\valu)!$ and similarly for $|\Aut(\valu^0)|$ and $|\Aut(\valu_v)|$;
	\item $\ell(\valu_v) = \ell(\valu^0_v)$ for all $v$;
	\item $|A_{0,v}|=n_{0,v}$;
	\item $|\Aut_V(\Gamma)| / \prod_v 2^{p_v} p_v! = |\Aut_V(\Gamma'_0)| / \prod_v 2^{n_{0,v}} n_{0,v}!$, because both are equal to the number of automorphisms fixing each vertex of the graph obtained from $\Gamma$ (or $\Gamma'_0$) be removing all loops;
	\item $\prod_{j,v} (2(g_{j,v} + n^\bl_{j,v})-1) = \prod_{a,b} (2 g(\Gamma'_{a,b})-1)$;
	\item $|\{ (j,v) : \phi'(a_v(j),v) = i \}| = s_i$;
	\item $\prod_v 2^{m_v} = \prod_i 2^{s_i}$.
\end{itemize}

Taking all of the above into account and carefully canceling out the terms, we obtain
\begin{equation*}
	\sum_{\valu^0} \sum_{\Gamma'}  \frac{1}{\prod_{i=1}^{\ell(\valu)} s_i! \cdot 2^{s_i}} \cdot \prod_{a,b} (2 g(\Gamma'_{a,b})-1)\cdot \prod_{i=1}^{\ell(\valu)} \val^0_i \cdot \frac{(\val_i - 2)!!}{(\val_i - 2s_i)!!} \cdot \Vol(\Gamma').
\end{equation*}
Denote $\u g_i = (g(\Gamma'_{i,1}),\ldots,g(\Gamma'_{i,s_i}))$ and $\u g = (|\u g_1|, \ldots,|\u g_{\ell(\valu)}|)$. Then $\valu^0 = \valu - 4\u g$. Grouping $\Gamma'$ with the same values of $s_a$ and $g(\Gamma'_{a,b})$, we obtain
\begin{multline*}
	\sum_{\u g :\ \u g < \valu/4} \quad  \sum_{\substack{i\;\textrm{s.t.}\\g_i>0}} \quad \sum_{\substack{\u g_{i} = \left(g_{i}^{(1)}, g_{i}^{(2)} \ldots\right) \\ | \u g_{i}| = g_{i},\ g_{i}^{(j)} > 0}}
	\frac{1}{\prod_{i=1}^{\ell(\valu)} \ell(\u g_i)! \cdot 2^{\ell(\u g_i)} } \cdot \prod_{i,j} (2 g_i^{(j)}-1)\cdot \\ \prod_{i=1}^{\ell(\valu)} \val^0_i \cdot \frac{(\val_i - 2)!!}{(\val_i - 2\ell(\u g_i))!!} \cdot \Vol \left(\cQ(\valu^0) \times \prod_{i,j} \cH(2g_i^{(j)}-2) \right),
\end{multline*}
which is equivalent to the formula in the statement of the theorem.
\end{proof}

\section{Applications and conjectures} \label{sec:appl-conj}

\subsection{Distribution of cylinders}
One main motivation to obtain a formula for the volumes of odd strata as a sum over stable graphs is to prove Conjecture 2 of \cite{DGZZlarge} for these strata in the large genus asymptotics, by analogy with the methods developed in \cite{DGZZ-vol} and \cite{DGZZlarge} in the case of principal strata of differentials. We recall this conjecture here with some additions coming from analogies with Theorem 1.4 of \cite{DGZZlarge}. We first begin be precising what we call a random square-tiled surface in $\cQ(\dgu)$. This is very similar to the notion of a random integer (in the prime number theorem) where we consider the asymptotics of the uniform distribution on intervals $[1, N] \subset \Z_+$ for $N\to\infty$. Here equation~\eqref{eq:def:Vol:sq} states that the number of square-tiled surfaces in $\cQ(\dgu)$ made of at most $N$ squares is asymptotically $c_{\dgu}N^d$ as $N\to\infty$. By the results of \cite{DGZZ-meanders}, these asymptotics extend to the subsets of square-tiled surfaces having a decomposition into horizontal cylinders encoded by a given stable graph $\Gamma$ (see~\eqref{eq:Vol:gamma}): the number of such surfaces with at most $N$ squares grows as $c_{\Gamma}N^d$ as $N\to\infty$. This allows us to interpret the frequency $\tfrac{c_{\Gamma}}{c_{\dgu}}$ as the probability that a random square-tiled surface has type $\Gamma$.

The odd strata of quadratic differentials are not all connected~\cite{Lanneau}~\cite{CM_quad}: the strata $\cQ(2(g-k)-3 , 2(g -k)-3, 2k+1 , 2k+1)$ for $-1 \leq  k \leq g-2$ and $g\geq 3$ as well as $\cQ(3^2, -1^2)$ have two connected components, one hyperelliptic, the other not; the only exception is $\cQ(3^4)$ that has precisely two non-hyperelliptic connected components, as for the other exceptional strata in genus 3 and 4 :  $\cQ(9,-1)$, $\cQ(3^3, -1)$ and $\cQ(9,3)$. The strata $\cQ(1,-1)$ and $\cQ(3,1)$ are empty. All other odd strata are nonempty and connected. 

Finally, we consider the regime where $g\to\infty$ for strata $\cQ(\dgu)$ with a bounded number poles $k_i=-1$. In this regime, $g\to\infty$ is equivalent to $d\to\infty$ where $d$ is the dimension of the stratum. 
\begin{conj}[Enhanced Conjecture 2 of \cite{DGZZlarge}]\label{conj:cyl}Let $M\in\Z_+$ be a fixed integer and denote by $\tilde{\dgu}$ the partition $\tilde{\dgu}=(\dgu, -1^p)$ where $p\leq M$ and $\dgu$ is a positive partition of $4g-4+p$. Let $K_{\tilde\dgu}$ be the random variable that gives the number of horizontal cylinders of a random square-tiled surface in $\cQ(\tilde\dgu)$.
\begin{itemize}
\item The probability that all singularities of a random square-tiled surface in $\cQ(\tilde{\dgu})$ are located at the same leaf of the horizontal foliation and at the same leaf of the vertical foliation tends to 1 as $g\to\infty$.
\item There exists a constant $R > 1$ such that the distribution of $K_{\tilde\dgu}$ converges mod-Poisson with parameter $\lambda_d=\log(d)/2$, limiting function $\frac{\sqrt\pi}{\Gamma(t/2)}$ and radius $R$ uniformly for all non-hyperelliptic components of $\cQ(\tilde{\dgu})$ of dimension $d$.
More precisely, let $\cC$ be such a component. Let $p_{\cC}(k)$ denote the probability that a random square-tiled surface in $\cC$ has k cylinders. Then 
\[\sum_{k\geq 1}p_{\cC}(k)t^k= (\dim_{\cC})^{\frac{t-1}{2}}\cdot \frac{\sqrt\pi }{\Gamma(t/2)}\left(1+O\left(\frac{1}{\dim_\C(\cC)}\right)\right),\]
where the error term is uniform over all non-hyperelliptic components of all strata of type $\cQ(\tilde{\dgu})$ and uniform over all $t$ varying in compact subsets of the complex disk $|t|<R$. 
\end{itemize}
\end{conj}

We hope to prove this conjecture following the strategy of ~\cite{DGZZlarge} for principal strata. For that, we first need precise estimates for the coefficients $\langle\tau_{\u d}\rangle_{m_*}$ as $g$ tends to $\infty$. In the case of classical intersection numbers $\langle\tau_{\u d}\rangle$, this was performed by Aggarwal in ~\cite{Agg} by a combinatorial analysis of the recursive relations (Virasoro constraints) characterizing these intersection numbers. Conjecture 3.1 of~\cite{Kon} states that the exponential generating series of the numbers $\langle\tau_{\u d}\rangle_{m_*}$ is a $\tau$-function for the KdV hierarchy when fixing the variables encoding the $m_*$. As far as we know, this conjecture is not proven yet, but recent results of Borot-Wulkenhaar~\cite{BW} show that the generating function for the $\langle\tau_{\u d}\rangle_{m_*}$ for a fixed $m_*\setminus \{m_1\}$ is a $\tau$-function of the BKP hierarchy. One can hope to use these recursions to derive precise asymptotics of the intersection numbers $\langle\tau_{\u d}\rangle_{m_*}$ as $g\to \infty$ for fixed $m_*\setminus \{m_1\}$. This would allow to study the contribution of each graph $\Gamma$ for strata $\cQ(\dgu)$ where we only fix the degree of singularities different from 1, and then get partial proof of the Conjecture in this regime. Note that already here compared to the case of principal strata, there would be an additional difficulty coming from the analysis of the completed terms appearing in the contribution $\Vol(\Gamma)$. The question of uniformity among the strata $\cQ(\dgu)$ of same dimension seems even more delicate, and would require some recursions between the $\langle\tau_{\u d}\rangle_{m_*}$ for different $m_*$, in the spirit of Proposition~\ref{prop:walls-bi-equal-0} or \cite[Conjecture 3.2]{Kon}. We plan to address these question in the future.

A first observation for the comparison of the contributions $\overline\Vol(\Gamma)$ and $\Vol\Gamma$ that is part of the strategy of the proof is the following: for stable graphs $\Gamma$ with only one edge (square-tiled surfaces with one cylinder), we have $\overline\Vol(\Gamma)=\Vol(\Gamma)$ since by Theorem~\ref{thm:counting-funcs-on-walls}, $N_{g,n}^{\valu}(b,b)=F_{g,n}^{\valu}(b,b)$. Then for stable graphs $\Gamma$ with more than one edge, the difference between $\overline\Vol(\Gamma)$ and $\Vol\Gamma$ is only expressed in terms of contributions of stable graphs with less edges of smaller odd strata and contributions of stable graphs for minimal strata $\cH(2g_i-2)$. The latter are completely determined in~\cite{Yakovlev}. Hence, it seems tractable to follow the contributions of each stable graph to the completed volumes versus the usual volumes in large genus, if we know the asymptotics of the corresponding intersection numbers.

\subsection{Volume asymptotics}
The second motivation for this work is to prove the conjecture about the Masur--Veech volume asymptotics in large genus for odd strata, as stated in \cite[Conjecture 1]{ADGZZ}. We reproduce this conjecture here for completeness, adapted to our setting. 

\begin{conj}[Adapted Conjecture 1 of \cite{ADGZZ}]Let $M\in\Z_+$ be a fixed integer. For any $0\leq p\leq M$ and for any odd positive partition $\dgu$ of $4g-4+p$ we have 
\[\Vol \cQ(\dgu, -1^p)=\frac{4}{\pi}\prod_{i=1}^n \frac{2^{\val_i}}{\val_i}((1+\varepsilon(p,\dgu))\]
with \[\lim\limits_{g\to\infty}\max\limits_{\substack{p\leq M\\ \dgu\vdash 4g-4+p}}|\varepsilon(p,\dgu)|=0.\]
\end{conj}

As for the previous conjecture, the strategy to prove the volume asymptotics would rely on an asymptotic formula for the intersection numbers $\langle \tau_{\u d}\rangle_{m_*}$ as $g$ goes to infinity. If such a formula is known, the rest of the strategy would be to compute the asymptotics of the completed volumes $\overline \Vol(\cQ(\dgu))$ following the lines of \cite{Agg} and identifying the stable graphs that contribute the most to the counting (according to Conjecture~\ref{conj:cyl} these should be the stable graphs with only one vertex). Then, Theorem~\ref{thm:coeffs} coupled with the known asymptotics for the volumes of minimal strata $\cH(2g_i-2)$ would allow to derive the asymptotics of the Masur--Veech volumes $\Vol(\cQ(\dgu))$ by induction. There is certainly a lot of additional technicalities  compared to the case of principal strata, but the result, at least in some particular regimes such as taking one fixed singularity of degree $k\geq 3$ and all other of degree one for instance, does not seem out of reach. We plan to work on this question in the future. 

\begin{Remark}
The two conjectures stated above were initially formulated for strata of quadratic differentials with no poles. Here we restrict to the case of odd singularities but we allow a bounded number of poles. By analogy to Theorem 1.7. of \cite{Agg}, this hypothesis could probably be relaxed even more to the case of a number of poles $p$ growing not faster than $\log (g)$ as $g$ goes to infinity. Of course each regime ($p=0$, $p\leq M$, $p\leq c\log(g)$) has an increasing amount of technicalities to overcome to prove the conjectures.
\end{Remark}

\begin{Remark} We focused here on the regime $g\to\infty$, another interesting regime to study is for the number of poles $p$ going to infinity (for a fixed genus or a moderately growing genus). However, even in the case of principal strata with poles, only the volume asymptotics is proven (see \cite[Theorem 1.4]{CMS}).
\end{Remark}

\appendix
\section{Metrics on ribbon graphs}
\label{appendix:metrics}

The proofs of this appendix are analogous to those in \cite[section 3.2]{Yakovlev} for the case of bipartite ribbon graphs.

\subsection{Proof of Lemma \ref{lem:properties-weight-funcs}}

The claims of Lemma \ref{lem:properties-weight-funcs} are proven below in Lemmas \ref{lem:odd-vpG-dimension}, \ref{lem:odd-non-const-lin-funcs}, \ref{lem:odd-vpg-lattice}.

\begin{defi}
    A \emph{one odd cycle (OOC) graph} is a connected graph with exactly one simple cycle of odd length.
\end{defi}

An OOC graph is simply an odd cycle with trees attached to its vertices. Thus an OOC graph on $n$ vertices has exactly $n$ edges. Any OOC graph is clearly non-bipartite. 

The following lemma states that for any choice of vertex perimeters, there exists a unique weight function on an OOC ribbon graph with these vertex perimeters.

\begin{lemma}
\label{lem:odd-weight-func-ooc}
    Let $G$ be an OOC ribbon graph with $n$ labeled vertices. Then the linear map $\operatorname{vp}_G : \mathbb{R}^{E(G)} \rightarrow \mathbb{R}^n$ is an isomorphism. Moreover, $w \in \mathbb{Z}^{E(G)}$ if and only if $\operatorname{vp}_G(w)$ is in the sublattice of $\mathbb{Z}^n$ defined by $b_1+\ldots+b_n = 0 \pmod{2}$.
\end{lemma}
\begin{proof}
    Note that every edge of $G$ is static. Indeed, deletion of any edge produces either a tree, or an OOC graph and a tree. Trees are bipartite.

    Hence, by Lemma \ref{lem:odd-weight-static-edge}, given the vertex perimeters $\u b$, the weight of each edge is uniquely determined. So $\operatorname{vp}_G$ is injective. For the surjectivity, note that assigning to each edge the weight given by Lemma \ref{lem:odd-weight-static-edge} produces the necessary weight function.
    
    $b_1+\ldots+b_n$ is twice the sum of weights of all edges. So if $w$ is integral, this sum is necessarily even. Conversely, if $b_1+\ldots+b_n = 0 \pmod{2}$, the weights of all edges are integral by Lemma \ref{lem:odd-weight-static-edge}.
\end{proof}

\begin{lemma}
\label{lem:odd-vpG-dimension}
    Let $G \in \RGc^{\valu, *}_{g,n}$ be a non-bipartite ribbon graph. Then $\operatorname{Im}(\operatorname{vp}_G) = \mathbb{R}^n$. Moreover, for every $\u b \in \mathbb{R}^n$, $\operatorname{vp}_G^{-1}(\u b)$ is an affine subspace of $\mathbb{R}^{E(G)}$ of dimension $|E(G)| - |V(G)|$.
\end{lemma}
\begin{proof}
    Choose a spanning tree $T$ of $G$. Complete it to an OOC graph $T'$ by adding one edge (such an edge exists since $G$ is non-bipartite). By Lemma \ref{lem:odd-weight-func-ooc} there is a weight function $w'$ on $T'$ such that $\vp_{T'}(w')=\u b$. Extend $w'$ to a weight function $w$ on $G$ by setting $w(e)=0$ if $e \notin E(T')$. Then $\vp_G(w) = \u b$, and so $\operatorname{vp}_G$ is surjective. Consequently, for any $\u b$, the dimension of $\vp_G^{-1}(\u b)$ is $\dim \ker \vp_G = |E(G)|-n = |E(G)| - |V(G)|$.
\end{proof}

\begin{lemma}
\label{lem:odd-non-const-lin-funcs}
    Let $G \in \RGc^{\valu, *}_{g,n}$ be a non-bipartite ribbon graph and let $\u b \in \mathbb{R}^n$. Regard the coordinates $w(e)$ of $\R^{E(G)}$ as linear functions on $\vp_G^{-1}(\u b)$. Then for all $e\in S(G)$, $w(e)$ is constant with value $f_e(\u b)$. All other functions $w(e), e\in E(G) \setminus S(G)$ are non-constant.
\end{lemma}
\begin{proof}
    The first claim follows from Lemma \ref{lem:odd-weight-static-edge}. Let now $e \in E(G) \setminus S(G)$. 
    
    \begin{figure}
        \centering
        \includegraphics[width=0.7\textwidth]{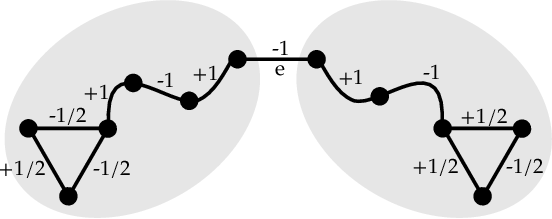}
        \caption{A way to change the weight of a non-static edge $e \in E(G)\setminus S(G)$ without changing the vertex perimeters.}
        \label{fig:odd-proof-non-static}
    \end{figure}

    Let $v_1$ and $v_2$ be the endpoints of $e$ (these might coincide). Since all connected components of $G-e$ are non-bipartite, there are odd cycles $C_1$ and $C_2$ in the connected components of $G-e$ containing $v_1$ and $v_2$, respectively. Let $\gamma_1$ and $\gamma_2$ be paths in $G-e$ connecting $v_1$ to $C_1$ and $v_2$ to $C_2$. Let $\gamma$ be the (non-simple) path in $G$ which is a concatenation (in this order) of $\gamma_1$, $C_1$, $\gamma_1$ in reverse, $e$, $\gamma_2$, $C_2$, $\gamma_2$ in reverse, $e$.

    The path $\gamma$ is a closed path of even length. Change the weights of successively visited edges of $\gamma$ by $+1/2$ and $-1/2$ alternately. This changes the weight of $e$ by $-1$, while preserving all vertex perimeters (see Figure \ref{fig:odd-proof-non-static} for an example). Thus $w(e)$ is not constant on $\vp_G^{-1}(\u b)$.
\end{proof}

\begin{lemma}
\label{lem:odd-vpg-lattice}
    Let $G \in \RGc^{\valu, *}_{g,n}$ be a non-bipartite ribbon graph and $\u b \in \Z^n$. If $b_1+\ldots+b_n = 0 \pmod{2}$, then $\vp_G^{-1}(\u b) \cap \Z^{E(G)}$ is a lattice in $\vp_G^{-1}(\u b)$. Otherwise, $\vp_G^{-1}(\u b) \cap \Z^{E(G)}$ is empty.
\end{lemma}
\begin{proof}
    If $b_1+\ldots+b_n$ is odd, there is clearly no integer weight function on $G$ with vertex perimeters $\u b$, and so $\vp_G^{-1}(\u b) \cap \Z^{E(G)}$ is empty.

    Suppose $b_1+\ldots+b_n$ is even. As in the proof of Lemma \ref{lem:odd-vpG-dimension}, choose a spanning tree $T$ of $G$ and complete it to an OOC graph $T'$ by adding one edge. There is a weight function $w'$ on $T'$ such that $\vp_{T'}(w')=\u b$. By Lemma \ref{lem:odd-weight-func-ooc} it is integral. Extend $w'$ to a weight function $w$ on $G$ by setting $w(e)=0$ for $e \in E(G)\setminus E(T')$. We thus have $\vp_G^{-1}(\u b) \cap \Z^{E(G)} = w + \vp_G^{-1}(0) \cap \Z^{E(G)}$. We now prove that $\vp_G^{-1}(0) \cap \Z^{E(G)}$ is a lattice by providing a basis of size $|E(G)|-n$. 
    
    For each $e \in E(G) \setminus E(T')$ construct a closed path $\gamma_e$ as follows. Let $v_1, v_2$ be the endpoints of $e$, let $\gamma_1$ and $\gamma_2$ be the paths in $T'$ connecting $v_1$ and $v_2$ to the unique odd cycle $C$ of $T'$. Then $\gamma_e$ is a concatenation of $\gamma_1$, $C$, $\gamma_1$ in reverse, $e$, $\gamma_2$, $C$, $\gamma_2$ in reverse, $e$. $\gamma_e$ is a path of even length. Starting from a zero weight function, modify the weight of each edge along the path $\gamma_e$ by $+1/2$ and $-1/2$ alternately. This gives an integral weight function $w^e$ such that $\vp_G(w^e)=0$, $w^e(e)=1$, $w^e(e')=0$ for all $e' \in E(G)\setminus E(T'),\ e' \neq e$.

    The weight functions $w^e,\ e \in E(G) \setminus E(T')$ form a basis of $\vp_G^{-1}(0) \cap \Z^{E(G)}$.
\end{proof}

\subsection{Proof of Proposition \ref{prop:quasi-poly}}

In what follows we will use some terminology coming from the theory of polyhedra. We refer the reader to \cite{Barvinok} for details.

For a polyhedron $P$ and a point $p$ in $\mathbb{R}^d$ the \emph{cone of feasible directions to $P$ at $p$} is defined as 
    \[\operatorname{fcone}(P,p) = \{v \in \mathbb{R}^d: p+\varepsilon v \in P \text{ for some } \varepsilon>0\}.\]
For example, if $P \subset \R^2$ is a two-dimensional convex polygon and $p \in \R^2$, the $\operatorname{fcone}(P,p)$ is: $\R^2$ if $p$ is an interior point of $P$; half-space if $p$ lies in the interior of a side of $P$; an acute cone if $p$ is a vertex of $P$; empty set if $p$ does not belong to $P$.

\begin{Theorem}[Theorem 18.4 in \cite{Barvinok}]
    \label{thm:odd-barvinok}
    Let $\{P_{\alpha} : \alpha \in A\}$ be a family of $d$-dimensional polytopes in $\mathbb{R}^d$ with vertices $v_1(\alpha),\ldots, v_n(\alpha)$ such that $v_i(\alpha) \in \mathbb{Q}^d$ and the cones of feasible directions at $v_i(\alpha)$ do not depend on $\alpha$:
    \[\operatorname{fcone}(P_{\alpha},v_i(\alpha)) = \operatorname{const}_i,\ i=1,\ldots, n. \]
    Suppose also that there are vectors $u_1,\ldots,u_n \in \Q^d$ such that
    \[v_i(\alpha) - u_i \in \Z^d,\ \alpha \in A,\ i=1,\ldots, n.\]
    Then there exists a polynomial $p: (\mathbb{R}^d)^n \rightarrow \mathbb{R}$ such that 
    \[|\operatorname{int} P_{\alpha} \cap \mathbb{Z}^d| = p(v_1(\alpha),\ldots, v_n(\alpha)).\]
\end{Theorem}

\begin{Theorem}[Chapter 9 in \cite{Barvinok}]
    \label{thm:odd-barvinok-volumes}
    Let $\{P_{\alpha} : \alpha \in A\}$ be a family of $d$-dimensional polytopes in $\mathbb{R}^d$ with vertices $v_1(\alpha),\ldots, v_n(\alpha)$ such that the cones of feasible directions at $v_i(\alpha)$ do not depend on $\alpha$:
    \[\operatorname{fcone}(P_{\alpha},v_i(\alpha)) = \operatorname{const}_i,\ i=1,\ldots, n. \]
    Then there exists a polynomial $p: (\mathbb{R}^d)^n \rightarrow \mathbb{R}$ such that 
    \[\Vol P_{\alpha} = p(v_1(\alpha),\ldots, v_n(\alpha)).\]
\end{Theorem}

Let us now identify the vertices and the corresponding cones of feasible directions of the polytopes $P_G(\u b)$.

\begin{lemma}
\label{lem:odd-vertex_characterisation}
    Let $G \in \RGc^{\valu,*}_{g,n}$ be a non-bipartite ribbon graph and let $\u b \in \R^n$. Then $w \in \vp_G^{-1}(\u b)$ is a vertex of $P_G(\u b)$ if and only if $w(e) \ge 0$ for all $e \in E(G)\setminus S(G)$ and the edges in $F = \{e \in E(G)\setminus S(G) : w(e) > 0\}$ form a disjoint union of trees and/or OOC graphs.
    
    The cone of feasible directions to $P_G(\u b)$ at such vertex $w$ is given by the system
        \[
        \begin{cases}
        v \in \operatorname{vp}_G^{-1}(0),\\
        v(e) \ge 0,\ e \in E(G) \setminus (S(G) \cup F).
        \end{cases}
        \]
    In particular, it only depends on $F$ and not on $\u b$.
\end{lemma}
\begin{proof}
    Recall that the vertices of a polytope are exactly its extreme points. We claim that a point $w \in \vp_G^{-1}(\u b)$ is an extreme point of $P_G(\u b)$ if and only if the corresponding set $F$ does not contain an even cycle, which is equivalent to being a disjoint union of trees and/or OOC graphs. Indeed, if there is an even cycle in $F$, then one can modify the weights in this cycle by alternately adding and subtracting $\varepsilon$ or $-\varepsilon$ to/from the weights of consecutive edges of this cycle, for some small $\varepsilon>0$. Then $w$ is a midpoint of these two modifications, both of which still belong to $P_G(\u b)$. Hence $w$ is not an extreme point. Conversely, suppose $F$ is disjoint union of trees and/or OOC graphs and $w=(w' + w'')/2$ with $w', w'' \in P_G(\u b)$. Since $w(e)=0, w'(e) \ge 0, w''(e) \ge 0$ for $e \in E(G) \setminus (S(G) \cup F)$, necessarily $w'(e) = w''(e) = 0$ for $e \in E(G) \setminus (S(G) \cup F)$. But then, by Lemma~\ref{lem:odd-weight-static-edge}, the weights $w'(e), w''(e), e \in S(G) \cup F$ of $w'$ and $w''$ are uniquely determined and are equal to the corresponding weights of $w$. Hence $w'=w''=w$ and $w$ is an extreme point of $P_G(\u b)$, hence a vertex. 

    The second claim follows directly from the definitions of $P_G(\u b)$ and of the cone of feasible directions.
\end{proof}

For a vertex $w$ of $P_G(\u b)$ we call the set $F \subset E(G) \setminus S(G)$ as in Lemma~\ref{lem:odd-vertex_characterisation} the \emph{support} of $w$.
    
\begin{lemma}
    \label{lem:odd-polytope_dependence}
    Fix a non-bipartite ribbon graph $G \in \RGc^{\valu,*}_{g,n}$ and a cell $C$.
    
    There exist subsets $F_1,\ldots, F_m \subset E(G) \setminus S(G)$ each forming a disjoint union of trees and/or OOC graphs, such that each polytope $P_G(\u b)$ with $\u b \in C$ has $m$ vertices $v_1(\u b), \ldots, v_m(\u b)$ with supports $F_1,\ldots, F_m$ respectively. 
    
    For each $i$ the coordinates of $v_i(\u b)$ are either identically zero or are linear functions (of $\u b$) of the form (\ref{eq:prelim-static-edge-weight-1}) or (\ref{eq:prelim-static-edge-weight-2}).
    
    For each $i$ the cone of feasible directions $\operatorname{fcone}(P_G(\u b), v_i(\u b))$ is constant (does not depend on $\u b$).
\end{lemma}
\begin{proof}    
    By Lemma \ref{lem:odd-vertex_characterisation}, a subset $F \subset E(G)\setminus S(G)$ forming a disjoint union of trees and/or OOC graphs is a support of some vertex of $P_G(\u b)$ if and only if for the unique weight function $w_F$ on $F \cup S(G)$ with vertex perimeters $\u b$ we have $w_F(e)>0$ for all $e\in F$. By Lemma \ref{lem:odd-weight-static-edge}, the weights of edges in $w_F$ are given by linear forms in $\u b$ of the form (\ref{eq:prelim-static-edge-weight-1}) or (\ref{eq:prelim-static-edge-weight-2}). When $\u b$ stays inside $C$, the signs ($+,-$ or $0$) of all these linear functions remain constant. Hence an $F$ corresponds to a vertex of $P_G(\u b)$ either for all $\u b \in C$ or for none.
\end{proof}

\begin{proof}[Proof of Proposition \ref{prop:quasi-poly}]
    Fix a non-bipartite ribbon graph $G \in \RGc^{\valu,*}_{g,n}$, a cell $C$ and a coset $\Lambda$ of $2\Z^n$ in $\Z^n$.

    If for $\u b \in \Lambda$ we have $b_1+\ldots+b_n = 1 \pmod{2}$, then $|\operatorname{int} P_G(\u b) \cap \Z^{E(G)}|$ is zero identically. 
    
    Otherwise, by Lemma \ref{lem:odd-vpg-lattice}, $\vp_G^{-1}(\u b) \cap \Z^{E(G)}$ is a lattice in $\vp_G^{-1}(\u b)$. Consider the family of polytopes $\{P_G(\u b), \u b \in C \cap \Lambda\}$. By Lemma \ref{lem:odd-polytope_dependence}, all these polytopes have the same cones of feasible directions at corresponding vertices. Moreover, the corresponding coordinates of these vertices are either integer or half-integer for all $\u b \in C \cap \Lambda$ (because, for $\u b$ in a fixed coset of $2\Z^n$ in $\Z^n$, the value of any linear function of the form (\ref{eq:prelim-static-edge-weight-1}) or (\ref{eq:prelim-static-edge-weight-2}) is either always integer or always half-integer). One can thus apply Theorem \ref{thm:odd-barvinok} to this family of polytopes.
\end{proof}

\section{Di Francesco-Itzykson-Zuber formula}\label{app:DFIZ}
We reproduce here the formula of \cite{DFIZ} relating the intersection numbers $\langle  \tau_{\u d}\rangle_{m_*}$ with the classical ones $\langle  \tau_{\u d}\rangle$ via their generating functions. Writing $t_*=(t_0, t_1, \dots)$, $s_*=(s_0, s_1, \dots)$ for the infinite sequences of variables and $m_*=(m_0, m_1, \dots)$ and $n_*=(n_0, n_1, \dots)$ for the infinite sequences of non-negative integers, almost all zero, we define 
\[F(t_*,s_*)=\sum_{n_*, m_*}\langle  \tau_*^{n_*}\rangle_{m_*}\frac{t_*^{n_*}}{n_*!}s_*^{m_*},\] where $n_*!=\prod_{i=0}^\infty n_i!$ and $t_*^{n_*}=\prod_{i=0}^\infty t_i^{n_i}$ and similarly for $s_*$ and $\tau_*$. Fixing non-negative integers $m_0, m_2, m_3, \dots$ almost all equal to zero, and setting $\hat s_*=(0,1,0,0,\dots)$ we get
\[\prod_{i\neq 1}\frac{1}{m_i!}\left(\frac{\partial }{\partial s_i}\right)^{m_i}\restriction{F(t_*,s_*)}{s_*=\hat s_*}=\sum_{n_*, m_1}\langle\tau_*^{n_*}\rangle_{m_*}\frac{t_*^{n_*}}{n_*!}\] and fixing non-negative integers $\mu_0, \mu_1, \mu_2, \dots$ almost all equal to zero, \[\prod_{i}\left(\frac{\partial }{\partial t_i}\right)^{\mu_i}\restriction{F(t_*,s_*)}{s_*=\hat s_*}=\sum_{n_*}\langle\prod\tau_i^{\mu_i}\tau_*^{n_*}\rangle\frac{t_*^{n_*}}{n_*!}.\] The formula of \cite{DFIZ} relates the $t$-derivatives of the exponential generating series $Z=\exp(F)$ with its $s$-derivatives at $\hat s_*$, giving implicitly a relation between the intersection numbers $\langle\tau_{\u d}\rangle_{m_*}$ and $\langle\tau_{\u d}\rangle$. An algorithm to get this formula is described in \cite{Bini}. We give here a simple way to reproduce the formulas of \cite[Appendix]{AC} that we programmed to compute more intersection numbers.

Let us set up some notations, that are mostly borrowed from \cite{Bini}. 
\begin{itemize}
    \item Let  $S(n)$ be the set of strict partitions $\lambda$ (with $\lambda_1>\lambda_2>\dots > \lambda_k$) of $n$.
    \item Let $O(n)$ be the set of odd partitions $\nu$ of $n$ (all parts of $\nu$ are odd). These partitions are usually denoted by $\nu=[(2k+1)^{\nu_k}]$.
    \item If $\sigma\in S(n)$ and $\nu\in O(n)$, $\langle \sigma \rangle (\nu)$ denotes the value of the spin character associated to $\sigma$ on the element of type $\nu$ (see \cite{Morris}). These characters can be efficiently computed using an equivalent of Murnaghan-Nakayama rule developed in \cite{MO}. They also correspond to the coefficients of Schur Q-functions.
    \item For a partition $\lambda$ we denote by $\ell(\lambda)$ its length (number of parts), $|\lambda|=\lambda_1+\dots +\lambda_k$ its weight, and $\varepsilon(\lambda)=\begin{cases} 1 \mbox{ if }  |\lambda|-\ell(\lambda) \mbox{ is odd},\\ 0 \mbox{ else}. \end{cases}$
    \item For a fixed integer $n$, we consider the matrix $B$ with lines indexed by strict partitions $\sigma\in S(n)$ and columns indexed by odd partitions $\nu\in O(n)$ (note that $|S(n)|=|O(n)|$). Its coefficients are defined by \[b_{\sigma \nu}=2^{\frac{\ell(\sigma)-\ell(\nu)+\varepsilon(\sigma)}{2}}\langle\sigma\rangle(\nu)\] (note that this notation differs from Bini's one). The coefficients of $B^{-1}$ are denoted by $b^{\nu \sigma}$.
    \item If $\sigma\in S(n)$, a reductive sequence $s=(s_1, \dots, s_{\ell(\sigma)})$ is a $\ell(\sigma)$-tuple of non-negative integers such that for all $1\leq i\leq \ell( \sigma),\ 3s_i\leq \sum_{i\leq k\leq \ell(\sigma)}\sigma_k$, and such that $\sigma-3s$ can be permuted into a sequence of the form $\lambda, -r_1, r_1, \dots, -r_k, r_k, 0, \dots, 0$, with $\lambda$ a strict partition.
    \item We define a function $L$ on tuples of integers recursively as follows:
     If $v=(v_1, \dots, v_k)\in \mathbb Z^k$ is a $k$-tuple of integers, then
    \begin{itemize}
        \item $L(v)=1$ if $\forall i, v_i>0$
        \item $L(v_1, \dots, v_{k-1},0)=L(v_1, \dots, v_{k-1})$
        \item $L(v_1, \dots, v_{k-1},a)=0$ if $a<0$
        \item $L(v_1,\dots, v_{j+1}, v_j, \dots, v_k)=-L(v_1, \dots, v_j, v_{j+1}, \dots, v_k)+2(-1)^{v_j}\delta_{v_j, -v_{j+1}}L(v_1\dots v_{j-1}, v_{v+2}, \dots, v_k)$
    \end{itemize}
    \item For $s\in\mathbb N$ and $m\in\mathbb N^*$ define $c_{s,m}=\sum_{i=0}^{2s}\frac{1}{2^i} {m-3s+i-1\choose  i}\frac{(6s-2i-1)!!}{6^{2s-i}(2s-i)!}.$
\end{itemize}

The formula of \cite[Proposition W, Table IV]{DFIZ} (see also \cite[Appendix]{AC}) reads as follows :

\begin{multline*}\prod_k \frac{1}{\nu_k}\left(-(2k-1)!!\frac{\partial }{\partial t_k}\right)^{\nu_k}\restriction{Z}{|s=s^*}\\=\sum_{\sigma\in S(|\nu|)}b^{\nu\sigma}\sum_{\substack{\textrm{reductions}\\s\textrm{ of }\sigma}}L(\sigma-3s)\prod_{i}(-1)^{s_i}c_{s_i, \sigma_i}\sum_{\mu\in O(|\lambda|)}b_{\lambda\mu}\prod_l\frac{1}{\mu_l!}\left(-\frac{1}{2^l}\frac{\partial}{\partial s_l}\right)^{\mu_l}\restriction{Z}{|s=s^*}.\end{multline*}

The $s$-derivatives of $Z$ are then obtained in terms of the $t$-derivatives by inverting these linear equations.
\section{Relation to Hurwitz counting, shifted symmetric functions and completed cycles}\label{app:hurwitz} 

In this section we explain what Theorem~\ref{thm:counting-funcs-on-walls} implies in terms of characters of the symmetric group, by reviewing some of the results of \cite{GM}. Indeed, one other way to count metrics on ribbon graphs is to count ramified covers of the sphere $\Proj^1$ with specific ramification profile at the three points $0,1$ and $\infty$. This construction is standard (see for instance \cite[Construction 1.3.22]{LandoZvonkin}) and we specify it to our case in the following lemma.

\begin{lemma}\label{lem:cor} For $(b_1, \dots b_n)\in \Z_+^n$ such that $\sum_i b_i=0\mod 2$, the set of integer metric ribbon graphs in $\cM_{m_*,n}$ with face lengths $(b_1, \dots b_n)$  is in one-to-one correspondence with the set $\Cov(\valu, \u b)$ of ramified covers of the sphere  $\Proj_1$ with ramification profile $\u b=(b_1, \dots, b_n)$ over $\infty$, $\tilde\valu=(\valu, 2, \dots, 2)$ over 0 and $\u t=(2,\dots 2)$ over 1 (we should have $|\u b|=|\tilde\valu|=|\u t|$). More precisely, disconnected covers correspond to disconnected ribbon graphs with circles allowed, disconnected covers without unramified components correspond to disconnected ribbon graphs, and connected covers correspond to connected ribbon graphs. 
\end{lemma}
\begin{proof} 
Starting from a ramified covering of $\Proj_1$, consider the lift of the ribbon graph on $\Proj_1$ consisting  of one segment joining $0$ and $1$ (see Figure~\ref{fig:revet}): it is a graph embedded in the covering, so a ribbon graph, with vertices corresponding to preimages of $0$ and $1$ and valencies given by the ramification profiles at these points. Declaring that the length of the edge $[0,1]$ is $1/2$ and forgetting about all vertices of valency 2 (i.e. all preimages of $1$ and some preimages of $0$), we obtain an integer metric on a ribbon graph with vertices of valency given by $\valu$. Note that the genus of the ribbon graph is the same as the genus of the covering by Euler-Poincaré formula, and that the lengths of the faces are exactly $b_1, \dots, b_n$. The reverse construction is direct after embedding the ribbon graph in a corresponding topological surface (i.e. by gluing topological disks on the faces).

Note that another way to visualize this construction, which is relevant to our count of square-tiled surfaces, is to add the flat structure of a ``pocket" on the sphere, i.e. to consider a quadratic differential with simple poles at $0$ and $1$, and a double pole at $\infty$. One can further ask that the relative period $[0,1]$ is real equal to $1/2$. The lift of this differential gives a Jenkins-Strebel differential in the stratum $\cQ(\dgu, -2^n)$ where $\dgu=\valu-2$ (see Convention~\ref{convmk}) after forgetting all marked points, such that the residues at the double poles are $-1/4\pi^2(b_1^2, \dots b_n^2)$. The critical graph of this Jenkins-Strebel differential is the ribbon graph that we have described above. Cutting the infinite cylinders along their waist curves, we obtain a piece of a square-tiled surface that we are counting.
\end{proof}

\begin{figure}[h]
    \centering
    \def\svgwidth{0.4\textwidth}
    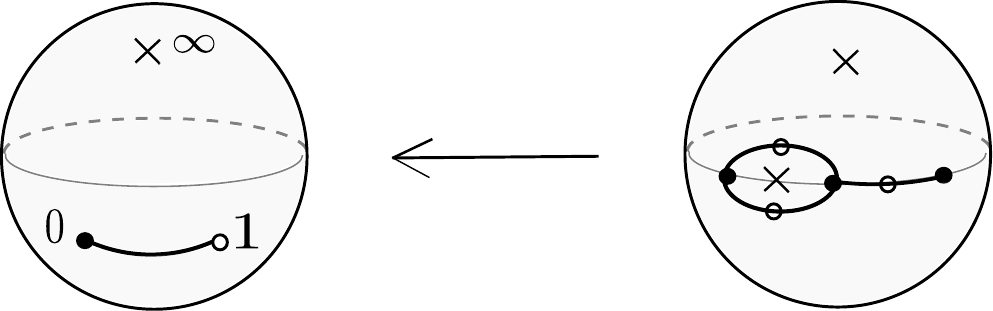       \hspace{1cm} \def\svgwidth{0.5\textwidth}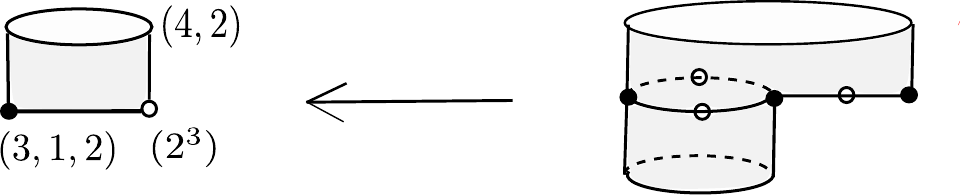
    \caption{\emph{Left}: Ribbon graph obtained by lifting the segment $[0,1]\in\Proj_1$ in the ramified cover. In this example the covering is ramified over $0$ with profile $(3,1,2)$, over $1$ with profile $(2,2,2)$, and over $\infty$ with profile $(4,2)$. The graph has two vertices of valency $3$ and $1$, after forgetting the marked points of valency 2.\\
    \emph{Right}: Same example with an additional flat structure: the base is a ``pocket" (flat surface in $\cQ(-1^2, -2)$ with a half-infinite cylinder), and the lift gives a surface in $\cQ(-1,1,-2^2)$ (after forgetting the marked points) with 2 half-infinite cylinders.}
    \label{fig:revet}
\end{figure}

With this correspondence, we can now derive explicit formulas for the count of metrics on ribbon graphs in terms of characters of the symmetric group, using Hurwitz theory. The next lemma is a version of Frobenius formula (see for instance \cite[Appendix]{LandoZvonkin}) adapted to our context (see \cite[Section 4]{GM}). We first introduce some notations and background on the space of shifted symmetric functions. 

For $\u \mu$ and $\u\lambda$ two partitions such that
 $|\u\mu|\leq |\u\lambda|$, 
 we denote by $f_{\u\mu}(\u\lambda) = |\u\mu|!\,\chi^{\u\lambda}(\u\mu)/(\mathfrak z({\u\mu})\dim \chi^{\u\lambda})$
  the normalized character, where    $\mathfrak z(\u\mu) =  \prod_{m=1}^{\infty} m^{r_m(\u\mu)} \prod_{m=1}^\infty r_m(\u\mu)! =  \prod_{i=1}^{\ell(\u\mu)} \u\mu_i \prod_{m=1}^\infty r_m(\u\mu)!$ denotes the order of the centralizer of the partition  $\u\mu = (1^{r_1},2^{r_2},3^{r_3},\cdots) $.
We also write $f_\ell$ for the special case that $\u\mu$ is
a $\ell$-cycle. The {\em algebra of shifted symmetric polynomials} is
defined as  $\Lambda^* = \varprojlim \Lambda^*(n)$, where $\Lambda^*(n)$
is the algebra of symmetric polynomials in the $n$ variables
\hbox{$\lambda_1-1$,\,}$\ldots,\lambda_n-n$. The functions 
\[
p_\ell(\u\lambda) = \sum_{i=1}^\infty
\left( (\lambda_i -i +\tfrac{1}{2})^\ell - 
(-i + \tfrac{1}{2})^\ell \right)  + (1-2^{-\ell})\,\zeta(-\ell) \quad\text{and}\quad
{p}_{\u\mu} =\prod_i {p}_{\mu_i}.
\] belong to $\Lambda^*$.
Here $\ell! \beta_{\ell+1} = (1-2^{-\ell})\zeta(-\ell)$ with~$\beta_k$
defined by $B(z) \,:=\, \frac{z/2}{\sinh(z/2)} = \sum_{k=0}^\infty \beta_k\,z^k$.
By a result of Kerov-Olshanski, the algebra $\Lambda^*$ is freely generated by all the $p_\ell$, and as $\u\mu$ ranges over all partitions, the functions $f_{\u\mu}$ form
a basis of~$\Lambda^*$.

The algebra $\Lambda^*$ is enlarged to the algebra 
\[
\overline{\Lambda} = \Q[p_\ell, \overline p_k (k,\ell > 0)]
\]
of {\em shifted symmetric quasi-polynomials}, where 
\[
\overline p_k(\lambda) = \sum_{i \geq 0} \Bigl((-1)^{\lambda_i-i+1}
(\lambda_i -i + \tfrac12)^\ell
- (-1)^{-i+1}(-i + \tfrac12)^\ell \Bigr) \,+\, \gamma_k\,,
\]
and where the constants $\gamma_i$ are zero for $i$~odd, $\gamma_0 = 1/2$, 
$\gamma_2 = -1/8$, $\gamma_4 = 5/32$ 
and in general defined by the expansion
$ C(z) = 1/(e^{z/2} + e^{-z/2}) = \sum_{k \geq 0}  \gamma_k z^k/k!$.
We provide the algebra $\overline{\Lambda}$ with a grading by defining 
the generators to have weight given by
\[\we(p_k) = k + 1 \quad \text{and} \quad \we(\overline p_k) = k\,. \]
\par
By \cite[Theorem~2]{EO_pillow}, for an odd partition $\valu$ there exists a function $g_{\valu}\in\overline{\Lambda}$ such that
\[
g_\valu(\u\lambda) = \frac{f_{(\valu,2,2,\ldots)}(\u\lambda)}
{f_{(2,2,\ldots)}(\u\lambda)}\, \quad \text{for $\u\lambda$ such that }\overline p_0(\u\lambda)=\frac{1}{2},
\] and by \cite[Theorem 3.3]{GM}, the elements $g_\valu$ generate $\overline{\Lambda}$ as a graded $\Lambda^*$-module, taking $\we(g_{\valu})=|\valu|/2$. Note however that the elements $g_\valu$ do not form a basis as
there are many more $g_\valu$ than products of $\overline p_k$ for a given weight. 
We define also 
\[\sqrt{w(\u\lambda)}=\frac{\dim(\u\lambda)}{|\u\lambda|!}f_{(2,2,\dots)}(\u\lambda),\]
the square root being here to ensure consistency with \cite{EO_pillow} and \cite{GM}.

\begin{lemma}[Frobenius formula]\label{lem:Hurw} The number $C(\valu, \u b)$ of possibly disconnected covers of $\Proj_1$ with ramification profile $\u b$ over $\infty$, $\tilde\valu=(\valu, 2, \dots, 2)$ over 0 and $\u t=(2,\dots 2)$ over 1 is given by:
\[C(\valu, \u b):=\sum_{\pi\in\Cov(\valu, \u b)}\frac{1}{|\Aut(\pi)|}= A_2(\u b, g_{\valu})\mbox{ with } A_2(\u b, F)= \frac{1}{\prod b_i}\sum_{|\lambda|=|\u b|}\sqrt{w(\lambda)}F(\lambda).\]
\end{lemma}

Note that the number $C(\valu, \u b)$ is called {\em simple Hurwitz numbers with 2-stabilization} in \cite{GM}, and we use here the notation $A_2$ from the same paper. By a classical inclusion-exclusion procedure, one can define similarly the count of covers without ramified components $C'(\valu, \u b)=A_2'(\u b, g_{\valu})$, and the count of connected covers $C^0(\valu, \u b)=A_2^0(\u b, g_{\valu})$ (see \cite{GM}). All these numbers are known to be piecewise quasi-polynomials in the variables $\u b$ for a fixed $\valu$. Lemma~\ref{lem:cor} implies 

\begin{equation}\label{eq:FC}
F^{\valu}_{g,n}(\u b)=C^0(\valu, \u b)+ \text{lower order terms.}
\end{equation}

Now we are ready to recall the main result of \cite{GM} stating the polynomiality of these Hurwitz numbers.

\begin{Theorem}[Theorem 7.2. of \cite{GM}]\label{thm:pbarpoly}
The simple Hurwitz number with $2$-stabilization $A_2'(\u b,F)$ without unramified components is a quasi-polynomial if~$F$  is a product of $\overline p_k$,
i.e.\ for each coset $\u m = (m_1,\ldots,m_n) \in \{0,1\}^n$ with $\sum m_i$
even there exists a polynomial $R_{F,\u m} \in \Q[b_1,\ldots,b_n]$ of degree $\we(F)-n$ such that
\[
A_2'(\u b,F) = R_{F,\u m}(\u b) \quad \text{for all}
\quad \u b \in 2\N^n + \u m\,.
\]
\end{Theorem}
Note that by inclusion-exclusion the same statement remains true for $A^0_2(\u b, F)$.

\begin{Remark}The proof of this theorem is based on a careful analysis of an explicit formula derived from a vertex operator expression for $A_2$. By analyzing this expression further, one can certainly show that the higher degree terms of $A_2'(\u b, F)$ for $F$ a product of $\overline p_k$ constitute a polynomial on the admissible cosets, and not only a quasipolynomial, meaning that the higher degree terms of $R_{F,\u m}(\u b)$ do not depend on $\u m$. However, we don't need such a result here, as we will see in the following.
\end{Remark}

Now, for each $\valu$, writing $g_{\valu}$ as a polynomial in $p_l$ and $\overline p_k$, regrouping the terms containing the variables $\overline p_k$ only, and using that $(f_{\u\mu})$ form a basis of $\Lambda^*$, we obtain a unique decomposition of $g_{\valu}$ as 
\begin{equation}\label{eq:compg}g_{\valu}=\overline g_{\valu} + \sum_{\mu\neq \emptyset}g_{\valu}^{deg, \u\mu}f_{\u\mu}, 
\end{equation}
where $\overline g_{\valu}$ and the $g_{\valu}^{deg, \u\mu}$ are polynomials in $\overline p_k$, and for each $\u\mu$, $\we(g_{\valu}^{deg, \u\mu})<\we(\overline g_{\valu})$. 
 Note that in the sum, some $g_{\valu}^{deg, \u\mu}$ might be equal to zero, and in particular all the $g_{\valu}^{deg, \u\mu}$ such that $\we(f_{\u\mu})=|\u\mu|+\ell(\u\mu)>\we(g_{\valu})$. This implies for the counting functions the following:
\begin{equation}\label{eq:compH}A_2'(\u b, g_{\valu})=A_2'(\u b, \overline g_{\valu}) + \sum_{\u\mu\neq \emptyset}A_2'(\u b, g_{\valu}^{deg, \u\mu}f_{\u\mu}).
 \end{equation}
By Theorem~\ref{thm:pbarpoly} the term $A_2'(\u b, \overline g_{\valu})$ is a quasipolynomial in the $b_i$'s. Furthermore, we have:
\begin{lemma} For each $F$ equal to a product of $\overline p_k$, and for each partition $\u\mu\neq \emptyset$, $A_2'(\u b, F\cdot f_{\u\mu})$ is zero outside an hyperplane. \end{lemma}
\begin{proof}As the $g_{\valu}$ generate $\overline\Lambda$ as a $\Lambda^*$-module, such an expression $F$ can be written as a linear combination of terms $g_{\valu}f_{\u\mu'}$ and $g_{\valu'}$. Hence $A_2'(\u b, F\cdot f_{\u\mu})$ is a linear combination of terms $A_2'(\u b, g_{\valu}f_{\u\mu'}\cdot f_{\u\mu})$ or $A_2'(\u b, g_{\valu}\cdot f_{\u\mu})$ that actually count ramified coverings of the half pillow, where one can declare that the ramification point of profile $(\u\mu, 1, \dots, 1)$ is actually a point at a positive height on the half pillow (so distinct from the corner). Then we can apply the same slicing procedure as in \cite{GM} and write this number as graph sums (\cite[Prop 4.1]{GM}) involving Hurwitz numbers $A'(\u b_i^\bl, \u b_i^\wh, f_\mu)$ as in \eqref{eq:hurwitzA}, that are zero outside an hyperplane.
\end{proof}

The polynomiality features associated to the decomposition~\eqref{eq:compg} and their similarity with the ones associated to the decomposition \[f_{\u\mu}=\frac{1}{\prod \mu_i}p_{\u\mu}+ \mbox{linear combination of products of more than one }p_{\u\mu'}\]
(see \cite{OP}) lead to interpret $\overline g_{\valu}$ as a {\em completed cycle} following~\cite{OP} (or rather a completed conjugacy class).

Considering only the higher degree terms in the polynomials involved (i.e. the highest weight terms in the  elements of $\overline{\Lambda}$ involved), we recognize the situation that we have for $F^{\valu}_{g,n}$ and $N^{\valu}_{g,n}$. In fact, the only difference is in the setting in which we count the metrics: on one side it is on possibly disconnected ribbon graphs, on the other side, on connected ribbon graphs. To conclude on the interpretation of Theorem~\ref{thm:counting-funcs-on-walls} in this setting we then introduce further notations. We denote by $\tilde F^{\valu}_{g,n}$ as in Definition~\ref{def:counting_fct} by replacing everywhere ``ribbon graph'' by ``possibly disconnected ribbon graph''. Actually $\tilde F^{\valu}_{g,n}$ can be computed as a polynomial in $F^{\valu'}_{g',n'}$. 
 We define $\tilde N^{\valu}_{g,n}$ as the same formula in $N^{\valu'}_{g',n'}$. By Proposition~\ref{prop:Kon} it is the unique polynomial that coincides with the higher order terms of $\tilde F^{\valu}_{g,n}$ outside the walls. Comparing with \eqref{eq:compH}, we obtain $\tilde F^{\valu}_{g,n}(\u b)=A_2'(\u b, g_{\valu})$, and $\tilde N^{\valu}_{g,n}(\u b)$ coincides with the higher degree terms of $A_2'(\u b, \overline g_{\valu})$.

On the other hand, by similar classical results (see \cite{EOab} or \cite{GMab}), for the count of metrics on face bicolored ribbon graphs (with only one vertex) introduced in Proposition~\ref{prop:face-bicolored}, we have 
\begin{equation} F_{g,(n^\bl, n^\wh)}^{[2k]}(\u b^\bl, \u b^\wh)=A'(\u b^\bl, \u b^\wh, f_{2k}), \mbox{ where } A(\u b^\bl, \u b^\wh, F)=\frac{1}{\prod b_i^\bl\prod b_i^\wh}\sum_{|\u\lambda|=|\u b^\bl|}\chi^{\u\lambda}(\u b^\bl)\chi^{\u\lambda}(\u b^\wh)F(\u\lambda)\label{eq:hurwitzA}
\end{equation}
and $A'$ and $A$ are related by the usual inclusion-exclusion (see \cite[Lemma 2.4]{EOab} for instance). Note that here $A'$ coincides with $A^0$ since there is only one ramification point.

In conclusion, using Theorem~\ref{thm:counting-funcs-on-walls}, we can obtain an explicit version of \eqref{eq:compH} as follows:

\begin{Corollary}For any odd partition $\valu$, for all $\u b\in \Z^n$ such that $\textstyle \sum b_i$ is even, we have
\begin{equation}\label{eq:comp2}A'_2(\u b, \overline g_{\valu})=A'_2(\u b, g_{\valu})+\sum_{\u g, \u A, \u\varepsilon, a} \tilde C_{\u g, \u A; \u\varepsilon, a}\cdot A'_2(\u b_{0}, g_{\valu'})\prod_i A'(\u b_i^{\bl}, \u b_i^{\wh}, f_{2k_i}),
\end{equation} where the sum runs on ${\u g, \u A; \u\varepsilon, a}$ as in \eqref{eq:V-recursion}, the coefficients $\tilde C_{\u g, \u A; \u\varepsilon, a}$ are explicitly obtained from the coefficients $C_{\u g, \u A; \u\varepsilon, a}$ by applying inclusion-exclusion to the formula \eqref{eq:V-recursion} of Theorem~\ref{thm:counting-funcs-on-walls}, and $2k_i=4g_i-2+2n_i^\bl+2n_i^\wh$.
\end{Corollary}

\begin{Remark}At the earlier stage of the project, one hope was to find an explicit relation between $\overline g_{\valu}$ and $g_{\valu}$ directly by studying the shifted symmetric functions and related vertex operator expressions, and to derive from that the completion coefficients for the Hurwitz numbers appearing in the previous Corollary for instance. One issue here is that, contrary to the $f_{\u\mu}$, the $g_{\valu}$ do not form a basis of $\overline{\Lambda}$ over $\Lambda^*$. This might be studied in future work: Philip Engel communicated to the authors several results in that direction. 
\end{Remark}

\section{Tables}\label{app:tables}
\noindent

\begin{table}[h]
\renewcommand{\arraystretch}{1.2}
\[
\begin{array}{|c|c|c|c|c|c|c|c|c|c|c|} \hline
d & g & \mathrm{Stratum}        & \Vol                  & \cVol & &
d & g & \mathrm{Stratum}        & \Vol                  & \cVol \\ \hline

4 & 1 & \cQ(3,-1^3)     & {5}/{9}           & {2}/{3} & &
8 & 3 & \cQ(3^3,-1)     & {4499}/{68040}    & 179/2520 \\
6 & 0 & \cQ(3,-1^7)     & {3}/{4}           & {3}/{4} & &
8 & 3 & \cQ(5,1^3)      & {49}/{1080}       & {71}/{1512} \\
6 & 1 & \cQ(3,1,-1^4)   & {1}/{3}           & {7}/{20} & &
8 & 3 & \cQ(5,3,1,-1)   & {17}/{216}       & 1/12 \\
6 & 1 & \cQ(5,-1^5)     & {7}/{10}          & 3/4 & &
8 & 3 & \cQ(5^2,-1^2)   & {421}/{2520}      & 5/28 \\
6 & 2 & \cQ(3,1^2,-1)   & {1}/{9}           & {7}/{60} & &
8 & 3 & \cQ(7,1^2,-1)   & {143}/{1400}      & 77/720 \\
6 & 2 & \cQ(3^2,-1^2,)  & {53}/{270}        & {13}/{60} & &
8 & 3 & \cQ(7,3,-1^2)   & {51}/{280}        & 211/1080 \\
6 & 2 & \cQ(5,1,-1^2)   & {7}/{30}          & {1}/{4} & &
8 & 3 & \cQ(9,1,-1^2)   & {9383}/{37800}    & 21/80 \\
6 & 2 & \cQ(7,-1^3)     & {27}/{50}         & {217}/{360} & &
8 & 3 & \cQ(11,-1^3)    & {4506281}/{7144200} & 341/504 \\
8 & 0 & \cQ(3,1,-1^8)   & {3}/{8}           & 3/8 & &
10& 0 & \cQ(3,1^2,-1^9) & {3}/{16}          & 3/16 \\
8 & 0 & \cQ(5,-1^9)     & {5}/{8}           & 5/8 & &
10& 0 & \cQ(3^2,-1^{10})  & {9}/{32}          & 9/32 \\
8 & 1 & \cQ(3,1^2,-1^5) & {13}/{72}         & 31/168 & &
10& 0 & \cQ(5,1,-1^{10})  & {5}/{16}          & 5/16 \\
8 & 1 & \cQ(3^2,-1^6)   & {13}/{42}         & 9/28 & &
10& 0 & \cQ(7,-1^{11})    & {35}/{64}         & 35/64 \\
8 & 1 & \cQ(5,1,-1^6)   & {3}/{8}           & 65/168 & &
10& 1 & \cQ(3,1^3,-1^6) & {1159}/{12096}    & 391/4032 \\
8 & 1 & \cQ(7,-1^7)     & {45}/{56}         & {5}/{6} & &
10& 1 & \cQ(3^2,1,-1^7) & {47}/{288}         & 1/6 \\
8 & 2 & \cQ(3,1^3,-1^2) & {23}/{378}        & 157/2520 & &
10& 1 & \cQ(5,1^2,-1^7) & {113}/{576}       & 115/576 \\
8 & 2 & \cQ(3^2,1,-1^3) & {104}/{945}       & {97}/{840} & &
10& 1 & \cQ(5,3,-1^8)   & {139}/{432}       & 95/288 \\
8 & 2 & \cQ(5,1^2,-1^3) & {47}/{360}        & 17/126 & &
10& 1 & \cQ(7,1,-1^8)   & {5}/{12}          & 245/576 \\
8 & 2 & \cQ(5,3,-1^4)   & {17}/{72}         & 1/4 & &
10& 1 & \cQ(9,-1^9)     & {385}/{432}       & 175/192 \\
8 & 2 & \cQ(7,1,-1^4)   & {429}/{1400}      & 77/240 & &
10 & 3& \cQ(3, 1^5)     & {13}/{1134}       & {703}/{60480} \\
8 & 2 & \cQ(9,-1^5)     & {9383}/{12600}    & 9383/12600 & &
10& 3 & \cQ(7,3,1,-1^3) & 2027/20160        & 67/640 \\
8 & 3 & \cQ(3^2,1^2)    & {859}/{22680}     & 5/126 & &
&&&&\\
\hline
\end{array}\]
\caption{Completed volumes}
\end{table}

\newgeometry{left=0.1cm,bottom=0.1cm, top=0.1cm, right=0.1cm}
\begin{landscape}
\begin{table}
\renewcommand{\arraystretch}{1.5}
\[
\begin{array}{c|c|c|c}
N_{0,3}^{[3^2]}= 2 & 
N_{0,2}^{[3,1]}= 1 &&\\
N_{1,1}^{[3^2]}= \frac{b^2}{24}&&&\\
N_{0,4}^{[3^4]}= 6\sum_i b_i^2& 
N_{0,3}^{[3^3, 1]}=\frac{3}{2} b_i^2&
N_{0,2}^{[3^2, 1^2]}=\frac{b_i^2}{2}& 
N_{0,1}^{[3,1^3]}= \frac{1}{4}b_1^2\\
N_{1,2}^{[3^4]}=\frac{1}{16}(b_1^4+b_2^4)+\frac{1}{8}b_1^2b_2^2&
N_{1,1}^{[3^3, 1]}=\frac{b^4}{64}&&\\
N_{0,5}^{[3^6]}= \frac{45}{2}\sum_i b_i^4+90\sum_{i<j} b_i^2b_j^2&
N_{0,4}^{[3^5,1]}= \frac{15}{4}\sum_i b_i^4+15\sum_{i<j} b_i^2b_j^2& 
N_{0,3}^{[3^4,1^2]}= \frac{3}{4}\sum_i b_i^4+3\sum_{i<j} b_i^2b_j^2&
N_{0,2}^{[3^3,1^3]}= \frac{3}{16}\sum_i b_i^4+\frac{3}{4}\sum_{i<j} b_i^2b_j^2\\
N_{0,4}^{[5,3]}=9 &
N_{0,3}^{[5, 1]}= 3  && 
N_{0,1}^{[3^2, 1^4]}=\frac{1}{16}b^4\\
N_{1,2}^{[5,3]}=\frac{1}{4}(b_1^2+b_2^2)&
N_{1,1}^{[5, 1]}= \frac{1}{8}b_1^2&& \\
N_{0,5}^{[5,3^3]}=45\sum_i b_i^2 &
N_{0,4}^{[5,3^2, 1]}= 9\sum_i b_i^2& 
N_{0,3}^{[5,3, 1^2]}=\frac{9}{4}\sum_i b_i^2&
N_{0,2}^{[5, 1^3]}=\frac{3}{4}(b_1^2+b_2^2)\\
N_{1,3}^{[5,3^3]}=\frac{35}{64}\sum_i b_i^4+\frac{3}{2}\sum_{i<j}b_i^2b_j^2 & 
N_{1,2}^{[5,3^2, 1]}=\frac{23}{192}(b_1 ^4+b_2^4)+\frac{3}{8}b_1^2b_2^2 & 
N_{1,1}^{[5,3, 1^2]}=\frac{7}{192}b_1^4 \\
N_{2,1}^{[5,3^3]}=\frac{19}{7680}b_1^6 && \\
N_{0,6}^{[5,3^5]}=\frac{465}{2}\sum_i b_i^4+945\sum_{i<j}b_i^2b_j^2 & 
N_{0,5}^{[5,3^4, 1]}=33b_i^4+135 b_i^2b_j^2&
N_{0,4}^{[5,3^3,1^2]}=\frac{87}{16}b_i^4+\frac{45}{2}b_i^2b_j^2&
N_{0,2}^{[5,3,1^4]}=\frac{1}{4}(b_1^4+b_2^4)+\frac{9}{8}b_1^2b_2^2 \\
N_{0,5}^{[7,3]}= 60& 
N_{0,4}^{[7,1]}= 15&
N_{0,3}^{[5,3^2,1^3]}=\frac{17}{16}b_i^4+\frac{9}{2}b_i^2b_j^2&
N_{0,1}^{[5, 1^5]}=\frac{1}{16}b^4\\
N_{1,3}^{[7,3]}= \frac{15}{8}\sum_i b_i^2 &
N_{1,2}^{[7,1]}= \frac{5}{8}(b_1^2+b_2^2)&&\\
N_{2,1}^{[7,3]}= \frac{7}{384}b_1^4& &&\\
N_{0,6}^{[7,3^3]}= 450\sum_i b_i^2& 
N_{0,5}^{[7,3^2, 1]}=75\sum_i b_i^2 &
N_{0,4}^{[7,3, 1^2]}=15\sum_i b_i^2&
N_{0,3}^{[7,1^3]}=\frac{15}{4}\sum_i b_i^2\\
N_{1,4}^{[7,3^3]}=  \frac{195}{32}\sum_i b_i^4+\frac{75}{4}\sum_{i<j}b_i^2b_j^2&
N_{1,3}^{[7,3^2, 1]}=\frac{215}{192}\sum_i b_i^4+\frac{15}{4}\sum_{i<j} b_i^2b_j^2&
N_{1,2}^{[7,3, 1^2]}=\frac{25}{96}(b_1^4+b_2^4)+\frac{15}{16}b_1^2b_2^2&
N_{1,1}^{[7, 1^3]}=\frac{5}{64}b_1^4\\
N_{2,2}^{[7, 3^3]}=\frac{1}{32} b_i^6+ \frac{95}{512} b_i^2b_j^4& 
N_{2,1}^{[7,3^2, 1]}=\frac{29}{4608}b_i^6&&\\
N_{0,6}^{[9,3]}=525 & 
N_{0,5}^{[9,1]}=105 &&\\
N_{1,4}^{[9,3]}=\frac{35}{2}\sum_i b_i^2 & 
N_{1,3}^{[9, 1]}=\frac{35}{8}\sum_i b_i^2 &&\\
N_{2,2}^{[9,3]}=\frac{35}{192}\sum_i b_i^4+
\frac{35}{64}b_1^2b_2^2 &
N_{2,1}^{[9,1]}=\frac{7}{128}b_1^4&& \\
N_{0,7}^{[9, 3^3]}=\frac{11025}{2}b_i^2 & 
N_{0,6}^{[9, 3^2, 1]}=\frac{1575}{2}b_i^2 &
N_{0,5}^{[9, 3, 1^2]}=\frac{525}{4}b_i^2 &
N_{0,4}^{9, 1^3]}=\frac{105}{4}b_i^2  \\
N_{1,5}^{[9, 3^3]}=\frac{5145}{64}b_i^4+\frac{525}{2}b_i^2b_j^2 & 
N_{1,4}^{[9, 3^2,1]}=\frac{805}{64}b_i^4+\frac{175}{4}b_i^2b_j^2 &
N_{1,3}^{[9, 3, 1^2]}=\frac{455}{192}b_i^4+\frac{35}{4}b_i^2b_j^2 & 
N_{1,2}^{[9, 1^3]}=\frac{35}{64}b_i^4+\frac{35}{16}b_i^2b_j^2 \\
N_{2,3}^{[9, 3^3]}=\frac{231}{512}b_i^6+\frac{735}{256}b_i^4b_j^2+\frac{525}{64}b_i^2b_j^2b_k^2 &
N_{2,2}^{[9, 3^2, 1]}=\frac{119}{1536}b_i^6+\frac{805}{1536}b_i^4b_j^2 &
N_{2,1}^{[9,3,1^2]}=\frac{77}{4608}b^6&\\
N_{3,1}^{[9, 3^3]}=\frac{571}{442368}b_i^8&&&\\
\end{array}\]
\caption{Numbered Kontsevich polynomials}
\end{table}
\end{landscape}
\restoregeometry

\newpage
\begin{table}
\renewcommand{\arraystretch}{1.7}
\[
\begin{array}{c|c}
N_{2,2}^{[7,5]}=\frac{9}{64}b_i^4+\frac{29}{64}b_i^2b_j^2 & 
N_{2,1}^{[5,5]}=\frac{11}{640}b^4\\
N_{1,4}^{[7,5]}=\frac{105}{8}b_i^2 & 
N_{1,3}^{[5,5]}=\frac{3}{2}b_i^2\\
N_{0,6}^{[7,5]}=450 & 
N_{0,5}^{[5,5]}=54 \\
N_{0,7}^{[9,5]}=4725 & 
N_{1,5}^{[9,5]}=\frac{1155}{8}b_i^2\\
N_{2,3}^{[9,5]}=\frac{49}{32}b_i^4+\frac{77}{16}b_i^2b_j^2 & 
N_{3,1}^{[9,5]}=\frac{127}{15360}b^6\\
N_{0,7}^{[7^2]}=4500 & \\
N_{1,5}^{[7^2]}=\frac{525}{4}b_i^2 & \\
N_{2,3}^{[7^2]}=\frac{45}{32}b_i^4+\frac{145}{32}b_i^2b_j^2 & \\
N_{3,1}^{[7^2]}=\frac{53}{7168}b_i^6& \\
N_{0,7}^{[11,3]}=5670 & 
N_{0,6}^{[11, 1]}=945 \\
N_{1,5}^{[11,3]}=\frac{1575}{8}b_i^2 & 
N_{1,4}^{[11, 1]}=\frac{315}{8}b_i^2\\
N_{2,3}^{[11, 3]}=\frac{273}{128}b_i^4+\frac{105}{16}b_i^2b_j^2 & 
N_{2,2}^{[11, 1]}=\frac{63}{128}b_i^4+\frac{105}{64}b_i^2b_j^2 \\
N_{3,1}^{[11,3]}=\frac{11}{1024}b^6& \\
N_{0,8}^{[13,3]}=72765 & 
N_{0,7}^{[13,1]}=10395\\
N_{1,6}^{[13,3]}=\frac{10395}{4}b_i^2 &
 N_{1,5}^{[13,1]}=\frac{3465}{8}b_i^2\\
N_{2,4}^{[13,3]}=\frac{231}{8}b_i^4+\frac{5775}{64}b_i^2b_j^2 & N_{2,3}^{[13,1]}=\frac{693}{128}b_i^4+\frac{1156}{64}b_i^2b_j^2\\
N_{3,2}^{[13,3]}=\frac{77}{512}b_i^6+\frac{1001}{1024}b_i^4b_j^2&\\
\end{array}\]
\caption{Numbered Kontsevich polynomials}
\end{table}
\bibliographystyle{alpha}
\bibliography{refsvol}
\end{document}